\documentclass[11pt]{article}
\usepackage{latexsym, amsmath, amssymb, amsthm, a4, epsfig}
\usepackage{mathtools}
\usepackage{amsfonts}
\usepackage{graphicx}
\usepackage{color}
\usepackage{tikz}
\usepackage{tikz-cd}
\usepackage{ascmac}
\usepackage{comment}
\usepackage{url}
\usepackage{bbm}
\usepackage{here, bm, multicol, tcolorbox, subfig}
\usepackage{enumitem}

\usepackage{aliascnt} 

\usepackage{hyperref}

\usetikzlibrary{arrows,cd,calc}
\theoremstyle{definition}

\newtheorem{theorem}{Theorem}[section]

\newaliascnt{corollary}{theorem}     
\newtheorem{corollary}[corollary]{Corollary} 
\aliascntresetthe{corollary}         

\newaliascnt{lem}{theorem}
\newtheorem{lem}[lem]{Lemma}
\aliascntresetthe{lem}

\newaliascnt{proposition}{theorem}
\newtheorem{proposition}[proposition]{Proposition}
\aliascntresetthe{proposition}

\newaliascnt{definition}{theorem}
\newtheorem{definition}[definition]{Definition}
\aliascntresetthe{definition}

\newaliascnt{notation}{theorem}
\newtheorem{notation}[notation]{Notation}
\aliascntresetthe{notation}

\newaliascnt{note}{theorem}

\aliascntresetthe{note}

\newaliascnt{example}{theorem}
\newtheorem{example}[example]{Example}
\aliascntresetthe{example}

\newaliascnt{remark}{theorem}
\newtheorem{remark}[remark]{Remark}
\aliascntresetthe{remark}

\newtheorem*{claim*}{Claim}

\newaliascnt{prob}{theorem}

\aliascntresetthe{prob}

\renewcommand{\proofname}{\bf Proof}

\newenvironment{proofbreak}[1][\proofname]{%
  \begin{proof}[#1]\mbox{}\\\ignorespaces
}{%
  \end{proof}
}

\renewcommand{\refname}{References}

\renewcommand{\phi}{\varphi}

\newcommand{\Cbb}{\mathbb{C}}
\newcommand{\Dbb}{\mathbb{D}}
\newcommand{\Ebb}{\mathbb{E}}

\newcommand{\Kbb}{\mathbb{K}}

\newcommand{\Nbb}{\mathbb{N}}

\newcommand{\Rbb}{\mathbb{R}}

\newcommand{\Tbb}{\mathbb{T}}

\newcommand{\Zbb}{\mathbb{Z}}

\newcommand{\Gcal}{\mathcal{G}}

\newcommand{\Ical}{\mathcal{I}}

\newcommand{\Mcal}{\mathcal{M}}

\newcommand{\Ocal}{\mathcal{O}}
\newcommand{\Pcal}{\mathcal{P}}

\newcommand{\Rcal}{\mathcal{R}}

\newcommand{\Tcal}{\mathcal{T}}
\newcommand{\Ucal}{\mathcal{U}}

\newcommand{\Zcal}{\mathcal{Z}}

\def\Bl{{\bf l}}
\def\Bm{{\bf m}}
\def\Bn{{\bf n}}

\def\BB{{\bf B}}

\newcommand{\Ga}{\alpha}
\newcommand{\Gb}{\beta}
\newcommand{\Gd}{\delta}
\newcommand{\Ge}{\epsilon}

\newcommand{\Gg}{\gamma}
\newcommand{\Gc}{\chi}
\newcommand{\Gi}{\iota}
\newcommand{\Gk}{\kappa}

\newcommand{\Gl}{\lambda}
\newcommand{\Gn}{\eta}
\newcommand{\Gm}{\mu}
\newcommand{\Gv}{\nu}

\newcommand{\Gs}{\sigma}

\newcommand{\Go}{\omega}
\newcommand{\Gx}{\xi}

\newcommand{\GD}{\Delta}

\newcommand{\GG}{\Gamma}

\newcommand{\GO}{\Omega}

\newcommand{\Fs}[0]{\mathfrak{s}}

\newcommand{\beq}{\begin{equation}}
\newcommand{\eeq}{\end{equation}}

\renewcommand{\ker}{\operatorname{Ker}}
\newcommand{\im}{\operatorname{Im}}
\newcommand{\act}{\curvearrowright}
\newcommand{\End}[1]{\mathrm{End}_{G}(#1)}
\newcommand{\aut}[1]{\mathrm{Aut}_{G}(#1)}
\newcommand{\id}[1]{\mathrm{id}_{#1}}
\newcommand{\oinf}{\mathcal{O}_{\infty}}
\newcommand{\otw}{\mathcal{O}_{2}}
\newcommand{\bbm}[1]{\mathbbm{#1}}
\newcommand{\vin}{\rotatebox[origin=c]{90}{$\in$}}
\newcommand{\autx}[1]{\mathrm{Aut}_{C(X),G}(#1)}
\newcommand{\cn}{C_{\nu}}
\newcommand{\cna}{C_{\nu_{A}}}

\newcommand{\homg}[2]{\mathrm{Hom}_{G}(#1,#2)_{u}}
\newcommand{\ad}{\operatorname{Ad}}
\newcommand{\tmat}[4]{\left(\begin{smallmatrix}#1&#2\\#3&#4\\\end{smallmatrix}\right)}
\newcommand{\pmat}[4]{\begin{pmatrix}#1&#2\\#3&#4\\ \end{pmatrix}}
\newcommand{\ip}[2]{\langle #1, #2 \rangle}
\newcommand{\cstar}[0]{\ensuremath{\mathrm{C}^*}}
\renewcommand{\asymp}{\sim_{\mathrm{asymp}}}
\newcommand{\cc}{\cong_{\mathrm{cc}}}
\newcommand{\sasymp}{\sim_{\substack{\mathrm{strong}\\ \mathrm{asymp}}}}

\newcommand{\one}{\mathbf{1}}
\newcommand{\isa}{\ensuremath{\text{isometrically shift-absorbing }}}
\newcommand{\ssa}{\ensuremath{\text{strongly self-absorbing }}}
\newcommand{\Aut}[2]{\mathrm{Aut}_{#1}(#2)}
\newcommand{\dst}{D\otimes\mathbb{K}}
\newcommand{\dstf}{(D\otimes\mathbb{K})^{\gamma\otimes\mathrm{id}_{\mathbb{K}}}}
\newcommand{\onee}{\mathbf{1}_{D}\otimes e}
\newcommand{\oned}{\mathbf{1}_{D}}
\newcommand{\dstn}{(D\otimes\mathbb{K})^{\otimes n}}
\newcommand{\dstm}{(D\otimes\mathbb{K})^{\otimes m}}
\newcommand{\ggd}{\mathcal{G}^{G}_{D}}
\newcommand{\hotimes}{\hat{\otimes}}
\newcommand{\on}{\mathcal{O}_{n}}

\newcommand{\ptx}{(X,x_{0})}
\newcommand{\pty}{(Y,y_{0})}
\newcommand{\coker}{\operatorname{Coker}}
\renewcommand{\dim}{\operatorname{dim}}
\newcommand{\Hom}{\mathrm{Hom}}
\newcommand{\Ext}{\mathrm{Ext}^{1}}
\newcommand{\cu}{C_{u}}
\newcommand{\zzp}{\mathbb{Z}[\zeta_{p}]}

\DeclareMathOperator{\Ker}{Ker}

\numberwithin{equation}{section}
\numberwithin{figure}{section}

\allowdisplaybreaks

\begin{document}

\title{   The homotopy groups of the equivariant automorphism group of Kirchberg algebras with compact group actions and equivariant Dadarlat-Pennig theory     }

\author{Hiroro Kamikawa}
\date{}
\maketitle
\begin{abstract}
In the first half of this paper, we describe the homotopy groups of the equivariant automorphism group of Kirchberg algebras with isometrically shift-absorbing actions of compact groups in terms of equivariant KK-theory. This provides an equivariant version of Dadarlat's result. In the second half, we present a unified treatment of the equivariant Dadarlat-Pennig theory for strongly self-absorbing actions.

\end{abstract}
\tableofcontents
\section*{Introduction}
\addcontentsline{toc}{section}{Introduction}

The homotopy groups of automorphism groups play a crucial role in classifying locally trivial bundles of \cstar-algebras via homotopy theory. If we denote by $\mathrm{Aut}(A)$ the automorphism group of a \cstar-algebra $A$ equipped with the point-norm topology, there is a bijection between the set of isomorphism classes of locally trivial continuous fields over a compact Hausdorff space $X$ and the set $[X,B\mathrm{Aut}(A)]$ of homotopy classes of maps from $X$ to the classifying space $B\mathrm{Aut}(A)$ for principal $\mathrm{Aut}(A)$-bundles. For the unreduced suspension $SX$ of $X$, there is a bijection $[SX,B\mathrm{Aut}(A)]\cong [X,\mathrm{Aut}(A)]/\pi_{0}(\mathrm{Aut}(A))$. Consequently, the problem of computing these homotopy sets and homotopy groups naturally arises. 

The homotopy groups of the automorphism group have been studied for Kirchberg algebras \cite{Cu3,D1,MS}, simple AF algebras \cite{Th2,Nistor}, Cuntz-Toeplitz algebras \cite{IS,Sogabe}, and strongly self-absorbing \cstar-algebras \cite{DP}. In the equivariant setting, to the best of the author's knowledge, only the case of UHF-algebras equipped with product actions has only been studied in \cite{BP,EP}.

The notion of strongly self-absorbing \cstar-algebras was introduced by Toms and Winter \cite{TW} to establish a \cstar-algebraic analogue of the McDuff-type absorption result. A unital separable \cstar-algebra $D$ is said to be strongly self-absorbing if the first factor embedding $D\ni x\mapsto x\otimes \one_{D}\in D\otimes D$ is approximately unitarily equivalent to an isomorphism. This class contains \cstar-algebras that are cornerstones of the Kirchberg-Phillips classification of simple purely infinite \cstar-algebras: the Cuntz algebras $\Ocal_{2}$ and $\oinf$.

For a strongly self-absorbing \cstar-algebra $D$, Dadarlat and Pennig provided a theory of locally trivial bundles with fibers $D\otimes \Kbb$. By exploiting the properties of strongly self-absorbing \cstar-algebras, they obtained some important properties of the associated automorphism groups: $\mathrm{Aut}(D)$ is contractible in the point-norm topology, and $\mathrm{Aut}(D\otimes \Kbb)$ is well-pointed and has the homotopy type of a CW-complex.

Most notably, $\mathrm{Aut}(D\otimes \Kbb)$ admits the structure of an infinite loop space. This means there exists an $\GO$-spectrum—a sequence of spaces $(Y_{k})_{k\in \Zbb_{\geq 0}}$ such that $Y_{0}\simeq \mathrm{Aut}(D\otimes \Kbb)$ and each $Y_{k}$ is homotopy equivalent to the loop space $\GO Y_{k+1}$ of $Y_{k+1}$—which naturally gives rise to a generalized cohomology theory $E^{k}(X)=[X,Y_{k}]$. There exist two deloopings of $\mathrm{Aut}(D\otimes \Kbb)$. One is the space $Y_{1}$ in the $\GO$-spectrum, which is the classifying space $B_{\otimes}\mathrm{Aut}(D\otimes \Kbb)$ with respect to the multiplication induced by the tensor product. The other is the classifying space $B\mathrm{Aut}(D\otimes \Kbb)$ with respect to the group operation induced by the usual composition of automorphisms. In the classification of bundles, the set $[X,B\mathrm{Aut}(D\otimes \Kbb)]$ naturally arises. However, since it has been shown that $B_{\otimes}\mathrm{Aut}(D\otimes\Kbb)\simeq B\mathrm{Aut}(D\otimes \Kbb)$, this set can be regarded as the first group in a generalized cohomology theory, allowing one to use the Atiyah-Hirzebruch spectral sequence to compute $[X,\mathrm{Aut}(D\otimes\Kbb)]$.

Let $G$ be a locally compact group. Strongly self-absorbing actions were introduced by Szab\'o \cite{SZ1,SZ2} as a $G$-equivariant analogue of strongly self-absorbing $C^*$-algebras. In the classification result of Gabe and Szab\'o, which is an equivariant version of the Kirchberg-Phillips theorem, $\mathcal{O}_{\infty}$-absorption also plays a crucial role. An equivariant analogue of the Dadarlat-Pennig theory for such actions was developed in \cite{EP,BP} in the context of infinite tensor product actions of certain compact groups on UHF algebras.

The purpose of this article is twofold. The first is to generalize Dadarlat's computation of homotopy groups \cite{D1} to the case of isometrically shift-absorbing actions of compact groups on Kirchberg algebras. The proof is essentially a combination of the arguments of Dadarlat \cite{D1} and Gabe-Szab\'o \cite{GS2}. However, this work is indispensable for the equivariant generalization of reciprocal algebras introduced by Sogabe \cite{So2}. Two unital UCT Kirchberg algebras $A$ and $B$ with finitely generated $K$-groups are said to be reciprocal if $A \sim_{KK} D(C_{\nu_B})$ and $B \sim_{KK} D(C_{\nu_A})$, where $D(C)$ denotes the Spanier-Whitehead $K$-dual of $C$. Sogabe proved that for any two unital UCT Kirchberg algebras $A$ and $B$ with finitely generated $K$-groups, $\pi_n(\mathrm{Aut}(A)) \cong \pi_n(\mathrm{Aut}(B))$ for all $n \ge 1$ if and only if $A$ and $B$ are either isomorphic or reciprocal. An equivariant version of reciprocality has not yet been formulated. To establish such an equivariant formulation, results concerning homotopy groups are required.

Our second purpose is to provide a unified treatment of the equivariant Dadarlat-Pennig theory, including the results of Evans and Pennig \cite{EP}, by using the results of Szab\'o \cite{SZ1}. While our proof is a straightforward synthesis of the techniques found in Szab\'o \cite{SZ1}, Dadarlat-Pennig \cite{DP}, and Evans-Pennig \cite{EP}, it is expected that the framework organized in this paper will be useful for establishing a general theory of the equivariant Dadarlat-Pennig theory in the future.

In the first half of this paper, we generalize Dadarlat's results \cite{D1} on Kirchberg algebras to the setting of compact group actions.

Recall that a Kirchberg algebra is a separable, nuclear, simple, and purely infinite \cstar-algebra. These algebras were classified by Kirchberg and Phillips \cite{Ki,Ph} in terms of $KK$-theory. More precisely, given two stable Kirchberg algebras $A$ and $B$, any invertible element in $KK(A,B)$ lifts to an isomorphism between $A$ and $B$. Let us now recall Dadarlat's results based on the Kirchberg-Phillips theorem. Any continuous map $\phi:X\to\mathrm{End}(A\otimes \Kbb)$ can be viewed as an element of $\mathrm{Hom}(A\otimes\Kbb,C(X)\otimes A\otimes\Kbb)$. One can associate the $KK$-class $KK(\phi)\in KK(A,C(X)\otimes A)$ to such $\phi$. Utilizing this correspondence, Dadarlat computed the homotopy sets via the map $[X,\mathrm{Aut}(A\otimes \Kbb)]\to KK(A,C(X)\otimes A)$.

In the equivariant setting, Gabe and Szab\'o recently established an equivariant version of the Kirchberg-Phillips theorem for Kirchberg algebras equipped with amenable and isometrically shift-absorbing actions of locally compact groups \cite{GS2}. In order to deal with the case where the second variable is the tensor product of $C(X)$ and a Kirchberg algebra, we extend their result.

Let $(A,G,\Ga)$ be a $G$-\cstar-algebra. We denote by $\End{A}$ and $\aut{A}$ the sets of all equivariant $*$-endomorphisms and equivariant $*$-automorphisms of $A$, respectively (which we assume to be unital if $A$ is unital). The remaining notation is described in \autoref{subsubsec:unital}.
\begin{theorem}
Let $G$ be a compact group, let $A$ be a Kirchberg algebra and let $\Ga:G\act A$ be an isometrically shift-absorbing action. Let $X$ be a compact connected metrizable space. Then
\begin{enumerate}
\item If $A$ is unital, then
\[[X,\mathrm{End}_{G}(A)]\ni[\phi]\mapsto [(C_{\nu}\phi,C_{\nu}j_{A})]\in KK^{G}(\cna,SC(X)\otimes A)\]
is a bijection.
\item If $A$ is stable, then 
\[[X,\mathrm{End}_{G}(A)]\ni[\phi]\mapsto KK^{G}(\phi)\in KK^{G}(A,C(X)\otimes A)\]
is a bijection.
\end{enumerate}
\end{theorem}
To obtain results for equivariant automorphism groups, we introduce an equivariant version of $KK$-continuity in \cite{D1}.
\begin{definition}
Let $A$ be a separable C$^{\ast}$-algebra and let $X$ be a compact metrizable space with trivial $G$-action. We say that the pair $(A,X)$ is $KK^{G}$-continuous if for any point $x\in X$ there exists a base of closed neighborhoods $(V_{n})_{n}$ of $x$ such that $V_{n}\supset V_{m}$ for $n<m$ and the natural map $\varinjlim KK^{G}(A,C(V_{n})\otimes A)\to KK^{G}(A,A)$ is injective.
\end{definition}
If $X$ is locally contractible, then for any $G$-\cstar-algebra $A$, $(A,X)$ is $KK^{G}$-continuous pair. In particular, we can always work with finite CW-complexes. Let $\mathrm{Aut}_{G}(A)^{0}$ denote the path component of $\id{A}$.
\begin{theorem}
Let $G$ be a compact group and let $X$ be a path connected compact metrizable space and $x_{0}\in X$. Let $A$ be a Kirchberg algebra and let $\Ga:G\act A$ be an isometrically shift-absorbing action. Suppose that $(A,X)$ is a $KK^{G}$-continuous pair.
\begin{enumerate}
\item If $A$ is unital, 
\[\Gc_{X}:[X,\mathrm{Aut}_{G}(A)^{0}]\to KK^{G}(\cna,SC\ptx\otimes A)\]
is bijective.\\
\item If $A$ is stable,
\[\overline{\Gc}_{X}:[X,\mathrm{Aut}_{G}(A)^{0}]\to KK^{G}(A,C\ptx\otimes A)\]
is bijective.
\end{enumerate}
Moreover, if $\ptx$ is a co-$H$-space, they are group isomorphisms. In particular, we have 
\[\pi_{n}(\mathrm{Aut}_{G}(A))\cong\begin{cases}
KK^{G}(\cna, S^{n+1}A) & A\text{ is unital},\\
KK^{G}(A,S^{n}A) & A\text{ is stable}.
\end{cases}\]
for $n\geq 1$. For $n=0$, we have
\[\pi_{0}(\mathrm{Aut}_{G}(A))\cong \begin{cases}
(KK^{G}(\cna,SA)^{-1},\:\circ\:)& A\text{ is unital},\\
KK^{G}(A,A)^{-1}& A\text{ is stable}.
\end{cases}\]
Here, the group structure on $KK^{G}(A,A)^{-1}$ comes from the multiplication on $KK^{G}(A,A)$. The group structure on $KK^{G}(\cna,SA)$ is given by
\[x\circ y=x+y+y\otimes KK^{G}(\iota_{A})\otimes x\]
where $\Gi_{A}$ is the inclusion $SA\to \cna$.
\end{theorem}

Using the Universal Coefficient Theorem established in \cite{MN2,Ko}, we compute $\pi_{n}(\mathrm{Aut}_{G}(A))$ for $n\geq 1$, where $G=\mathbb{Z}/p\mathbb{Z}$ and $A$ is either $\mathcal{O}_{2}$ or $\mathcal{O}_{\infty}$ equipped with an outer $G$-action belonging to the equivariant bootstrap class.

In the second half of this paper, following the approach of \cite{EP}, we provide an equivariant version of the Dadarlat-Pennig theory for strongly self-absorbing actions with certain conditions of compact groups on unital separable \cstar-algebras.
\begin{theorem}
Let $G$ be a compact group and let $D$ be a unital separable \cstar-algebra. Let $\Gg:G\act D$ be a strongly self-absorbing action such that $D^{\Gg}$ is $K_{1}$-injective and $\dstf$ has the cancellation property. Then $\Aut{G}{\dst}$ is an infinite loop space. Let us denote the associated cohomology theory by $E^{*}_{D,G}$. Then we have
\[E^{0}_{D,G}(X)=[X,\Aut{G}{\dst}]\text{ and }E^{1}_{D,G}=[X,B\Aut{G}{\dst}].\]
We equip isomorphism classes of $G$-equivariant locally trivial fiber bundles with fibers $(\dst,\Gg\otimes\id{\Kbb})$ over a finite CW-complex $X$ with trivial action, with a group multiplication given by the fiberwise tensor product. Then this group is isomorphic to $E^{1}_{D,G}(X)$.\\
If $D$ is a unital Kirchberg algebra and $\Gg$ is an isometrically shift-absorbing and strongly self-absorbing action, then the coefficients are given by  
\[\check{E}^{k}_{D,G}\cong\begin{cases}
0 & k>0,\\
K_{0}^{G}(D)^{\times}& k=0,\\
K_{-k}^{G}(D)& k<0.
\end{cases}\]
\end{theorem}

This paper is organized as follows. In Section 1, we first review the basics of Kirchberg algebras, equivariant $K$-theory, and equivariant $KK$-theory. Next, we recall the property of being isometrically shift-absorbing, which we assume for the actions. Finally, we extend the result of Gabe and Szab\'o to the case where the second variable is the tensor product of a Kirchberg algebra and the algebra of continuous functions $C(X)$.

In Section 2, we deal with an equivariant version of Dadarlat's result. Namely, for an isometrically shift-absorbing action of a compact group $G$ on a Kirchberg algebra $A$, we provide $KK$-theoretic descriptions of the homotopy sets and homotopy groups of the set of equivariant $*$-homomorphisms $\mathrm{End}_G(A)$ and the equivariant automorphism group $\mathrm{Aut}_G(A)$. First, using the result extended at the end of the previous section, we describe $[X,\mathrm{End}_G(A)]$ in terms of equivariant $KK$-theory. We then introduce the concept of being $KK^G$-continuous for a pair consisting of a Kirchberg algebra $A$ and a topological space $X$, and based on this concept, we describe $[X,\mathrm{Aut}_G(A)]$. Furthermore, we review some facts concerning homotopy sets and give descriptions of the homotopy sets for pointed spaces and homotopy groups.

In Section 3, we extend the equivariant version of the Dadarlat-Pennig theory for strongly self-absorbing actions with certain conditions of a compact group $G$ on a unital separable \cstar-algebra $D$. First, we examine the topological properties of $\mathrm{Aut}_{G}(D)$ and $\mathrm{Aut}_{G}(D\otimes\Kbb)$. Next, to investigate the infinite loop space structure, we review the theory of commutative $\Ical$-monoids and show that $\mathrm{Aut}_{G}(D\otimes \Kbb)$ admits an infinite loop space structure. Finally, using this fact, we provide a classification result for locally trivial $(D\otimes \Kbb, \Gg, \otimes \id{\Kbb})$-bundles over finite CW-complexes.

This paper is based on the author's master's thesis.

\subsection*{Acknowledgments}
I would like to express my deepest gratitude to my supervisor, Professor Masaki Izumi, for his invaluable guidance and continuous support throughout my master's studies. I am also deeply grateful to him for carefully reading this thesis and providing many helpful corrections and suggestions. I would like to thank Professor Yosuke Kubota and Professor Taro Sogabe for their helpful comments and valuable advice on this research. I would also like to thank Professor Ralf Meyer for teaching me about the splitting of the exact sequence in the Universal Coefficient Theorem. Finally, I am grateful to my friend and colleague, Keiya Ohara, for our helpful daily discussions and encouragement.

\section{Preliminaries}
This section is organized as follows. In Section 2.1, we establish our notation for \cstar-algebras and recall some facts regarding pure infiniteness and proper infiniteness. Section 2.2 provides a brief review of basic definitions and properties of K-theory, equivariant K-theory, and equivariant KK-theory. In Section 2.3, we recall the notion of isometrically shift-absorbing actions, introduced by Gabe and Szab\'o \cite{GS2}, which play a central role in our classification results. Finally, in Section 2.4, we prove a Gabe-Szab\'o type theorem (cf. \cite{GS2}, Theorems 5.5--5.8), extending the result for $KK^G(A,B)$ to the case of $KK^G(A, B \otimes C(X))$.

\subsection{$C^{*}$-algebras}

\begin{notation}\label{not1}
For the general theory of \cstar-algebras, we refer the reader to \cite{Ped}. Throughout this paper, capital letters $A, B, C$ will denote C$^{*}$-algebras. The multiplier algebra of $A$ is denoted as $\Mcal(A)$. The Unitization of $A$ is denoted by $\tilde{A}$. (We add a new unit even if $A$ is unital.) For a $*$-homomorphism $\phi:A\to B$, $\tilde{\phi}$ denotes the unital $*$-homomorphism from $\tilde{A}$ to $\tilde{B}$ which is an extension of $\phi$. When $A$ is unital, we let $\Ucal(A)$ denote the set of all unitaries of $A$ and $\Ucal_{0}(A)$ denote the path-component of the unit.  We write $\Ucal(\one+A)$ for the set of all unitaries in $\tilde{A}$ whose scalar part is $\one$. For a separable infinite dimensional Hilbert space $H$, we write $\Kbb$ for a C$^{*}$-algebra of compact operators $\Kbb(H)$. A C$^{*}$-algebra $A$ is called stable if $A\cong A\otimes \Kbb$. We write $A\otimes B$ for the minimal tensor product. A simple C$^{*}$-algebra $A$ is called purely infinite if $A$ is not isomorphic to $\Cbb$ and for any two non-zero positive elements $a,b\in A$, there is $c\in A$ such that $a=cbc^{*}$. A simple purely infinite nuclear separable C$^{*}$-algebra is called a Kirchberg algebra.
\end{notation}
\begin{example}(The Cuntz algebra)\\
The Cuntz algebra $\Ocal_{n}$, where $2\leq n<\infty$, is the universal unital C$^{*}$-algebra generated by isometries $S_{1},\:S_{2},\:\cdots,\:S_{n}$ satisfying $\sum_{i=1}^{n}S_{i}S_{i}^{*}=1$. The Cuntz algebra $\oinf$ is the universal unital C$^{*}$-algebra generated by an infinite sequence of isometries $S_{1},\:S_{2},\cdots$ with mutually orthogonal range projections $S_{i}S_{i}^{*}$. It is known that the Cuntz algebras $\Ocal_{n}$ and $\oinf$ are Kirchberg algebras (\cite{Cu}).
\end{example}
We recall a few results on Kirchberg algebras that we will use later
\begin{theorem}[{\cite[Theorem 1.2]{Zh}}]
  Every separable, purely infinite, simple C$^{*}$-algebra is either unital or stable.
\end{theorem}
\begin{theorem}[{\cite[Theorem 2.1.6]{Ph}}]\label{absorption}
Let $A$ be a simple, separable, and nuclear C$^{*}$-algebra. Then $A$ is isomorphic to $A\otimes \oinf$ if and only if $A$ is purely infinite. Thus any Kirchberg algebra absorbs $\oinf$ tensorially.
\end{theorem}
We define proper infiniteness of elements in a C$^{*}$-algebra which plays an important role in $K$-theory defined in the next section.
\begin{definition}
A non-zero positive element $a$ in a C$^{*}$-algebra $A$ is called \textit{properly infinite} if there exists a sequence $(x_{k})_{k=1}^{\infty}$ in $M_{2}(A)$ with $x_{k}^{*}\mathrm{diag}(a,0)x_{k}\to\mathrm{diag}(a,a)$.
\end{definition}
Pure infiniteness of a C$^{*}$-algebra is useful for determining proper infiniteness of elements. We now recall the general definition of purely infinite C$^{*}$-algebras (which generalizes the condition given in \autoref{not1} for simple algebras).
\begin{definition}[{\cite[Definition 4.1]{KR}}]
A C$^{*}$-algebra is said to be \textit{purely infinite} if there are no characters on $A$ and if for every pair of positive elements $a,b$ in $A$, such that $a$ lies in the closed two-sided ideal generated by $b$, there exists a sequence $(r_{j})_{j=1}^{\infty}$ in $A$ with $r_{j}^{*}br_{j}\to a$.
\end{definition}
\begin{theorem}[{\cite[Theorem 4.16]{KR}}]
A C$^{*}$-algeba $A$ is purely infinite if and only if every non-zero positive element in $A$ is properly infinite.
\end{theorem}
It is known that $\oinf\otimes B$ is purely infinite for every C$^{*}$-algebra $B$ (\cite[Proposition 4.5]{KR}). Thus, combining this with \autoref{absorption}, we see that every non-zero positive element in $A\otimes B$ is properly infinite for every Kirchberg algebra $A$ and C$^{*}$-algebra $B$.
\begin{definition}
An element in a \cstar-algebra $B$ is full if it is not contained in any proper ideal of $B$. A $*$-homomorphism $\phi:A\to B$ is full if $\phi(a)$ is full in $B$ for any nonzero element $a\in A$
\end{definition}
\begin{lem}\label{lem:full-elem}
Let $A$ be a Kirchberg algebra and $X$ be a compact connected Hausdorff space. Then any nonzero projection in $C(X)\otimes A$ is a full element.
\end{lem}
\begin{proof}
Let $p$ be a nonzero projection in $C(X)\otimes A$. Since $X$ is compact and connected, $\|p_{x}\|=1$ for every $x\in X$. For any $\Ge>0$, there exists an open cover $(U_{n})_{n=1}^{N}$ such that $\sup_{x,y\in U_{n}}\|p_{x}-p_{y}\|<\Ge$. Let $(f_{n})_{n=1}^{N}$ be a partition of unity subordinate to $(U_{n})_{n=1}^{N}$. For each $n$, fix a point $x_{n}\in U_{n}$. Let $a\in A_{+}$ be any element with norm $1$. Since $A$ is simple and purely infinite, for each $n$, there exists $b_{n}\in A_{+}$ such that $b_{n}p_{x_{n}}b_{n}^{*}=a$ and $\|b_{n}\|\leq 2$. Define $h_{n}\in C(X)\otimes A$ to be $h_{n}(x)=f_{n}(x)p_{x}$. Then $h_{n}$ is in the ideal generated by $p$.
\begin{align*}
&\|1\otimes a-\sum_{n=1}^{N}(1\otimes b_{n})h_{n}(1\otimes b_{n}^{*})\|\\
&=\|\sum_{n=1}^{N}(1\otimes b_{n})(f_{n}\otimes p_{x_{n}})(1\otimes b_{n}^{*})-\sum_{n=1}^{N}(1\otimes b_{n})h_{n}(1\otimes b_{n}^{*})\|\\
&=\sup_{x\in X}\|\sum_{n=1}^{N}f_{n}(x)b_{n}(p_{x_{n}}-p_{x})b_{n}^{*}\|\\
&=\sup_{x\in X}\sum_{n=1}^{N}f_{n}(x)\|b_{n}\|\|p_{x_{n}}-p_{x}\|\|b_{n}^{*}\|\\
&<4\Ge
\end{align*}
Since $\Ge$ is arbitrary, $1\otimes a$ belongs to the ideal generated by $p$. Thus $p$ is a full element.
\end{proof}
\begin{lem}\label{lem:full-hom}
Let $A,B$ be Kirchberg algebras and $X$ be a compact connected Hausdorff space. Then any nonzero $*$-homomorphism $\phi:A\to B\otimes C(X)$ is full.
\end{lem}
\begin{proof}
Let $\phi:A\to B\otimes C(X)$ be a nonzero $*$-homomorphism. Since $A$ is simple, $\phi$ is injective. Let $a \in A$ be a nonzero element. Since $A$ is simple purely infinite, there exist $x,y\in A$ such that $p \coloneqq xay$ is a nonzero projection. By \autoref{lem:full-elem}, $\phi(p)$ is full in $B\otimes C(X)$. Since $\phi(p) = \phi(x)\phi(a)\phi(y)$, the ideal generated by $\phi(a)$ contains the ideal generated by $\phi(p)$. Thus, $\phi(a)$ is full in $B\otimes C(X)$, which implies that $\phi$ is full.
\end{proof}

\subsection{Equivariant $K$-theory and $KK$-theory}

In this section, we introduce equivariant $K$-theory and $KK$-theory. First, however, we recall the definition of ordinary $K$-theory for C*-algebras. For the basics of $K$-theory and its standard notation, we refer the reader to Blackadar's book \cite{Bl}.

Let $A$ be a C*-algebra.
Two projections $p, q$ in $A$ are unitarily equivalent, denoted by $p \sim_u q$, if there exists a unitary $u \in \mathcal{U}(\tilde{A})$ such that $q = upu^*$.
They are Murray-von Neumann equivalent, denoted by $p \sim q$, if there exists $v \in A$ such that $p = v^*v$ and $q = vv^*$.
Two unitaries $u, v \in \tilde{A}$ are homotopic, denoted by $u \sim_h v$, if there exists a continuous path in $\mathcal{U}(\tilde{A})$ from $u$ to $v$.
\begin{lem}
Let $A$ be a C$^{*}$-algebra and $p,q$ be projections in $A\otimes \Kbb$. Then $p\sim_{u}q$ if and only if $p\sim q$
\end{lem}
\begin{proof}
It is clear that unitary equivalence implies Murray-von Neumann equivalence. Suppose that $p\sim q$.
Note that $A\otimes \Kbb$ is regarded as an inductive limit of $M_{n}(A)\cong A\otimes M_{n}$ with connecting map defined by $
a\mapsto\mathrm{diag}(a,0)
$. So there exists a natural number $n$ and projections $p_{0},q_{0}$ in $M_{n}(A)$ such that $\|p-p_{0}\|<\frac{1}{4}$ and $\|q-q_{0}\|<\frac{1}{4}$. Then $p$ and $p_{0}$ (resp. $q$ and $q_{0}$) are homotopic in $\Pcal_{n}(A)$. Thus they are unitarily equivalent. Since $p\sim q$ and $p_{0}\sim q_{0}$, $\mathrm{diag}(p_{0},0_{n})$ and $\mathrm{diag}(q_{0},0_{n})$ are unitarily equivalent. Hence $p\sim_{u} q$.
\end{proof}
\begin{definition}\label{def:K0}
Let $A$ be a unital C*-algebra. For a positive integer $n$, let $\mathcal{P}_n(A)$ denote the set of all projections in $M_n(A)$. We set $\mathcal{P}_\infty(A) = \bigcup_{n=1}^\infty \mathcal{P}_n(A)$ (viewing the sets $\mathcal{P}_n(A)$ as being pairwise disjoint). We define an equivalence relation $\sim_0$ on $\mathcal{P}_\infty(A)$ as follows: For $p \in \mathcal{P}_n(A)$ and $q \in \mathcal{P}_m(A)$, we write $p \sim_0 q$ if there exists $v \in M_{m,n}(A)$ such that $p = v^*v$ and $q = vv^*$. The group $K_0(A)$ is defined to be the Grothendieck group of the commutative monoid $\mathcal{P}_\infty(A)/{\sim_0}$. We denote the equivalence class of a projection $p \in \mathcal{P}_\infty(A)$ in $K_0(A)$ by $[p]_0$. The group operation is determined by $[p]_0 + [q]_0 = [\mathrm{diag}(p, q)]_0$. If $A$ is non-unital, $K_0(A)$ is defined as the kernel of the map $K_0(\tilde{A}) \to K_0(\mathbb{C})$ induced by the quotient homomorphism $\tilde{A} \to \mathbb{C}$.
\end{definition}

\begin{definition}\label{def:K1}
Suppose first that $A$ is unital. Let $\mathcal{U}_n(A)$ denote the set of all unitaries in $M_n(A)$ and set $\mathcal{U}_\infty(A) = \bigcup_{n=1}^\infty \mathcal{U}_n(A)$. We define an equivalence relation $\sim_1$ on $\mathcal{U}_\infty(A)$ as follows: For $u \in \mathcal{U}_n(A)$ and $v \in \mathcal{U}_m(A)$, we write $u \sim_1 v$ if there exists a natural number $k \ge \max\{m, n\}$ such that $u \oplus 1_{k-n} \sim_h v \oplus 1_{k-m}$ in $\mathcal{U}_k(A)$. The group $K_1(A)$ is defined as the quotient group $\mathcal{U}_\infty(A)/{\sim_1}$, denoted by $\{[u]_1 \mid u \in \mathcal{U}_\infty(A)\}$. If $A$ is non-unital, we define $K_1(A)$ to be $K_1(\tilde{A})$. (Note that $K_1(\mathbb{C}) = \{0\}$, so this is consistent with the standard definition using the kernel).
\end{definition}

\begin{proposition}
Let $A,B$ be C$^{*}$-algebras.
\begin{enumerate}
\item (Functoriality) For any $*$-homomorphism $\phi:A\to B$, there exists an induced group homomorphism $K_{*}(\phi):K_{*}(A)\to K_{*}(B)$ given by $[x]_{*}\mapsto[\phi(x)]_{*}$ for $*=0,1$.
\item (Homotopy invariance) If $A$ and $B$ are homotopic then $K_{*}(A)\cong K_{*}(B)$ for $*=0,1$.
\item (Stability) $K_{*}(A)\cong K_{*}(A\otimes \Kbb)$ for $*=0,1$.
\item (Suspension isomorphism) Let $SA=C_{0}(0,1)\otimes A$ be the suspension of $A$. Then $K_{*}(SA)\cong K_{1-*}(A)$ for $*=0,1$.
\item (The six-term exact sequence) For every short exact sequence of C$^{*}$-algebras
\[0\to I\xrightarrow{\phi}A\xrightarrow{\psi}B\to 0,\]
the associated six-term sequence 
\begin{center}
\begin{tikzpicture}[auto]
\node(1) at (-3,1) {$K_{0}(I)$};
\node(2) at (0,1) {$K_{0}(A)$};
\node(3) at (3,1) {$K_{0}(B)$};
\node(4) at (3,-1) {$K_{1}(I)$};
\node(5) at (0,-1) {$K_{1}(A)$};
\node(6) at (-3,-1) {$K_{1}(B)$};
\draw[->] (1) to node {$K_{0}(\phi)$} (2);
\draw[->] (2) to node {$K_{0}(\psi)$} (3);
\draw[->] (3) to node {$\Gd$} (4);
\draw[->] (4) to node {$K_{1}(\phi)$} (5);
\draw[->] (5) to node {$K_{1}(\psi)$} (6);
\draw[->] (6) to node {$\Gd$} (1);
\end{tikzpicture}
\end{center}
is exact. 
\end{enumerate}
\end{proposition}
For later use, We now describe the index map $\delta : K_1(B) \to K_0(I)$. For any element $x \in K_1(B)$, take a unitary $u \in \mathcal{U}_n(\tilde{B})$ such that $x = [u]_1$.Since $\mathrm{diag}(u, u^*)$ is homotopic to $one_{2n}$ in $\mathcal{U}_{2n}(\tilde{B})$, there exists a unitary $v \in \mathcal{U}_{2n}(\tilde{A})$ such that$$\tilde{\psi}(v) = \begin{pmatrix} u & 0 \\ 0 & u^* \end{pmatrix}.$$Furthermore, let $p \in \mathcal{P}_{2n}(\tilde{I})$ be the projection satisfying$$\tilde{\varphi}(p) = v \begin{pmatrix} \one_n & 0 \\ 0 & 0 \end{pmatrix} v^* \quad \text{and} \quad s(p) = \begin{pmatrix} \one_n & 0 \\ 0 & 0 \end{pmatrix}.$$The index map is defined by $\delta(x) = [p]_0 - [s(p)]_0$. This construction is known to be well-defined.\\
Under suitable conditions, every element in the K-groups can be represented by a projection or a unitary in $A$ (without passing to matrix algebras). The following result for $K_{0}$-group is due to Cuntz \cite{Cu2}.
\begin{proposition}
Let $A$ be a C$^{*}$-algebra which contains a full properly infinite projection. For any $x\in K_{0}(A)$, there exists a full properly infinite projection $p\in A$ such that $x=[p]_{0}$. If $p,q\in A$ are two full properly infinite projections such that $[p]_{0}=[q]_{0}$, then $p\sim q$. Moreover, if we also assume that $A$ is unital and that both $\one-p$ and $\one-q$ are full and properly infinite, then $upu^{*}=q$ for some unitary $u\in \Ucal(A)$.
\end{proposition}
\begin{proposition}[{\cite[Lemma 2.1.7]{Ph}}]\label{prop:K1-oinf}
Let $A$ be any unital C$^{*}$-algebra. Then the canonical map $\Ucal(A\otimes\oinf)/\Ucal_{0}(A\otimes\oinf)\to K_{1}(A\otimes \oinf)$ is an isomorphism.
\end{proposition}

To introduce the equivariant $K$-theory and $KK$-theory, first we define group actions on C$^{*}$-algebras.
\begin{definition}
Let $G$ be a topological group and $A$ be a C$^{*}$-algebra. We equip $\mathrm{Aut}(A)$ with the point-norm topology. A group action of $G$ on $A$ is a continuous group homomorphism $\Ga:G\to\mathrm{Aut}(A)$. We may denote $\id{A}$ either for the identity map on $A$ or the trivial $G$-action on $A$. We call a C$^{*}$-algebra with a $G$-action a $G$-\textit{C$^{*}$-algebra.}
\end{definition}
We assume that $G$ is compact and second countable in what follows.
\begin{definition}[Equivariant $K_{0}$-group]
Let $G$ be a compact group and $A$ be a unital C$^{*}$-algebra and $\Ga$ be an action of $G$ on $A$. Let $\Pcal_{G}(A)$ be the set of all $G$-invariant projections in $\BB(V)\otimes A$, where $V$ is a finite dimensional Hilbert space with a unitary representation of $G$. We define an equivalence relation $\sim_{0}$ on $\Pcal_{G}(A)$ as follows: For $p\in \BB(V)\otimes A$ and $q\in \BB(W)\otimes A$, we write $p\sim_{0}q$ if there are $G$-invariant elements $u\in \BB(V,W)\otimes A$ and $v\in \BB(W,V)\otimes A$ such that $vu=p$ and $uv=q$. The group $K_{0}^{G}(A)$ is defined to be the Grothendieck group of the commutative monoid $\Pcal_{G}(A)/\sim_{0}$. We denote the equivalence class of a projection $p\in\Pcal_{G}(A)$ in $K_{0}^{G}(A)$ by $[p]_{0}$. The group operation is determined by $[p]_{0}+[q]_{0}=[\mathrm{diag}(p,q)]$. If $A$ is nonunital, $K_{0}^{G}(A)$ is defined as the kernel of the map $K_{0}^{G}(\tilde{A})\to K_{0}^{G}(\Cbb)$ induced by the quotient homomorphism $(\tilde{A},\Ga)\to(\Cbb,\id{} )$.
\end{definition}
Before defining the equivariant $K_{1}$-group, we see that the equivariant $K_{0}$-group has an additional structure.
\begin{definition}\label{def:rep-ring}
  The representation ring $R(G)$ of $G$ is the ring whose elements are formal differences of equivalence classes of finite-dimensional representations of $G$, with direct sum and tensor product as the ring operations. The trivial one-dimensional representation is the multiplicative identity.
\end{definition}
Note that $R(G)$ is identified with $K_{0}^{G}(\Cbb)$.\\
For every C$^{*}$-algebra $A$, $K_{0}^{G}(A)$ is not only a group but also a $R(G)$-module. The module structure is given as follows:\\
If $p\in \BB(V)\otimes A$ is a $G$-invariant projection, and $W$ is a finite dimensional representation space of $G$, then $[W]\cdot[p]_{0}$ is represented by $\one\otimes p\in \BB(W)\otimes\BB(V)\otimes A$.\\
$K_{1}^{G}(A)$ is defined to be $K_{0}^{G}(SA)$, where $SA$ is the suspension of $A$ with $G$-action $\id{}\otimes \Ga$. Alternatively, we use the picture in terms of invariant unitaries given by Phillips.\\
Let $G$ be a compact group and $A$ be a $G$-C$^{*}$-algebra. First, we assume that $A$ is unital.\\
Let $R_{+}(G)$ be the set of unitary equivalence classes of finite dimensional representation of $G$. Then it is a directed set for the ordering $V\leq W$ if $V$ is equivalent to a direct summand of $W$.
For $V\in R_{+}(G)$, we let $U^{G}(V,A)$ be the group of $G$-invariant unitary elements in $\BB(V)\otimes A$, and we let $U_{0}^{G}(V,A)$ be the connected component of the identity in $U^{G}(V,A)$. Let $\overline{U}^{G}(V,A)=U^{G}(V,A)/U_{0}^{G}(V,A)$. Let $V,W\in R_{+}(G)$ with $V\leq W$. Then $W\cong V\oplus V_{0}$ for some $V_{0}\in R_{+}(G)$. We define an embedding $U^{G}(V,A)\to U^{G}(W,A)$ by $u\mapsto u\oplus \one_{\BB(V_{0})\otimes A}$. This map induces a group homomorphism $i_{V,W}:\overline{U}^{G}(V,A)\to \overline{U}^{G}(W,A)$. This $i_{V,W}$ is well-defined and if $V,W,X\in R_{+}(G)$ with $V\leq W\leq X$, then $i_{W,X}\circ i_{V,W}=i_{V,X}$.
\begin{theorem}[{\cite[Theorem 2.8.8]{Ph2}}]\label{thm:eq-K1}
Let $A$ be a $G$-C$^{*}$-algebra. Then there is a natural isomorphism $\displaystyle K_{1}^{G}(A)\cong \varinjlim_{V\in R_{+}(G)}\overline{U}^{G}(V,A)$.
\end{theorem}
\begin{theorem}[{\cite[Theorem 2.8.3]{Ph2}}]
Let $G$ be a compact group. 
\begin{enumerate}
\item (Functoriality) For any equivariant $*$-homomorphism $\phi:A\to B$, there exists an induced $R(G)$-module map $K_{*}(\phi):K_{*}^{G}(A)\to K_{*}^{G}(B)$ given by $[x]_{*}\mapsto [\phi\otimes\id{}(x)]_{*}$ for $*=0,1$.
\item (Homotopy invariance) If $\phi,\psi:A\to B$ are homotopic equivariant $*$-homomorphisms, then $K_{*}^{G}(\phi)$ and $K_{*}^{G}(\psi)$ are equal.
\item (Stability) $K_{*}^{G}(\Kbb(H)\otimes A)\cong K_{*}^{G}(A)$ if the action of $G$ on $\Kbb(H)$ is given by a unitary representation of $G$ on the Hilbert space $H$.
\item (Six-term exact sequence) For every short exact sequence of $G$-C$^{*}$-algebras
\[0\rightarrow I\xrightarrow{\phi}A\xrightarrow{\psi} B\rightarrow 0,\]
the associated six-term sequence
\begin{center}
\begin{tikzpicture}[auto]
\node(1) at (-3,1) {$  K_{0}^{G} (I)  $};
\node(2) at (0,1) {$    K_{0}^{G} (A) $};
\node(3) at (3,1) {$  K_{0}^{G} (B)   $};
\node(4) at (3,-1) {$   K_{1}^{G}(I)   $};
\node(5) at (0,-1) {$  K_{1}^{G} (A)   $};
\node(6) at (-3,-1) {$    K_{1}^{G}(B)  $};
\draw[->] (1) to node {$  K_{0}^{G} (\phi)   $} (2);
\draw[->] (2) to node {$  K_{0}^{G} (\psi)   $} (3);
\draw[->] (3) to node {$ \Gd    $} (4);
\draw[->] (4) to node {$   K_{1}^{G}(\phi)   $} (5);
\draw[->] (5) to node {$   K_{1}^{G} (\psi)  $} (6);
\draw[->] (6) to node {$ \Gd    $} (1);
\end{tikzpicture}
\end{center}
\end{enumerate}
\end{theorem}
The index map is described in more or less the same way as in the non-equivariant case.
For any element $x\in K^{G}_{1}(A/I)$, there exists a $G$-invariant unitary $u\in (A/I)\otimes \BB(V)$ with $x=[u]_{1}$ for some finite-dimensional representation space $V$ of $G$. The unitary $\mathrm{diag}(u,u^{*})$ is homotopic to $\mathbf{1}$ in the unitary group of the fixed point algebra. Since $G$ is compact, there exists a $G$-invariant unitary $v\in M_{2}(A\otimes \BB(V))$ which is a lift of $\mathrm{diag}(u,u^{*})$. Then $\Gd(x)$ is given by $[v\mathrm{diag}(\mathbf{1}_{A\otimes \BB(V)},0)v^{*}]_{0}-[\mathrm{diag}(\mathbf{1}_{A\otimes \BB(V)},0)]_{0}$.
\vskip\baselineskip
We now describe the concept of saturation of an action introduced by Rieffel, which links equivariant $K$-theory to ordinary $K$-theory.
\begin{definition}
Let $G$ be a compact group with a Haar measure $\Gm$ and $A$ be a \cstar-algebra with an action $\Ga$ of $G$. We write $C_{c}(G,A)$ for the set of all continuous functions from $G$ to $A$ with compact support. Define a multiplication and an involution on $C_{c}(G,A)$ by
\[(fg)(s)=\int_{G}f(t)\Ga_{t}(g(t^{-1}s))d\Gm(t),\quad f^{*}(s)=\Ga_{s}(f(s^{-1})^{*})\]
for $f,g\in C_{c}(G,A)$. Then $L^{1}(G,A)$ is the completion of $C_{c}(G,A)$ with respect to the norm $\|f\|\coloneqq \int_{G}\|f(s)\|d\Gm(s)$. Let $\pi:A\to \BB(H)$ be a faithful representation and $\Gl:G\to \Ucal(L^{2}(G))$ be the left regular representation. These representations induce a representation $\tilde{\pi}\times \Gl$ of $L^{1}(G,A)$ on $L^{2}(G,H)$ defined by 
\[((\tilde{\pi}\times \Gl)(f)\Gx)(s)=\int_{G}\pi(\Ga_{s^{-1}}(f(t)))\Gx(t^{-1}s)d\Gm(t)\]
for $f\in C_{c}(G,A)$ and $\Gx\in L^{2}(G,H)$. The norm closure of the image of $L^{1}(G,A)$ under $\tilde{\pi}\times \Gl$ is a \cstar-algebra called the \textit{reduced crossed product} and denoted by $A\rtimes_{\Ga,r}G$.
\end{definition}
\begin{definition}\label{def:hilb-mod}
Let $B$ be a C$^{*}$-algebra. A pre-Hilbert module over $B$ is a right $B$-module $E$ equipped with a $B$-valued inner product, a function $\ip{\cdot}{\cdot}:E\times E\to B$, with the following properties:
\begin{align*}
&(1)\:\ip{\cdot}{\cdot}\text{ is sesquilinear. (conjugate-linear in the first variable)}\\
&(2)\:\ip{x}{yb}=\ip{x}{y}b\text{ for all }x,y\in E,b\in B\\
&(3)\:\ip{y}{x}=\ip{x}{y}^{*}\text{ for all }x,y\in E\\
&(4)\:\ip{x}{x}\geq 0\text{ holds for all }x\in E,\text{ and if }\ip{x}{x}=0,\text{ then }x=0.
\end{align*}
For $x\in E$, put $\|x\|=\|\ip{x}{x}\|^{\frac{1}{2}}$. Then this is a norm on $E$. If $E$ is complete with respect to this norm, $E$ is called a Hilbert module over $B$.
\end{definition}
\begin{definition}
Let $G$ be a compact group, and let $A$ be a C$^{*}$-algebra with an action $\Ga$ of $G$. We equip $A$ with the structure of a pre-Hilbert $A\rtimes_{\Ga,r}G$-module and define a left action of $A^{G}$ and an $A^{G}$-valued inner product on $A$ as follows.\\
For $a\in A^{G}$, $x,y\in A$ and $f\in L^{1}(G,A)$, we make the following definitions:
\begin{align*}
&ax \text{ is the usual product of $a$ and $x$ as elements of $A$,}\\
&xf=\int_{G}\Ga_{g}^{-1}(xf(g))dg,\\
&\langle x,y  \rangle_{A^{G}}=\int_{G}\Ga_{g}(xy^{*})dg,\\
&\langle  x,y \rangle_{A\rtimes_{\Ga,r}G}\in C_{c}(G,A)\subset A\rtimes_{\Ga,r}G\text{ is defined by }\ip{x}{y}_{A\rtimes_{\Ga,r}G}(g)=x^{*}\Ga_{g}(y)\text{ for all }g\in G.
\end{align*}
For $x\in A$, put $\|x\|=\|\ip{x}{x}_{A\rtimes_{\Ga,r}G}\|_{A\rtimes_{\Ga,r}G}^{\frac{1}{2}}$. Let $E$ be the completion of $A$ with respect to this norm. Note that the left action of $A^{G}$ and the $A^{G}$-valued inner product extend to $E$. The action $\Ga$ is said to be \textit{saturated} if $E$ is an equivalence bimodule in the sense of Rieffel \cite{Ri}. In other words, $\mathrm{span}\{\ip{x}{y}_{A^{G}}\:|\:x,y\in E\}$ and $\mathrm{span}\{\ip{x}{y}_{A\rtimes_{\Ga,r}G}\:|\:x,y\in E\}$ are dense in $A^{G}$ and $A\rtimes_{\Ga,r}G$, respectively. Since the former condition is already satisfied in our setting, this amounts to the latter.
\end{definition}
\begin{proposition}[{\cite[Proposition 7.1.8]{Ph2}}]
Let $G$ be a second countable compact group and $A$ be a separable C$^{*}$-algebra with $G$-action $\Ga$. If $\Ga$ is saturated, then the natural map $K_{*}(A^{G})\to K_{*}^{G}(A)$ is an isomorphism.
\end{proposition}
Finally, we recall the definition of equivariant $KK$-theory due to Kasparov \cite{Ka}. We begin with the original definition, known as the module picture, and then introduce the Cuntz-Thomsen picture. 
\begin{definition}
A grading on a C$^{*}$-algebra $B$ is a decomposition of $B$ into a direct sum of two self-adjoint closed subspaces $B^{(0)}$ and $B^{(1)}$ such that $B^{(m)}B^{(n)} \subset B^{(m+n)}$ (addition mod $2$). In this paper, we restrict our attention to trivially graded \cstar-algebras, i.e., $B^{(1)}=0$.

A Hilbert $B$-module $E$ is said to be \textit{graded} if there exists a decomposition into a direct sum of closed subspaces $E^{(0)}$ and $E^{(1)}$ satisfying $E^{(m)}B^{(n)}\subset E^{(m+n)}$ and $\ip{E^{(m)}}{E^{(n)}} \subset B^{(m+n)}$. We denote by $\BB(E)$ the set of all adjointable maps on $E$. That is, a module homomorphism $T:E\to E$ belongs to $\BB(E)$ if there exists a map $T^{*}:E\to E$ such that $\ip{Tx}{y}=\ip{x}{T^{*}y}$ for all $x,y\in E$.
For $x,y\in E$, let $\theta_{x,y}$ be the operator on $E$ defined by $\theta_{x,y}(z)=x\ip{y}{z}$. Then $\theta_{x,y}$ is adjointable with adjoint $\theta_{y,x}$. Let $\Kbb(E)$ be the norm closure of $\mathrm{span}\{\theta_{x,y}\:|\: x,y\in E\}$. It is well known that $\Kbb(E)$ is a closed ideal in $\BB(E)$. A grading on $E$ induces a grading on $\BB(E)$ (and hence on $\Kbb(E)$) as follows. An operator $T \in \BB(E)$ is called \textit{even} (resp. \textit{odd}) if $T E^{(i)} \subset E^{(i)}$ (resp. $T E^{(i)} \subset E^{(1-i)}$) for $i=0,1$. We denote the subspace of even operators by $\BB(E)^{(0)}$ and that of odd operators by $\BB(E)^{(1)}$.

Let $\Gb$ be an action of a compact group $G$ on $B$.
A continuous action of $G$ on $E$ is a group homomorphism from $G$ into the group of invertible bounded linear transformations on $E$ such that the map $g\mapsto g\cdot x$ is continuous for all $x\in E$, and each transformation preserves the grading and satisfies $g\cdot (xb)=(g\cdot x)\Gb_{g}(b)$ and $\ip{g\cdot x}{g\cdot y}=\Gb_{g}(\ip{x}{y})$ for all $g\in G$, $x,y\in E$, and $b\in B$. A Hilbert $B$-module equipped with such a continuous action of $G$ is called a Hilbert $(B,\Gb)$-module. A $G$-action on $E$ induces a $G$-action on $\BB(E)$ as follows. For $T\in \BB(E)$ and $x\in E$, $(g\cdot T)(x)=g\cdot(T(g^{-1}\cdot x))$.
\end{definition}
\begin{definition}
Let $E_{1},E_{2}$ be Hilbert $(B,G,\Gb)$-modules. $T\in \BB(E_{1},E_{2})$ is $G$-continuous if $g\mapsto g\cdot T$ is norm continuous.
\end{definition}
\begin{definition}
Let $A$ and $B$ be $G$-C$^{*}$-algebras. The set of Kasparov $A$-$B$-bimodules $\Ebb_{G}(A,B)$ for $(A,B)$, is the set of triples $(E,\phi, F)$, where $E$ is a countably generated Hilbert $B$-module with a continuous action of $G$, $\phi:A\to\BB(E)$ is an equivariant graded $*$-homomorphism, and $F$ is a $G$-continuous operator in $\BB(E)$ such that $FE^{(m)}\subset E^{(1-m)}$ and $[F,\phi(a)],(F^{2}-1)\phi(a),(F-F^{*})\phi(a)$ and $(g\cdot F-F)\phi(a)$ are all in $\Kbb(E)$ for all $a\in A$ and $g\in G$. We say that a kasparove $A$-$B$ bimodule ($E,\phi,F)$ is degenerate if  $[F,\phi(a)],(F^{2}-1)\phi(a),(F-F^{*})\phi(a)$ and $(g\cdot F-F)\phi(a)$ are $0$ for all $a\in A$ and $g\in G$. Let $\Dbb_{G}(A,B)$ be the set of degenerate Kasparov $A$-$B$-bimodules.
\end{definition}
\begin{definition}
Two Kasparov $A$-$B$-bimodules $(E_{0},\phi_{0},F_{0})$ and $(E_{1},\phi_{1},F_{1})$ are unitarily equivalent if there is an even unitary $U$ in $\BB(E_{0},E_{1})$ that intertwines the representation of $A$ and the groups action and satisfies $F_{0}=U^{*}F_{1}U$.
\end{definition}
If $(B,\Gb$) is a $G$-C$^{*}$-algebra, then we denote by $B[0,1]$ the algebra $C([0,1],B)$ with point-wise $G$-action $\Gb[0,1]\coloneqq \Gb\otimes\id{C[0,1]}$. Let $(E,\phi, F)$ be a Kasparov $A-B[0,1]$ modules. For each $t\in [0,1]$, let $\mathrm{ev}_{t}:B[0,1]\to B$ be evaluation at $t$. Then we obtain a Kasparov $A-B$-bimodule $(E_{t},\phi_{t},F_{t})$ by setting $E_{t}=E\otimes_{\mathrm{ev}_{t}}B$, $\phi_{t}=\phi\otimes 1$ and $F_{t}=F\otimes 1$. 
\begin{definition}
Two Kasparov $A$-$B$-bimodules $(E_{0},\phi_{0},F_{0})$ and $(E_{1},\phi_{1},F_{1})$ are homotopic if there exists a Kasparov $A$-$B[0,1]$-bimodule $(E,\phi,F)$ such that $(E_{0},\phi_{0},F_{0})$ is unitarily equivalent to the module obtained by the evaluation at $0$ and $(E_{1},\phi_{1},F_{1})$ is unitarily equivalent to the module obtaines by the evaluation at $1$. If it is the case, we write $(E_{0},\phi_{0},F_{0})\sim_{h}(E_{1},\phi_{1},F_{1})$. Note that every element in $\Dbb_{G}(A,B)$ is homotopic to $(0,0,0)$.
\end{definition}
\begin{definition}
For $G$-C$^{*}$-algebras $A$ and $B$, $KK^{G}(A,B)$ is the quotient of $\Ebb_{G}(A,B)$ by $\sim_{h}$.
\end{definition}
\begin{remark}
Let $A,B$ be $G$-\cstar-algebras and let $\phi:A\to B$ be an equivariant $*$-homomorphism. Then $(B,\phi,0)$ is a Kasparov triple and We denote the equivalence class of $(B,\phi,0)$ in $KK^{G}(A,B)$ by $KK^{G}(\phi)$.
\end{remark}
\begin{proposition}[{\cite[Section 2]{Ka}}]
$KK^{G}(A,B)$ is an abelian group, where addition is defined by direct sum of Kasparov bimodules. $KK^{G}$ is a bifunctor  from pairs of $G$-C$^{*}$-algebras to abelian groups, contravariant in the first variable and covariant in the second. 
\end{proposition}
\begin{proposition}[{\cite[Proposition 2.11]{Ka}}]
Suppose that $(A,\Ga),(B,\Gb),(C,\Gg)$ are separable $G$-C$^{*}$-algebras. Then there is a bilinear pairing, called the Kasparov product,
\[\otimes :KK^{G}(A,B)\times KK^{G}(B,C)\to KK^{G}(A,C).\]
Moreover, the Kasparov product is associative. The elements $1_{A}\coloneqq[A,\id{A},0]\in KK^{G}(A,A)$ and $1_{B}\coloneqq[B,\id{B},0]\in KK^{G}(B,B)$ act as identities from the left and right on $KK^{G}(A,B)$. For any $G$-\cstar-algebras $A,B,C$ and an equivariant $*$-homomorphism $\phi:A\to B$, we write $\phi_{*}$ for the induced map $KK^{G}(C,A) \to KK^{G}(C,B)$ given by $x\mapsto x\otimes KK^{G}(\phi)$, and $\phi^{*}$ for the induced map $KK^{G}(B,C) \to KK^{G}(A,C)$ given by $y\mapsto KK^{G}(\phi)\otimes y$.
\end{proposition}
\begin{definition}
Let $(A,\Ga)$ and $(B,\Gb)$ be $G$-\cstar-algebras and let $\phi:A\to B$ be an equivariant $*$-homomorphism. Its mapping cone $C_{\phi}$ is defined by 
\[C_{\phi}=\{(a,f)\in A\oplus C_{0}([0,1),B)\:|\:f(0)=\phi(a)\}.\]
The sequence 
\[SB\xrightarrow{i}C_{\phi}\xrightarrow{\pi}A\xrightarrow{\phi}B\] is called the \textit{mapping cone sequence}. Here, $i$ is the inclusion map $f\mapsto (0,f)$ and $\pi$ is projection onto $A$.
\end{definition}
\begin{proposition}[{\cite[Section 2.2]{MN}}]
For any mapping cone sequence 
\[SB\xrightarrow{i}C_{\phi}\xrightarrow{\pi}A\xrightarrow{\phi}B\]
and any $G$-C$^{*}$-algebra $(C,\Gg)$, we have the following exact sequence:
\[KK^{G}(B,C)\xrightarrow{\phi^{*}}KK^{G}(A,C)\xrightarrow{\pi^{*}}KK^{G}(C_{\phi},C)\xrightarrow{i^{*}}KK^{G}(SB,C)\xrightarrow{(S\phi)^{*}}KK^{G}(SA,C).\]
Here, $S\phi$ denotes the suspension of $\phi$.
\end{proposition}
Let $B$ be a unital $G$-C$^{*}$-algebra and $p\in\BB(V)\otimes B$ be a $G$-invariant projection for some finite representation space $V$ of $G$. Let $\phi_{p}:\Cbb\to \BB(V\otimes_{\Cbb} B)\cong \Kbb(B^{\dim V})$ be an equivariant $*$-homomorphism defined by $\Gl\mapsto\Gl p$. Then $(B^{\dim V},\phi_{p},0)$ is a Kasparov $\Cbb$-$B$-bimodule, and hence $[B^{\dim V},\phi_{p},0]\in KK^{G}(\Cbb,B)$. This construction induces a homomorphism from $K^{G}_{0}(B)$ into $KK^{G}(\Cbb,B)$. In fact, it is known that this gives an isomorphism of abelian groups. If $B$ is non-unital, we obtain the same isomorphism by using the split exact sequeznce for the unitization of $B$ and the six-term exact sequence for $KK^{G}$.

We now introduce the Cuntz-Thomsen picture of $KK^{G}$. The strict topology on $\Mcal(A)$ is the topology given by the seminorms $x\mapsto \|ax\|$ and $x\mapsto\|xa\|$ for $a\in A$, i.e. $x_{i}\to x$ strictly if and only if $ax_{i}\to ax$ and $x_{i}a\to xa$ in norm for all $a\in A$.
\begin{definition}
Let $(A,\Ga)$ be a $G$-C$^{*}$-algebra. A $1$-cocycel with respect to $\Ga$ is a strictly continuous map $\bbm{u}:G\to \Ucal(\Mcal(A))$ that satisfies the cocycle identity $\bbm{u}_{gh}=\bbm{u}_{g}\Ga_{g}(\bbm{u}_{h})$ for all $g,h\in G$. 
\end{definition}
\begin{definition}
Let $\Ga:G\act A$ and $\Gb:G\act B$ be two actions on C$^{*}$-algebras.
\begin{enumerate}
\item A cocycle representation $(\phi,\bbm{u}):(A,\Ga)\to (\Mcal(B),\Gb)$ consists of a $*$-homomorphism $\phi:A\to \Mcal(B)$ and a $\Gb$-cocycle $\bbm{u}:G\to\Ucal(\Mcal(B))$ satisfying $\ad{\bbm{u}_{g}}\circ \Gb_{g}\circ\phi=\phi\circ\Ga_{g}$ for all $g\in G$.
\item If additionally $\phi(A)\subset B$, then $(\phi,\bbm{u})$ is called a \textit{cocycle morphism}, and we denote $(\phi,\bbm{u}):(A,\Ga)\to(B,\Gb)$.
\item If the range of $\phi$ is contained in $B$ and $\bbm{u}$ takes values in $\Ucal(1+B)$, then $(\phi,\bbm{u})$ is called a \textit{proper cocycle morphism}.
\item A cocycle morphism $(\phi,\bbm{u})$ is called a \textit{cocycle conjugacy} if $\phi$ is an isomorphism. Furthermore, wa say that $(A,\Ga)$ and $(B,\Gb)$ are cocycle conjugate, denoted by $(A,\Ga)\cc(B,\Gb)$, if there exists a \textit{cocycle conjugacy} $(\phi,\bm{u}):(A,\Ga)\to(B,\Gb)$. Moreover, if we can take the cocycle to be $\mathbf{1}$, we say that they are \textit{conjugate} and denote it by $(A,\Ga)\cong(B,\Gb)$.
\end{enumerate}
\end{definition}
\begin{definition}
Let $(B,\Gb)$ be a $G$-C$^{*}$-algebra. $\Gb$ is \textit{strongly stable} if $(B,\Gb)$ is conjugate to $(B\otimes \Kbb,\Gb\otimes \id{\Kbb})$.
\end{definition}
If $\Gb$ is strongly stable, then there exists an embedding $(\Kbb,\id{\Kbb})\to (\Mcal(B),\Gb)$. Thus there is a sequence of isometries $r_{n}\in \Mcal(B)^{\Gb}$ such that $1=\sum_{n=1}^{\infty}r_{n}r_{n}^{*}$ holds in the strict topology.
\begin{definition}
Let $\Ga:G\act A$ and $\Gb:G\act B$ be two actions on C$^{*}$-algebras such that $A$ is separable and $B$ is $\Gs$-unital. An $(\Ga,\Gb)$-\textit{Cuntz pair} is a pair of cocycle representations
\[(\phi,\bbm{u}),(\psi,\bbm{v}):(A,\Ga)\to(\Mcal(B\otimes\Kbb),\Gb\otimes \id{\Kbb})\]
such that $\phi(a)-\psi(a)$ and $\bbm{u}_{g}-\bbm{v}_{g}$ are in $B\otimes \Kbb$ for all $a\in A$ and $g\in G$. If $\Gb$ is strongly stable we also allow $(B,\Gb)$ in place of $(B\otimes\Kbb,\Gb\otimes\id{\Kbb})$ appearing in the definition.
\end{definition}
\begin{definition}
Let $\Gb:G\act B$ be an action. Suppose that there exists a unital inclusion $\otw\subset \Mcal(B)^{\Gb}$. For two isometries $t_{1},t_{2}\in \Mcal(B)^{\Gb}$ with $t_{1}t_{1}^{*}+t_{2}t_{2}^{*}=1$, we may consider an equivariant $*$-homomorphism
\[\Mcal(B)\oplus \Mcal(B)\rightarrow \Mcal(B),\; b_{1}\oplus b_{2}\mapsto b_{1}\oplus_{t_{1}t_{2}}b_{2}\coloneqq t_{1}b_{1}t_{1}^{*}+t_{2}b_{2}t_{2}^{*}.\]
This $*$-homomorphism does not depend on the choice of $t_{1}$ and $t_{2}$ up to unitary equivalence with a unitary in $\Mcal(B)^{\Gb}$. Let $\Ga:G\act A$ be another action on a C$^{*}$-algebra, and $(\phi,\bbm{u}),(\psi,\bbm{v}):(A,\Ga)\to(\Mcal(B),\Gb)$ be cocycle representations. We define the \textit{Cuntz sum}
\[(\phi,\bbm{u})\oplus_{t_{1},t_{2}}(\psi,\bbm{v})=(\phi\oplus_{t_{1},t_{2}}\psi,\bbm{u}\oplus_{t_{1},t_{2}}\bbm{v}):(A,\Ga)\to(\Mcal(B),\Gb).\]
Since it does not depend on the choice of isometries, up to unitary equivalence, we will omit $t_{1}$ and $t_{2}$ if it is clear from context.
\end{definition}
\begin{definition}
Let $A$ be a separable C$^{*}$-algebra and $B$ be a $\Gs$-unital C$^{*}$-algebra. For two actions $\Ga:G\act A$ and $\Gb:G\act B$, let $\Ebb^{G}(\Ga,\Gb)$ denote the set of all $(\Ga,\Gb)$-Cuntz pairs.

Two elements $((\phi^{0}\bbm{u}^{0}),(\psi^{0},\bbm{v}^{0}))$ and $((\phi^{1}\bbm{u}^{1}),(\psi^{1},\bbm{v}^{1}))$ in $\Ebb^{G}(\Ga,\Gb)$ are homotopic, denoted by $((\phi^{0}\bbm{u}^{0}),(\psi^{0},\bbm{v}^{0}))\sim_{h}((\phi^{1}\bbm{u}^{1}),(\psi^{1},\bbm{v}^{1}))$ if there exists an $(\Ga,\Gb\otimes\id{C[0,1]})$-Cuntz pair that restricts to $((\phi^{0}\bbm{u}^{0}),(\psi^{0},\bbm{v}^{0}))$ at $t=0$ and restricts to $((\phi^{1}\bbm{u}^{1}),(\psi^{1},\bbm{v}^{1}))$ at $t=1$. An $(\Ga,\Gb)$-Cuntz pair with $\phi=\psi=0$ is called a cocycle pair and is denoted by $(\bbm{u},\bbm{v})$. We define $\Ebb_{0}^{G}(\Ga,\Gb)$ to be the set of all anchored $(\Ga,\Gb)$-Cuntz pairs: an element $((\phi,\bbm{u}),(\psi,\bbm{v}))\in \Ebb^{G}(\Ga,\Gb)$ is said to be anchored if $((0,\bbm{u}),(0,\bbm{v}))\sim_{h}((0,\mathbf{1}),(0,\mathbf{1}))$, that is, $(\bbm{u},\bbm{v})\sim_{h}(\mathbf{1},\mathbf{1})$ in our notation.

For any unital inclusion $\otw\subset \Mcal(B\otimes \Kbb)^{\Gb\otimes\id{\Kbb}}$ with isometries $t_{1},t_{2}$ that satisfy the Cuntz relation, we can define the Cuntz addition for two $(\Ga,\Gb)$-Cuntz pairs by
\begin{align*}&((\phi^{0}\bbm{u}^{0}),(\psi^{0},\bbm{v}^{0}))\oplus_{t_{1},t_{2}}((\phi^{1}\bbm{u}^{1}),(\psi^{1},\bbm{v}^{1}))\\
  &=((\phi^{0}\bbm{u}^{0})\oplus_{t_{1},t_{2}}(\phi^{1}\bbm{u}^{1}),(\psi^{0},\bbm{v}^{0})\oplus_{t_{1},t_{2}}(\psi^{1},\bbm{v}^{1})).
\end{align*}
This Cuntz pair is independent of the choice of $t_{1},t_{2}$ up to homotopy.
\end{definition}
The quotient $\Ebb^{G}(\Ga,\Gb)/\sim_{h}$ is an abelian group with Cuntz addition. The homotopy classes of cocycle pairs form a subgroup, denoted by $H_{\Gb}$. It was proved by Thomsen (\cite[Theorem 3.5]{Th}) that the group quotient of $\Ebb^{G}(\Ga,\Gb)/\sim_{h}$ modulo $H_{\Gb}$ is naturally isomorphic to $KK^{G}(\Ga,\Gb)$. For an $(\Ga,\Gb)$-Cuntz pair $((\phi,\bbm{u}),(\psi,\bbm{v}))$, we denote its equivalence class in $KK^{G}(\Ga,\Gb)$ by $[(\phi,\bbm{u}),(\psi,\bbm{v})]$. Under this identification, the inclusion map $\Ebb^{G}_{0}(\Ga,\Gb)\subset \Ebb^{G}(\Ga,\Gb)$ induces a natural isomorphism of abelian groups $\Ebb^{G}_{0}(\Ga,\Gb)/\sim_{h}\cong KK^{G}(\Ga,\Gb)$.
\begin{remark}\label{rem:kk-elem}
Let $\Gb:G\act B$ be an action and $\bbm{u}:G\to \Ucal(\Mcal(B))$ be a $\Gb$-cocycle. We denote by $B^{\bbm{u}}$  the Hilbert $(B,\Gb)$-module that is equal to $B$ as a non-equivariant Hilbert $B$-module with group action defined by $g\cdot b=\bbm{u}_{g}\Gb_{g}(b)$ for all $g\in G$ and $b\in B$. For any cocycle morphism $(\phi,\bbm{u}):(A,\Ga)\to(B,\Gb)$, we have a Kasparov triple $(B^{\bbm{u}},\phi,0)$ and denote its equivalence class in $KK^{G}(\Ga,\Gb)$ by $KK^{G}(\phi,\bbm{u})$. 
\end{remark}
\begin{proposition}[{\cite[Proposition 1.12]{GS2}}]\label{prop:kk-elem}
Let $(\phi,\bbm{u}):(A,\Ga)\to(B,\Gb)$ be a cocycle morphism. Then $KK^{G}(\phi,\bbm{u})$ is represented by the Kasparov triple $(\overline{\phi(A)B^{\bbm{u}}},\phi,0)$.
\end{proposition}
\begin{proof}
It suffices to show that $(B^{\bbm{u}},\phi,0)$ and $(\overline{\phi(A)B^{\bbm{u}}},\phi,0)$ are homotopic. Let $D=\{f\in C([0,1],B)\:|\: f(0)\in \overline{\phi(A)B}\}$ be the Hilbert  $(B[0,1],\Gb[0,1])$-module with $G$-action given by $(g\cdot f)(t)=\bbm{u}_{g}\Gb_{g}(f(t))$. Let $\Phi:A\to \Kbb(D)$ be the representation given by $(\Phi(a)f)(t)=\phi(a)f(t)$. Then, $(D,\Phi,0)$ is the desired homotopy.
\end{proof}
\begin{proposition}[{\cite[Corollary 1.13]{GS2}}]
Let $(\phi,\bbm{u}):(A,\Ga)\to(B,\Gb)$ and $(\psi,\bbm{v}):(B,\Gb)\to(C,\Gg)$ be proper cocycle morphisms. Then
\[KK^{G}(\phi,\bbm{u})\otimes KK^{G}(\psi,\bbm{v})=KK^{G}((\psi,\bbm{v})\circ (\phi,\bbm{u}))\]
where $(\psi,\bbm{v})\circ (\phi,\bbm{u})=(\psi\circ\phi,\tilde{\psi}(\bbm{u})\bbm{v})$.
\end{proposition}
\begin{proof}
$KK^{G}(\phi,\bbm{u})\otimes KK^{G}(\psi,\bbm{v})$ is represented by the Kasparov triple $(E,\Gk,0)$, where $E=\overline{\phi(A)B^{\bbm{u}}}\otimes_{\psi}\overline{\psi(B)C^{\bbm{v}}}$ and $\Gk=\phi\otimes \mathbf{1}$. We have $E\cong \overline{(\psi\circ\phi)(A)C^{\tilde{\psi}(\bbm{u})\bbm{v}}}$ via $b\otimes c\mapsto \psi(b)c$ as Hilbert $C$-modules. Under this identification, $\Gk$ coressponds to $\psi\circ\phi$. We claim that this isomorphism is equivariant. The $G$-action on $E$ is given by $g\cdot(b\otimes c)=\bbm{u}_{g}\Gb_{g}(b)\otimes\bbm{v}_{g}\Gg_{g}(c)$. This element is mapped to
\[\psi(\bbm{u}_{g}\Gb_{g}(b))\bbm{v}_{g}\Gg_{g}(c)=\tilde{\psi}(\bbm{u}_{g}\psi(\Gb_{g}(b))\bbm{v}_{g}\Gg_{g}(c))=\tilde{\psi}(\bbm{u}_{g})\bbm{v}_{g}\Gg_{g}(\psi(b)c)=g\cdot (\psi(b)c).\]
Since $KK^{G}((\psi,\bbm{v})\circ (\phi,\bbm{u}))$ is represented by $(\overline{(\psi\circ\phi)(A)C^{\tilde{\psi}(\bbm{u})\bbm{v}}},\psi\circ\phi,0)$, $KK^{G}(\phi,\bbm{u})\otimes KK^{G}(\psi,\bbm{v})=KK^{G}((\psi,\bbm{v})\circ (\phi,\bbm{u}))$ holds.
\end{proof}

\subsection{Isometrically shift-absorbing actions}\label{sec:isa}
First, we recall the definition of isometrically shift-absorbing actions. The definition of isometrically shift-absorbing actions requires some preparation.

Let $H$ be an infinite-dimensional separable Hilbert space. The Cuntz algebra $\oinf$ is isomorphic to $\Ocal_{H}$, the universal unital C$^{*}$-algebra generated by the range of a linear map $\Fs:H\to \Ocal_{H}$ subject to the relation $\Fs(\Gx)^{*}\Fs(\eta)=\ip{\Gx}{\eta}\mathbf{1}$ for all $\Gx,\eta\in H$. For every unitary $U$ on $H$, there is a unique automorphism on $\Ocal_{H}$ such that $\Fs(\Gx)$ is sent to $\Fs(U\Gx)$ for all $\Gx\in H$. A unitary representation $G\to\Ucal(H)$ induces a $G$-action on $\Ocal_{H}$ via the correspondence described above. The action defined in this way is called a quasi-free action.

Let $H_{F}=\bigoplus_{n=0}^{\infty}H^{\otimes n}$ be the Fock space and consider the linear map $\Fs_{F}:H\to \BB(H_{F})$ given by $\Fs_{F}(\Gx)(x)=\Gx\otimes x$ for all $\Gx\in H,$ $x\in H^{\otimes n}$ and $n\geq 0$. Then $\Fs_{F}$ induces the Fock representation $\pi:\Ocal_{H}\to\BB(H_{F})$, which is irreducible. 
\begin{proposition}[{\cite[Proposition 3.2]{GS2}}]
Every automorphism on $\Ocal_{H}$ induced by a unitary $U\in \Ucal(H)\setminus \{1\}$ is outer.
\end{proposition}
\begin{definition}\label{def:qf}
Let $G$ be a second-countable, compact group with a normalized Haar measure $\Gm$. We denote $H_{G}=L^{2}(G,\Gm)$ and $H^{\infty}_{G}=\ell^{2}(\Nbb)\otimes H_{G}$ which are separable Hilbert spaces. The latter is infinite-dimensional. The left regular representation $\Gl:G\to\Ucal(H_{G})$ is given by $\Gl_{g}(\Gx)(h)=\Gx(g^{-1}h)$ for all $g,h\in G$ and $\Gx\in H_{G}$. Then its infinite direct sum is denoted by $\Gl^{\infty}:G\to \Ucal(H^{\infty}_{G})$. In this subsection, we denote by $\Gg:G\act \Ocal_{H^{G}_{\infty}}\cong\oinf$ the quasi-free action determined by $\Gg_{g}\circ \Fs=\Fs\circ \Gl_{g}^{\infty}$ for all $g\in G$.
\end{definition}
\begin{definition}
Let $\Gb:G\act B$ be an action on a C$^{*}$-algebra. We denote by $\ell^{\infty}_{\Gb}(\Nbb,B)$ the C$^{*}$-algebra of every bounded sequence $(b_{n})_{n}$ such that the map $G\ni g\mapsto (\Gb_{g}(b_{n}))_{n}\in\ell^{\infty}(\Nbb,B)$ is continuous. The quotient $\ell^{\infty}_{\Gb}(\Nbb,B)/c_{0}(\Nbb,B)$ is denoted by $B_{\infty,\Gb}$ and the action induced by $\Gb$ is denoted by $\Gb_{\infty}$. We note that $B$ canonically embeds into $B_{\infty,\Gb}$ as constant sequences. The ($\Gb$-continuous) central sequence algebra is
\[F_{\infty,\Gb}(B)=(B_{\infty,\Gb}\cap B')/(B_{\infty,\Gb}\cap B^{\perp}),\]
where $B'$ is the commutant of B and $B_{\infty,\Gb}\cap B^{\perp}=\{x\in B_{\infty,\Gb}\:|\: xb=bx=0\text{ for all }b\in B\}$. The $G$-action on $F_{\infty,\Gb}(B)$ induced by $\Gb_{\infty}$ is denoted by $\overline{\Gb}_{\infty}$. If $B$ is $\Gs$-unital, then $F_{\infty,\Gb}(B)$ is unital, and the unit is represented by any sequential approximate unit of $B$.
\end{definition}
\begin{definition}[{\cite[Definition 3.7]{GS2}}]\label{def:isa}
Let $\Gb:G\act B$ be an action on a separable \cstar-algeba. $\Gb$ is \textit{isometrically shift-absorbing} if there is a linear equivariant map
\[\Fs:(H_{G},\Gl)\to(F_{\infty,\Gb}(B),\tilde{\Gb}_{\infty})\]
satisfying the identity $\Fs(\Gx)^{*}\Fs(\Gn)=\ip{\Gx}{\Gn}\mathbf{1}$ for all $\Gx,\Gn\in H_{G}$.
\end{definition}
\begin{proposition}[{\cite[Proposition 3.8]{GS2}}]\label{prop:isa}
Let $G$ be a second-countable, compact group with more than one element. Let $\Gb:G\act B$ an action on a separable \cstar-algebra, and let $\Gg:G\act \oinf$ be the action in \autoref{def:qf}. Then, the followings are equivalent:\\
(i) $\Gb$ is isometrically shift-absorbing\\
(ii) There exists a unital equivariant $*$-homomorphism from $(\oinf,\Gg)$ to $(F_{\infty,\Gb}(B),\tilde{\Gb}_{\infty})$.
\end{proposition}
\begin{proposition}[{\cite[Proposition 3.9]{GS2}}]\label{prop:oinf-ab}
Let $G$ be a second countable, compact group with more than one element. Let $\Gb:G\act B$ be an action on a separable \cstar-algebra. If $\Gb$ is isometrically shift absorbing, then $\Gb$ is equivariantly $\oinf$-absorbing, i.e., $\Gb\simeq_{cc}\Gb\otimes\id{\oinf}$.
\end{proposition}
\begin{theorem}[{\cite[Theorem 4.8]{GS2}}]\label{thm:otw-embed}
  Let $G$ be a compact group and let $A$ be a separable exact \cstar-algebra with an action $\Ga:G\act A$. Suppose that $\Gb:G\act B$ is an isometrically shift-absorbing action on a Kirchberg algebra with $\Gb\simeq_{cc}\Gb\otimes \id{\otw}$. Then there exists a proper cocycle embedding from $(A,\Ga)$ to $(B,\Gb)$.
\end{theorem}
For a compact group $G$, let us denote by $\hat{G}$ the set of equivalence classes of irreducible representations of $G$. For $\pi\in \hat{G}$, we denote $\Gc_{\pi}(g)=\mathrm{Tr}(\pi(g))$ and $d(\pi)=\mathrm{dim}(\pi)$. We choose and fix an orthonormal basis $\{\Gx(\pi)_{i}\}_{i=1}^{d(\pi)}$ of the representation space $H_{\pi}$ of $\pi$, and identify $\pi(g)$ with its matrix representation $(\pi(g)_{ij})_{i,j}$. Let $A$ be a unital \cstar-algebra with a $G$-action $\Ga$. Let $\{\Gl_{g}\}_{g\in G}$ be the implementing unitary representation of $G$ in $\Mcal(A\rtimes_{\Ga}G)$. For $(\pi, H_{\pi})\in \hat{G}$, set
\[E(\pi)_{ij}=\mathrm{dim}(\pi)\int_{G}\overline{\pi(g)}_{ij}\Gl_{g} dg,\]
where $dg$ is the normalized Haar measure. By the Peter-Weyl theorem, we have $C^{*}(L^{1}(G))\cong \bigoplus_{\pi\in \hat{G}}\BB(H_{\pi})$. Then $\{E(\pi)_{ij}\}$ forms a system of matrix units of $\BB(H_{\pi})$. For $\pi\in\hat{G}$,
\[z(\pi)\coloneqq\sum_{i=1}^{d(\pi)}E(\pi)_{ii}=d(\pi)\int_{G}\overline{\Gc_{\pi}(g)}\Gl_{g}dg\]
is a central projection of $C^{*}(L^{1}(G))$ and $\sum_{\pi\in\hat{G}}z(\pi)=1$ in the strict topology in $\Mcal(A\rtimes_{\Ga}G)$. Let $e_{\Ga}=z(\mathbf{1})=\int_{G}\Gl_{g}dg$. Since $A\rtimes_{\Ga}G=C^{*}(A\cup\{\int_{G}f(g)\Gl_{g}dg\:|\:f\in L^{1}(G)\})$, we have $e_{\Ga}(A\rtimes_{\Ga}G)e_{\Ga}=A^{\Ga}e_{\Ga}$.
\begin{proposition}\label{prop:isa-sat}
Let $G$ be a compact group and let $A$ be a unital Kirchberg algebra. Let $\Ga:G\act A$ be an isometrically shift-absorbing action. Then $\Ga$ is saturated.
\end{proposition}
\begin{proof}
By \cite[Corollary 2.6]{Br}, it suffices to show that $e_{\Ga}$ is a full projection.

By \cite[Corollary 6.8]{GS2} and \cite[Theorem 3.7 and Proposition 4.8]{SZ1}, we have an equivariant $*$-isomorphism $\phi:(A\otimes \oinf,\Ga\otimes \Gg^{\infty})\to(A,\Ga)$. Let $\Fs_{\pi}:(H_{\pi},\pi)\to (\oinf,\Gg^{\infty})$ be an equivariant linear map satisfying $\Fs(\Gx)^{*}\Fs(\Gn)=\ip{\Gx}{\Gn}\mathbf{1}$ for all $\Gx,\Gn\in H_{\pi}$. Set $S(\pi)_{i}=\phi(\mathbf{1}_{A}\otimes \Fs_{\pi}(\Gx(\pi)_{i}))$ for $i=1,\cdots,d(\pi)$. Then we have $S(\pi)_{i}^{*}S(\pi)_{j}=\Gd_{ij}\mathbf{1}_{A}$ and 
\begin{align*}
\Ga_{g}(S(\pi)_{i})=\phi(\mathbf{1}_{A}\otimes \Fs_{\pi}(\pi(g)\Gx(\pi)_{i}))=\phi(\mathbf{1}_{A}\otimes \Fs_{\pi}(\sum_{j=1}^{d(\pi)}\pi(g)_{ji}\Gx(\pi)_{j}))=\sum_{j=1}^{d(\pi)}\pi(g)_{ji}S(\pi)_{j}.
\end{align*}
Then it follows that
\begin{align*}
\sum_{i=1}^{d(\pi)}S(\pi)_{i}^{*}e_{\Ga}S(\pi)_{i}&=\int_{G}S(\pi)_{i}^{*}\Ga_{g}(S(\pi)_{i})\Gl_{g}dg\\
&=\sum_{i=1}^{d(\pi)}\int_{G}\pi(g)_{ii}\Gl_{g}dg\\
&=\frac{1}{d(\pi)}z(\overline{\pi}).
\end{align*}
This implies that the ideal generated by $e_{\Ga}$ contains $z(\pi)$ for all $\pi\in\hat{G}$. Since $\sum_{\pi\in\hat{G}}z(\pi)=\mathbf{1}$ in the strict topology in $\Mcal(A\rtimes_{\Ga}G)$, we can conclude that $e_{\Ga}$ is a full projection.
\end{proof}

\subsection{Gabe-Szab\'o type theorems}
We now turn to the preparation for Gabe-Szab\'o type theorems. We begin by recalling some notation concerning cocycle representations.
\begin{definition}
Let $(\phi,\bbm{u}),(\psi,\bbm{v}):(A,\Ga)\to(\Mcal(B),\Gb)$ be two cocycle representations. We say that $(\phi,\bbm{u})$ and $(\psi,\bbm{v})$ are asymptotically unitarily equivalent, denoted by $(\phi,\bbm{u})\asymp(\psi,\bbm{v})$, if there exists a norm-continuous path $u:[0,\infty)\to\Ucal(\Mcal(B))$ such that
\begin{align*}
&\text{(i)}\;\psi(a)=\lim_{t\to\infty}u_{t}\phi(a)u_{t}^{*}\text{ for all }a\in A,\\
&\text{(ii)}\;\psi(a)-u_{t}\phi(a)u_{t}^{*}\in B\text{ for all }a\in A\text{ and }t\geq 0,\\
&\text{(iii)}\;\lim_{t\to\infty}\max_{g\in G}\|\bbm{v}_{g}-u_{t}\bbm{u}_{g}\Gb_{g}(u_{t})^{*}\|=0\\
&\text{(iv)}\;\bbm{v}_{g}-u_{t}\bbm{u}_{g}\Gb_{g}(u_{t})^{*}\in B\text{ for all }t\geq 0\text{ and }g\in G.
\end{align*}
If it is possible to choose $u$ to have its range in $\Ucal(\mathbf{1}+B)$, then $(\phi,\bbm{u})$ and $(\psi,\bbm{v})$ are called \textit{properly asymptotically unitarily equivalent}. If it is additionally possible to arrange $u_{0}=\mathbf{1}$, then we say that $(\phi,\bbm{u})$ and $(\psi,\bbm{v})$ are \textit{strongly asymptotically unitariliy equivalent}.
\end{definition}
\begin{definition}
Suppose that there exists a unital inclusion $\otw\subset\Mcal(B)^{\Gb}$. Let $(\phi,\bbm{u}),(\psi,\bbm{v}):(A,\Ga)\to(\Mcal(B),\Gb)$ be two cocycle representations. We say that $(\phi,\bm{u})$ \textit{absorbs} $(\psi,\bbm{v})$ if $(\phi\oplus\psi,\bbm{u}\oplus\bbm{v})\asymp(\phi,\bbm{u})$. A cocycle representation is called \textit{absorbing} if it absorbs every cocycle representation.
\end{definition}
\begin{definition}
Suppose that $\Gb$ is strongly stable. Let $r_{n}\in \Mcal(B)^{\Gb}$ be any sequence of isometries such that $\sum_{n=1}^{\infty}r_{n}r_{n}^{*}=\mathbf{1}$ in the strict topology. Then we have a $\Gb$-equivariant $*$-homomorphism
\[\ell^{\infty}(\Nbb,\Mcal(B))\to\Mcal(B),\quad (b_{n})_{n}\mapsto\sum_{n=1}^{\infty}r_{n}b_{n}r_{n}^{*}\]
which does not depend on the choice of $r_{n}$ up to unitary equivalence. For any sequence of cocycle representations $(\phi^{(n)},\bbm{u}^{(n)}):(A,\Ga)\to(\Mcal(B),\Gb)$, we define the countable sum
\[(\Phi,\bbm{U})=\bigoplus_{n=1}^{\infty}(\phi^{(n)},\bbm{u}^{(n)}):(A,\Ga)\to(\Mcal{B},\Gb)\]
via the pointwise strict limits
\[\Phi(a)=\sum_{n=1}^{\infty}r_{n}\phi^{(n)}(a)r_{n}^{*},\quad \bbm{U}_{g}=\sum_{n=1}^{\infty}r_{n}\bbm{u}_{g}^{(n)}r_{n}^{*}.\]
This cocycle representation does not depend on the choice of $(r_{n})$ up to unitary equivalence via a unitary in $\Mcal(B)^{\Gb}$. In the special case that $(\phi^{(n)},\bbm{u}^{(n)})=(\phi,\bbm{u})$ for all $n$, we denote the resulting countable sum by $(\phi^{\infty},\bbm{u}^{\infty})$ and call it the infinite repeat of $(\phi,\bbm{u})$.
\end{definition}
\begin{proposition}[{\cite[Corollary 3.17]{GS2}}]\label{prop:absorbing}
Let $\Ga:G\act A$ and $\Gb:G\act B$ be two actions on separable \cstar algebras, and assume that $\Gb$ is isometrically shift-absorbing and strongly stable. Suppose that $A$ or $B$ is nuclear. Let $(\phi,\bbm{u}):(A,\Ga)\to(B,\Gb)$ be a cocycle morphism such that $\phi$ is full. Then the infinite repeat $(\phi^{\infty},\bbm{u}^{\infty}):(A,\Ga)\to(\Mcal(B),\Gb)$ is an absorbing cocycle representation.
\end{proposition}
\begin{lem}[{\cite[Lemma 7,1]{Ga1}}]\label{lem:unitary-path}
There exists a continuous map $u:[0,\infty)\to \Ucal(1+\otw\otimes\Kbb)$ with $u_{0}=\one$ such that\\
(i)\;$u_{t}(1\otimes e_{1,1})u_{t}\to \one$ in the strict topology as $t\to\infty$,\\
(ii)\;for all $x\in\otw\otimes\Kbb$, one has that $u_{t}x$ converges in norm as $t\to\infty$
\end{lem}

\begin{lem}\label{lem:commuting}
Let $(A,\Ga)$ and $(B,\Gb)$ be $G$-\cstar-algebras and let $X$ be a compact Hausdorff space. Suppose that $\Gb$ is strongly stable and conjugate to $\Gb\otimes\id{\oinf}$. Assume that there exists a proper cocycle embedding from $(A,\Ga)$ into $(B\otimes\otw,\Gb\otimes\id{\otw})$. Then there exists a proper cocycle embedding $(\theta,\bbm{y}):(A,\Ga)\to (B\otimes C(X),\Gb\otimes\id{C(X)})$ together with a unital embedding $\Gi_{0}:\otw\to \Mcal(B\otimes C(X))^{\Gb\otimes\id{C(X)}}$ whose range commutes with the ranges of $\theta$ and $\bbm{y}$.
\end{lem}
\begin{proof}
It follows from [{\cite[Corollary 5.3]{GS2}}] that there exists a proper cocycle embedding $(\theta',\bbm{y}'):(A,\Ga)\to(B,\Gb)$ together with a unital embedding $\iota_{0}':\otw\to \Mcal(B)^{\Gb}$ whose range commutes with the ranges of $\theta'$ and $\bbm{y}'$. The claim follows by composing $(\theta',\bbm{y}')$ with $\id{B}\otimes \mathbf{1}:(B,\Gb)\to(B\otimes C(X),\Gb\otimes \id{C(X)})$ and $\iota_{0}'$ with $\id{B}\otimes\mathbf{1}:\Mcal(B)^{\Gb}\hookrightarrow \Mcal(B)^{\Gb}\otimes C(X)\subset \Mcal(B\otimes C(X))^{\Gb\otimes \id{C(X)}}$.
\end{proof}
\begin{lem}\label{lem:comm-range}
Let $G$ be a compact group. Let $A$ be a separable nuclear \cstar-algebra and $B$ be a separable \cstar-algebra with $B\cong B\otimes \oinf\otimes\Kbb$. Let $\Ga$ be an action on $A$, and $\Gb$ be a strongly stable, isometrically shift-absorbing action on $B$. Let 
\[(\phi,\bbm{u}),(\theta,\bbm{y}):(A,\Ga)\to(B,\Gb)\]
be two proper cocycle morphisms. Suppose $\phi$ is full and there exist a unital embedding $\otw\to\Mcal(B)^{\Gb}$ commuting with the ranges of $\theta$ and $\bbm{y}$, and a unital embedding $\oinf\to\Mcal(B)^{\Gb}$ commuting with the ranges of $\phi$ and $\bbm{u}$. Then $(\phi,\bbm{u})$ and $(\phi\oplus\theta,\bbm{u}\oplus\bbm{y})$ are strongly asymptotically unitarily equivalent.
\end{lem}
\begin{proof}
By \cite[Lemma 2.8 and Lemma 3.16]{GS2}, it suffices to show that for every finite subset $F\subset A$, every $\Ge$, and every contraction $b\in B$, there exists $c\in B$ such that
\[\|b^{*}b-c^{*}c\|\leq \Ge,\qquad\max_{a\in F}\|b^{*}\theta(a)b-c^{*}\phi(a)c\|\leq \Ge.\]
Fix $F$, $\Ge$ and $b$.
By \cite[Proposition 3.12]{Ga1}, there exist $c_{1},\cdots ,c_{n}\in B$ such that 
\[\max_{a\in F}\|\theta(a)-\sum_{i=1}^{n}c_{i}^{*}\phi(a)c_{i}\|\leq \Ge.\]
Let $\theta':\phi(A)\to B$ be a $*$-homomorphism defined by $\theta'(\phi(a))=\theta(a)$. This is a well-defined map because $\phi$ is injective. Set $F'=\phi(F)$. Then, by \cite[Theorem 7.21]{KR2}  there exists $c\in B$ such that 
\[\max_{x\in F'}\|\theta'(x)-c^{*}xc\|\leq \Ge.\]
By \cite[Lemma 7.4]{KR2}, there exists an isometry $s\in \Mcal(B)$ such that $\max_{x\in F'}\|\theta'(x)-s^{*}xs\|\leq \frac{\Ge}{\|b\|^{2}+1}$. We set $c=sb$. Then we obtain
\begin{align*}
  \max_{a\in F}\|b^{*}\theta(a)b-c^{*}\phi(a)c\|&=\max_{x\in F'}\|b^{*}\theta'(x)b-c^{*}xc\|\leq \max_{x\in F'}\|b\|^{2}\|\theta'(x)-s^{*}xs\|<\Ge\\
  \|b^{*}b-c^{*}c\|&=\|b^{*}b-b^{*}b\|=0
\end{align*}
\end{proof}
\begin{remark}
Suppose that $\Gb$ is strongly stable. For a proper cocycle morphism $(\phi,\bbm{u}):(A,\Ga)\to(B,\Gb)$, we denote the homotopy class of the Cuntz pair $((\phi,\bbm{u}),(0,\one))$ by $[(\phi,\bbm{U})]_{h}$. $(\phi,\bbm{u})$ is anchored if $(\bbm{u},\mathbf{1})\sim_{h}(\mathbf{1},\mathbf{1})$.
\end{remark}
\subsubsection*{Existence theorem}
The proof of the following theorem is almost identical to that in \cite{GS2}.
\begin{theorem}[cf. {\cite[Theorem 5.5]{GS2}}]\label{thm:ex-stable}
>Let $G$ be a compact group. Let $A$ be a separable exact C$^{\ast}$-algebra and let $\Ga:G\act A$ be an action. Let $B$ be a Kirchberg algebra and $\Gb:G\act B$ a strongly stable and isometrically shift-absorbing action. Let $X$ be a compact metrizable space. Then\\
(i) For every $z\in\Ebb^{G}(\Ga,\Gb\otimes \id{C(X)})/\sim_{h}$, there exists a proper cocycle embedding $(\phi,\bbm{u}):(A,\Ga)\to (B\otimes C(X),\Gb\otimes \id{C(X)})$ such that $[(\phi,\bbm{u})]_{h}=z$.\\
(ii) For every $z\in KK^{G}(\Ga,\Gb\otimes\id{C(X)})$, there exists an anchored proper cocycle embedding $(\phi,\bbm{u}):(A,\Ga)\to(B\otimes C(X),\Gb\otimes\id{C(X)})$ with $KK^{G}(\phi,\bbm{u})=z$.
\end{theorem}
\begin{proof}
  Because (ii) is a special case of (i), we only prove (i)

  We know from \autoref{prop:oinf-ab} that $\Gb\simeq_{cc}\Gb\otimes\id{\oinf}$. By \cite[Theorem 5.6]{SZ3}, it follows that $\id{B}\otimes 1:(B,\Gb)\to(B\otimes\oinf,\Gb\otimes\id{\oinf})$ is strongly asymptotically unitarily equivalent to a proper cocycle conjugacy, which is anchored since the first factor embedding is trivially anchored. So we may replace $(B,\Gb)$ with $(B\otimes \oinf,\Gb\otimes\id{\oinf})$ and assume that $\Gb$ is conjugate to $\Gb\otimes\id{\oinf}$. Apply \autoref{lem:commuting} and \autoref{thm:otw-embed} with $\Gb\otimes\id{\otw}$ in place of $\Gb$. This allows us to find a proper cocycle embedding $(\theta,\bbm{y}):(A,\Ga)\to(B\otimes C(X),\Gb\otimes\id{C(X)})$ and a unital embedding $\Gi_{0}:\otw\to \Mcal(B\otimes C(X))^{\Gb\otimes\id{C(X)}}$ whose range commutes with $\theta$ and $\bbm{y}$.

Since $\beta$ is strongly stable, let us choose a sequence of isometries $r_{n}\in \Mcal(B\otimes C(X))^{\Gb\otimes\id{C(X)}}$ such that $\sum_{n=1}^{\infty}r_{n}r_{n}^{*}=\one$ in the strict topology. Let $\{e_{k,l}\:|\: k,l\geq 1\}$ be a set of matrix units generating $\Kbb$. We then consider the non-degenerate embedding
\[\Gi:\otw\otimes\Kbb\to\Mcal(B\otimes C(X))^{\Gb\otimes\id{C(X)}},\quad \Gi(x\otimes e_{k,l})=r_{k}\Gi_{0}(x)r_{l}^{*}\]
Using \autoref{lem:unitary-path}, we may find a continuous unitary path $u:[0,\infty)\to \Ucal(1+\Gi(\otw\otimes\Kbb))\subset \Ucal(\Mcal(B\otimes C(X))^{\Gb\otimes\id{C(X)}})$ with $u_{0}=\one$ such that
\begin{align}\one=\lim_{t\to\infty}u_{t}^{*}(1\otimes e_{1,1})u_{t}=\lim_{t\to\infty}u_{t}^{*}r_{1}r_{1}^{*}u_{t}\label{eq:1}
\end{align}
holds in the strict topology as $t\to\infty$, and moreover $u_{t}x$ converges in norm as $t\to\infty$ for all $x\in B\otimes C(X)$.

By \autoref{lem:full-hom}, $\theta$ is full. So by \autoref{prop:absorbing}, it follows that $(\theta^{\infty},\bbm{y}^{\infty})$ is an absorbing cocycle representation, where we note that the infinite repeat is meant to be formed via the sequence $r_{n}$. Thus the range of $\Gi$ commutes pointwise with the ranges of $\theta^{\infty}$ and $y^{\infty}$. It follows from \cite[Corollary 3.17]{GS1} that we can find some cocycle representation $(\psi,\bbm{v}):(A,\Ga)\to (\Mcal(B\otimes C(X)),\Gb\otimes \id{C(X)})$ that forms an $(\Ga,\Gb\otimes\id{C(X)})$-Cuntz pair together with $(\theta^{\infty},\bbm{y}^{\infty})$ such that one has $z=[(\psi,\bbm{v}),(\theta^{\infty},\bbm{y}^{\infty})]_{h}$.

By the properties of the unitary path $u$ constructed above, it follows that for all $a\in A$ 
\[u_{t}\psi(a)u_{t}^{*}=u_{t}(\psi(a)-\theta^{\infty}(a))u_{t}^{*}+\theta^{\infty}(a)\]
converges in norm as $t\to\infty$ because $\psi(a)-\theta^{\infty}(a)$ belongs to $B\otimes C(X)$. Let $\phi'=\lim_{t\to\infty}\ad{u_{t}}\circ\psi$ be the $*$-homomorphism arising as the point-norm limit. Then, for every $a\in A$
\begin{align*}
\|(1-r_{1}r_{1}^{*})(\phi'(a)-\theta^{\infty}(a))\|&=\lim_{t\to\infty}\|(\mathbf{1}-r_{1}r_{1}^{*})(u_{t}\psi(a)u_{t}^{*}-\theta^{\infty}(a))\|\\
&=\lim_{t\to\infty}\|(\mathbf{1}-u_{t}^{*}r_{1}r_{1}^{*}u_{t})(\psi(a)-\theta^{\infty}(a))\|\\
&\stackrel{\eqref{eq:1}}{=}0
\end{align*}
Let us consider the isometry $r_{\infty}=\sum_{n=1}^{\infty}r_{n+1}r_{n}^{*}\in\Mcal(B\otimes C(X))^{\Gb\otimes\id{C(X)}}$,which has the property that $r_{1}r_{1}^{*}+r_{\infty}r_{\infty}^{*}=\mathbf{1}$. Since all partial isomoetries of the form $r_{k}r_{l}^{*}$ commute with the range of $\theta^{\infty}$, we can conclude that the projection $r_{1}r_{1}^{*}$ commutes with the range of $\phi'$. Hence it follows that for all $a\in A$,
\begin{align*}
\phi'(a)&=r_{1}r_{1}^{*}\phi'(a)+(\mathbf{1}-r_{1}r_{1}^{*})\phi'(a)\\
&=r_{1}r_{1}^{*}\phi'(a)+(\mathbf{1}-r_{1}r_{1}^{*})\theta^{\infty}(a)\\
&=\phi(a)\oplus_{r_{1},r_{\infty}}\theta^{\infty}(a)
\end{align*}
where $\phi:A\to\Mcal(B\otimes C(X))$ is the $*$-homomorphism defined by $\phi(a)=r_{1}^{*}\phi'(a)r_{1}$. Appealing to the properties of the unitary path $u$, it follows that for all $g\in G$
\[u_{t}\bbm{v}_{g}u_{t}^{*}=u_{t}(\bbm{v}_{g}-\bbm{y}^{\infty}_{g})u_{t}^{*}+\bbm{y}^{\infty}_{g}\]
converges in norm as $t\to\infty$ because $\bbm{v}_{g}-\bbm{y}^{\infty}_{g}$ belongs to $B$ for all $g\in G$. Since the assignment $g\mapsto\bbm{v}_{g}-\bbm{y}^{\infty}_{g}$ is norm-continuous by \cite{SZ3}, this convergence is uniform over $G$. Let $\bbm{u}'_{\bullet}=\lim_{t\to\infty}u_{t}\bbm{v}_{\bullet}u_{t}^{*}$ be the $\Gb\otimes\id{C(X)}$-cocylce arising as the pointwise limit in norm. We observe for every $g\in G$ that
\begin{align*}
\|(\mathbf{1}-r_{1}r_{1}^{*})(\bbm{u}'_{g}-\bbm{y}^{\infty}_{g})\|&=\lim_{t}\|(\mathbf{1}-r_{1}r_{1}^{*})(u_{t}\bbm{v}_{g}u_{t}^{*}-\bbm{y}^{\infty}_{g})\|\\
&=\lim_{t\to\infty}\|(\mathbf{1}-u_{t}^{*}r_{1}r_{1}^{*}u_{t})(\bbm{v}_{g}-\bbm{y}^{\infty}_{g})\|\\
&\stackrel{\eqref{eq:1}}{=}0
\end{align*}
Since all partial isometries of the form $r_{k}r_{l}^{*}$ commute with the range of $\bbm{y}^{\infty}$, we can conclude that the projection $r_{1}r_{1}^{*}$ commutes with the range of $\bbm{u}'$. Hence it follows for all $g\in G$ that
\begin{align*}
\bbm{u}'_{g}&=r_{1}r_{1}^{*}\bbm{u}'_{g}+(\mathbf{1}-r_{1}r_{1}^{*})\bbm{u}'_{g}\\
&=r_{1}r_{1}^{*}\bbm{u}'_{g}+(\mathbf{1}-r_{1}r_{1}^{*})\bbm{y}^{\infty}_{g}\\
&=\bbm{u}_{g}\oplus_{r_{1},r_{\infty}}\bbm{y}^{\infty}_{g}
\end{align*}
where $\bbm{u}:G\to\Ucal(\Mcal(B))$ is the $\Gb\otimes \id{C(X)}$-cocycle defined as $\bbm{u}_{g}=r_{1}^{*}\bbm{u}'_{g}r_{1}$. In conclusion, we have constructed a cocycle representation $(\phi,\bbm{u}):(A,\Ga)\to(\Mcal(B\otimes C(X),\Gb\otimes \id{C(X)})$ such that the two $(\Ga,\Gb\otimes\id{C(X)})$-Cuntz pairs $((\phi,\bbm{u})\oplus_{r_{1},r_{\infty}}(\theta^{\infty},\bbm{y}^{\infty}),(\theta^{\infty},\bbm{y}^{\infty}))$ and $((\psi,\bbm{v}),(\theta^{\infty},\bbm{y}^{\infty}))$ are homotopic. Note that our choice of isometries to define the Cuntz sum leads to the equation $(\theta^{\infty},\bbm{y}^{\infty})=(\theta,\bbm{y})\oplus_{r_{1},r_{\infty}}(\theta^{\infty},\bbm{y}^{\infty})$. Thus $(\phi,\bbm{u})$ and $(\theta,\bbm{y})$ necessarily form an $(\Ga,\Gb\otimes\id{C(X)})$-Cuntz pair representing the class $z$. Since $(\theta,\bbm{y})$ is a proper cocycle morphism, so is $(\phi,\bbm{u})$. By construction, we have $[(\theta,\bbm{y})]_{h}=0$. Thus we can conclude that
\[z=[(\phi,\bbm{u}),(\theta,\bbm{y})]_{h}=[(\phi,\bbm{u})]_{h}-[(\theta,\bbm{y})]_{h}=[(\phi,\bbm{u})]_{h}\]
If $\phi$ is not an embedding, we may replace $(\phi,\bbm{u})$ by $(\phi,\bbm{u})\oplus(\theta,\bbm{y})$. Because $\phi\oplus\theta$ is full and the homotopy class does not change, the proof is completed.
\end{proof}

\begin{theorem}[cf. {\cite[Theorem 5.6]{GS2}}]\label{thm:ex-unital}
Suppose that $G$ is compact. Let $A,\;B$ be unital Kirchberg algebras and $\Ga:G\act A,\;\Gb:G\act B$ be isometrically shift-absorbing actions. Let $X$ be a compact connected metrizable space. \\
Then for every $x\in KK^{G}(\Ga,\Gb\otimes\id{C(X)})$ such that $[\one_{A^{G}}]\otimes x=[p]\in K^{G}_{0}(B\otimes C(X))=K_{0}((B\otimes C(X))^{\Gb\otimes \id{C(X)}})$ for some full properly infinite projection $p\in (B\otimes C(X))^{\Gb\otimes\id{C(X)}}$, there exists a cocycle embedding $(\psi,\bbm{v}):(A,\Ga)\to(B\otimes C(X),\Gb\otimes\id{C(X)})$ such that $\phi(\one_{A})=p$ and $KK^{G}(\psi,\bbm{v})=x$.
\end{theorem}
\begin{proof}
Let us consider the  invertible element $\Gk\in KK^{G}(\Gb\otimes \id{C(X)},\Gb\otimes\id{C(X)\otimes\Kbb})$ given by the canonical inclusion $B\otimes C(X)\ni x\mapsto x\otimes e_{1,1}\in B\otimes C(X)\otimes\Kbb$. By the existence theorem, there is a proper cocycle embedding $(\phi,\bbm{u}):(A,\Ga)\to(B\otimes C(X)\otimes\Kbb,\Gb\otimes\id{C(X)\otimes\Kbb})$ such that $KK^{G}(\phi,\bbm{u})=x\otimes \Gk$. By the assumptions, $\phi(\one_{A})$ and $p\otimes e_{1,1}$ are full properly infinite projections and represent the same element in $K_{0}(B\otimes C(X))$. Thus they are unitarily equivalent, that is, there is a unitary $U\in \Ucal(\one+B\otimes C(X)\otimes \Kbb)$ such that $U\phi(\one_{A})U^{*}=p\otimes e_{1,1}$. For all $g\in G$,
\begin{align*}
p\otimes e_{1,1}&=\ad{U}\circ\phi(\mathbf{1}_{A})\\
&=\ad{U}\circ\phi(\Ga_{g}(\mathbf{1}_{A}))\\
&=\ad{U}\circ \ad{\bbm{u}_{g}}\circ(\Gb_{g}\otimes \id{C(X)\otimes\Kbb})\circ\phi(\mathbf{1}_{A})\\
&=\ad({U}\bbm{u}_{g}(\Gb_{g}\otimes \id{C(X)\otimes \Kbb})(U)^{*})(\Gb_{g}\otimes\id{C(X)\otimes\Kbb})(U\phi(\mathbf{1}_{A})U^{*})\\
&=\ad({U}\bbm{u}_{g}(\Gb_{g}\otimes \id{C(X)\otimes \Kbb})(U)^{*}))(p\otimes e_{1,1})
\end{align*}
holds because $p\otimes e_{1,1}$ and $\mathbf{1}_{A}$ are $G$-invariant. This implies that  $\mathrm{Ad}(U)\circ (\phi,\bbm{u})=(\ad{U}\circ\phi,U\bbm{u}_{\bullet}\Gb_{\bullet}(U)^{*})$ is of the form $(\psi\otimes e_{1,1},\bbm{v}\otimes e_{1,1}+\bbm{v}')$ for a unital cocycle embedding $(\phi,\bbm{v}):(A,\Ga)\to(p(B\otimes C(X))p,\Gb\otimes \id{C(X)})$ and a $\Gb\otimes\id{C(X)\otimes \Kbb}$-cocycle $\bbm{v}'$ which takes values in $(\mathbf{1}-p\otimes e_{1,1})+(\mathbf{1}-p\otimes e_{1,1})(B\otimes C(X)\otimes \Kbb)(\mathbf{1}-p\otimes e_{1,1})$. By \autoref{rem:kk-elem} and \autoref{prop:kk-elem}, $(\psi\otimes e_{1,1},\bbm{v}\otimes e_{1,1}+\bbm{v}')$ and $(\psi\otimes e_{1,1},\bbm{v}\otimes e_{1,1}-(\mathbf{1}-e_{1,1}))$ represent the same element in $KK^{G}(\Ga,\Gb\otimes\id{C(X)\otimes\Kbb})$. Thus $KK^{G}(\ad{U}\circ (\phi,\bbm{u}))=KK^{G}(\psi,\bbm{v}+(\mathbf{1}-p))\otimes \Gk$ holds. We claim that $KK^{G}(\ad{U}\circ (\phi,\bbm{u}))$ is equal to $KK^{G}(\phi,\bbm{u})$. These elements are represented by Kasparov triples $((B\otimes C(X)\otimes \Kbb)^{\ad{U}\circ\bbm{u}},\ad{U}\circ\phi,0)$ and $((B\otimes C(X)\otimes\Kbb)^{\bbm{u}},\phi,0)$ respectively. Let us consider $U^{*}:(B\otimes C(X)\otimes \Kbb)^{\ad{U}\circ \bbm{u}}\to (B\otimes C(X)\otimes\Kbb)^{\bbm{u}}$. This is a $B$-module map which preserves the $B$-valued inner product. For all $g\in G$, $a\in A$ and $x\in B\otimes C(X)\otimes \Kbb$,
\begin{align*}
&U^{*}(g\cdot b)=U^{*}(U\bbm{u}_{g}(\Gb_{g}\otimes\id{C(X)\otimes \Kbb}))(\Gb_{g}\id{C(X)\otimes \Kbb})(b)=\bbm{u}(\Gb_{g}\otimes\id{C(X)\otimes\Kbb})(U^{*}b),\\
&U^{*}(\ad{U}\circ\phi(a)b)=\phi(a)U^{*}b.
\end{align*}
Thus $U^{*}$ is an isomorphism of Kasparov triples, and we get $KK^{\ad{U}\circ(\phi,\bbm{u})}=KK^{G}(\phi,\bbm{u})$. By using what we showed above, we conclude that 
\[KK^{G}(\psi,\bbm{v}+(\mathbf{1}-p))=KK^{G}(\ad{U}\circ(\phi,\bbm{u}))\otimes \Gk^{-1}=KK^{G}(\phi,\bbm{u})\otimes \Gk^{-1}=x.\]
\end{proof}
We can take the cocycle to be $\mathbf{1}$, but we will discuss this after the uniqueness theorem.
\subsubsection*{Uniqueness theorem}

\begin{theorem}[cf. {\cite[Theorem 5.7]{GS2}}]\label{thm:unique-stable}
  Let $A$ and $B$ be a Kirchberg algebra and $X$ be a compact connected metrizable space. Let $\Ga:G\act A$ be an action, and $\Gb:G\act B$ be a strongly stable and isometrically shift-absorbing action. Let $(\phi,\bbm{u}),\;(\psi,\bbm{v}):(A,\Ga)\to (B\otimes C(X),\Gb\otimes\id{C(X)})$ be two proper cocycle embeddings that form an anchored Cuntz pair. Then $KK^{G}(\phi,\bbm{u})=KK^{G}(\psi,\bbm{v})$ if and only if $(\phi,\bbm{u})$ and $(\psi,\bbm{v})$ are strongly asymptotically unitarily equivalent.
\end{theorem}
\begin{proof}
"if" part is clear.

By \autoref{prop:oinf-ab}, we have $\Gb\cong_{cc}\Gb\otimes\id{\oinf}$. It follows from \cite[Theorem 5.6]{SZ3} that there exists a proper cocycle conjugacy $(\Gk,\bbm{x}):(B\otimes C(X),\Gb\otimes\id{C(X)})\to(B\otimes C(X)\otimes\oinf,\Gb\otimes\id{C(X)}\otimes\id{\oinf})$ that is strongly asymptotically unitarily equivalent to $\id{B\otimes C(X)}\otimes \mathbf{1}_{\oinf}$. Thus the proper cocycle morphism $(\Gk,\bbm{x})^{-1}\circ(\id{B\otimes C(X)}\otimes \mathbf{1}_{\oinf},\mathbf{1})$ is strongly asymptotically inner, that is, it is strongly asymptotically unitarily equivalent to $(\id{B\otimes C(X)\otimes \oinf},\mathbf{1})$. Hence it suffices to show that the two proper cocycle morphisms
\[(\phi\otimes 1_{\oinf},\bbm{u}\otimes 1_{\oinf}),(\psi\otimes 1_{\oinf},\bbm{v}\otimes 1_{\oinf}):(A,\Ga)\to(B\otimes C(X)\otimes\oinf,\Gb\otimes \id{C(X)}\otimes\id{\oinf})\]
are strongly asymptotically unitarily equivalent. To simplify the notation, we set $C \coloneqq B\otimes \oinf$ and $\Gg \coloneqq \Gb\otimes\id{\oinf}$. We also write $\phi',\; \psi',\; \bbm{u}',$ and $\bbm{v}'$ for the amplifications $\phi\otimes \mathbf{1}_{\oinf},\; \psi\otimes\mathbf{1}_{\oinf},\; \bbm{u}\otimes\mathbf{1}_{\oinf},$ and $\bbm{v}\otimes\mathbf{1}_{\oinf}$, respectively. 

We have a unital inclusion $\oinf\hookrightarrow\Mcal(C\otimes C(X))^{\Gg\otimes\id{C(X)}}$ that commutes pointwise with the ranges of the maps $\phi',\;\psi',\bbm{u}',$ and $\bbm{v}'$. By \autoref{lem:commuting}, there exists a proper cocycle embedding $(\theta,\bbm{y}):(A,\Ga)\to(C\otimes C(X),\Gg\otimes \id{C(X)})$ and a unital inclusion $\Gi_{0}:\otw\to \Mcal(C\otimes C(X))^{\Gg\otimes \id{C(X)}}$ whose range commutes with the ranges of $\theta$ and $\bbm{y}$.

By \autoref{lem:comm-range}, it follows that $(\phi',\bbm{u}')$ is strongly asymptotically unitarily equivalent to $(\phi'\oplus\theta,\bbm{u}'\oplus\bbm{y})$ and that $(\psi',\bbm{v}')$ is strongly asymptotically unitarily equivalent to $(\psi'\oplus\theta,\bbm{v}'\oplus\bbm{y})$. In particular it suffices to show that the two proper cocycle morphisms
\[(\phi'\oplus\theta,\bbm{u}'\oplus\bbm{y}),(\psi'\oplus\theta,\bbm{v}'\oplus\bbm{y}):(A,\Ga)\to(C\otimes C(X),\Gb\otimes \id{C(X)})\]
are strongly asymptotically unitarily equivalent.

We take a sequence of isometries $r_{n}\in\Mcal(C\otimes C(X))^{\Gg\otimes\id{C(X)}}$ such that $\sum_{n=1}^{\infty}r_{n}r_{n}^{*}=\mathbf{1}$ holds in the strict topology. We use this sequence when we take infinite repeats. Let $r_{\infty}=\sum_{n=1}^{\infty}r_{n+1}r_{n}^{*}$. This isometry satisfies the equation $r_{1}r_{1}^{*}+r_{\infty}r_{\infty}^{*}=\mathbf{1}$. We use these isometries when we take Cuntz sum. Let $\{e_{k,l}\:|\:k,l\geq 1\}$ be a set of matrix units generating $\Kbb$.

By \autoref{prop:absorbing}, the infinite repeat $(\theta^{\infty},\bbm{y}^{\infty})$ is an absorbing cocycle representation. We define a non-degenerate embedding to be
\[\Gi:\otw\otimes\Kbb\to\Mcal(C\otimes C(X))^{\Gg\otimes\id{\Kbb}},\quad\Gi(a\otimes e_{k,l})=r_{k}\Gi_{0}(a)r_{l}.\]
By \autoref{lem:unitary-path}, We can take a continuous path $u:[0,\infty)\to\Ucal(\mathbf{1}+r_{\infty}\Gi(\otw\otimes \Kbb)r_{\infty}^{*})\subset \Ucal(\Mcal(C\otimes C(X))^{\Gg\otimes\id{C(X)}})$ such that 
\[r_{\infty}r_{\infty}^{*}=\lim_{t\to\infty}u_{t}^{*}r_{\infty}\Gi(\mathbf{1}\otimes e_{1,1})r_{\infty}^{*}u_{t}=\lim_{t\to\infty}u_{t}^{*}r_{2}r_{2}^{*}u_{t}\]
in the strict topology. Since the range of $\Gi$ commutes with the ranges of $\theta^{\infty}$ and $\bbm{y}^{\infty}$, the range of $u$ commutes with the ranges of $0\oplus_{r_{1}.r_{\infty}}\theta^{\infty}$ and $0\oplus_{r_{1},r_{\infty}}\bbm{y}^{\infty}$. Because $u_{t}r_{1}=r_{1}u_{t}=r_{1}$ holds for all $t\in [0,\infty)$, $\lim_{t\to\infty}u_{t}^{*}(r_{1}r_{1}^{*}+r_{2}r_{2}^{*})u_{t}=\mathbf{1}$ holds in the strict topology. Since $(\phi,\bbm{u})$ and $(\psi,\bbm{v})$ form an anchored Cuntz pair and $KK^{G}(\phi,\bbm{u})=KK^{G}(\psi,\bbm{v})$ holds,  $(\phi',\bbm{u}')$ and $(\psi,\bbm{v}')$ forms an anchored Cuntz pair and 
\[KK^{G}(\phi',\bbm{u})=KK^{G}(\phi,\bbm{u})\otimes KK^{G}(\Gk,\bbm{x})=KK^{G}(\psi,\bbm{v})\otimes KK^{G}(\Gk,\bbm{x})=KK^{G}(\psi',\bbm{v}')\] holds. Then we have
\[0=KK^{G}(\phi',\bbm{u}')-KK^{G}(\psi',\bbm{v}')=[(\phi',\bbm{u}'),(\psi',\bbm{v}')]_{h}\in KK^{G}(\Ga,\Gg\otimes \id{C(X)})\]
By the stable uniqueness theorem \cite[Theorem 5.4]{GS1}, it follows that $(\phi',\bbm{u}')\oplus(\theta^{\infty},\bbm{y}^{\infty})$ and $(\psi',\bbm{v}')\oplus(\theta^{\infty},\bbm{y}^{\infty})$ are strongly asymptotically unitarily equivalent, that is, there exists a norm-continuous path $w:[0,\infty)\to\Ucal(\mathbf{1}+C\otimes C(X))$ such that
\begin{align*}
&\bullet\quad u_{0}=\mathbf{1},\\
&\bullet\quad \lim_{t\to\infty}w_{t}(\phi'(a)\oplus\theta^{\infty}(a))w_{t}^{*}=\psi'(a)\oplus\theta^{\infty}(a)\quad\text{for all }a\in A,\\
&\bullet\quad\lim_{t\to\infty}\max_{g\in G}\|w_{t}(\bbm{u'}_{g}\oplus\bbm{y}^{\infty}_{g})(\Gg_{g}\otimes\id{C(X)})(w_{t})^{*}-\bbm{v}'_{g}\oplus\bbm{y}^{\infty}_{g}\|=0.
\end{align*}
To simplify notation, we write $p$ instead of the projection $r_{1}r_{1}^{*}+r_{2}r_{2}^{*}$. By cutting off an initial segment of $u$ and reparameterizing $u$, we can ensure that $u:[0,\infty)\to \Ucal(\mathbf{1}+r_{\infty}\Gi(\otw\otimes\Kbb)r_{\infty}^{*})$ additionally satisfies 
\[\lim_{t\to\infty}\|(u_{t}^{*}pu_{t}-\mathbf{1})(w_{t}-\mathbf{1})\|=\lim_{t\to\infty}\|(p-\mathbf{1})(u_{t}w_{t}u_{t}^{*}-\mathbf{1})\|=0\]
and
\[\|(u_{t}^{*}pu_{t}-\mathbf{1})(w_{t}-\mathbf{1})\|<\frac{1}{4}\quad \text{for all }t\in[0,\infty)\]
Let $z':[0,\infty)\to\mathbf{1}+p(C\otimes C(X))p$ be defined by $z'_{t}=pu_{t}w_{t}u_{t}^{*}p+(\mathbf{1}-p)$. This is a norm-continuous path satisfying $z'_{0}=\one$ and $\|u_{t}w_{t}u_{t}^{*}-z'_{t}\|<1$ for all $t\geq 0$. Since $u_{t}w_{t}u_{t}^{*}$ is invertible for all $t$, so is $z'_{t}$. Thus the unitary part of its polar decomposition defines a continuous path $z:[0,\infty)\to\Ucal(\mathbf{1}+p(C\otimes C(X))p)$. It follows from
\begin{align*}
\lim_{t\to\infty}\|u_{t}w_{t}u_{t}^{*}-z'_{t}\|&=\lim_{t\to\infty}\|u_{t}w_{t}u_{t}^{*}-pu_{t}w_{t}u_{t}^{*}p-(\mathbf{1}-p)\|\\
&=\lim_{t\to\infty}\|(\mathbf{1}-p)(u_{t}w_{t}u_{t}^{*}-\mathbf{1})+p(u_{t}w_{t}u_{t}^{*}-1)(\mathbf{1}-p)\|\\
&\leq \lim_{t\to\infty}2\|(\mathbf{1}-p)(u_{t}w_{t}u_{t}^{*}-\mathbf{1})\|=0
\end{align*}
that $\lim_{t\to\infty}\|u_{t}w_{t}u_{t}^{*}-z_{t}\|=0$. Then for all $a\in A$, we have
\begin{align*}
\lim_{t\to\infty}z_{t}(\phi'(a)\oplus\theta^{\infty}(a))z_{t}^{*}&=\lim_{t\to\infty}u_{t}w_{t}u_{t}(\phi'(a)\oplus\theta^{\infty}(a))u_{t}w_{t}^{*}u_{t}^{*}\\
&=\lim_{t\to\infty}u_{t}w_{t}(\phi'(a)\oplus\theta^{\infty}(a))w_{t}^{*}u_{t}^{*}\\
&=\lim_{t\to\infty}u_{t}(\psi'(a)\oplus\theta^{\infty}(a))u_{t}^{*}\\
&=\phi'(a)\oplus\theta^{\infty}(a)
\end{align*}
and
\begin{align*}
&\lim_{t\to\infty}\max_{g\in G}\|z_{t}(\bbm{u}'_{g}\oplus\bbm{y}^{\infty}_{g})(\Gg_{g}\otimes\id{C(X)})(z_{t})^{*}-\bbm{v}'_{g}\oplus\bbm{y}^{\infty}_{g}\|\\
&=\lim_{t\to\infty}\max_{g\in G}\|u_{t}w_{t}u_{t}^{*}(\bbm{u}'_{g}\oplus\bbm{y}^{\infty}_{g})u_{t}(\Gg_{g}\otimes\id{C(X)})(w_{t})^{*}u_{t}^{*}-\bbm{v}'_{g}\oplus\bbm{y}^{\infty}_{g}\|\\
&=\lim_{t\to\infty}\max_{g\in G}\|u_{t}w_{t}(\bbm{u}'_{g}\oplus\bbm{y}^{\infty}_{g})(\Gg_{g}\otimes\id{C(X)})(w_{t})^{*}u_{t}^{*}-\bbm{v}'_{g}\oplus\bbm{y}^{\infty}_{g}\|\\
&=\lim_{t\to\infty}\max_{g\in G}\|u_{t}(\bbm{v}'_{g}\oplus\bbm{y}^{\infty}_{g})u_{t}^{*}-\bbm{v}'_{g}\oplus\bbm{y}^{\infty}_{g}\|\\
&=\|\bbm{v}'_{g}\oplus\bbm{y}^{\infty}_{g}-\bbm{v}'_{g}\oplus\bbm{y}^{\infty}_{g}\|=0.
\end{align*}
By compressing by $p$, it follows that
\begin{align*}
&\lim_{t\to\infty}z_{t}(r_{1}\phi'(a)r_{1}^{*}+r_{2}\theta^{\infty}r_{2}^{*})z_{t}^{*}=r_{1}\psi'(a)r_{1}+r_{2}\theta(a)r_{2}\\
&\lim_{t\to\infty}\max_{g\in K}\|z_{t}(r_{1}\bbm{u}'_{g}r_{1}^{*}+r_{2}\bbm{y}_{g}r_{2}^{*})z_{t}-(r_{1}\bbm{v}'_{g}r_{1}^{*}+r_{2}\bbm{y}_{g}r_{2}^{*})\|=0
\end{align*}
Let $S$ be the isometry defined by $S=r_{1}r_{1}^{*}+r_{2}r_{\infty}^{*}\in\Mcal(C\otimes C(X))^{\Gg\otimes\id{C(X)}}$ and define a unitary path $v:[0,\infty)\to\Ucal(\mathbf{1}+C\otimes C(X))$ by $v_{t}=S^{*}z_{t}S$. Note that $SS^{*}=p$, $Sr_{1}=r_{1}$, $Sr_{\infty}=r_{2}$ and $v_{0}=1$ holds. By further conjugating by $S$, we obtain
\begin{align*}
&\lim_{t\to\infty}v_{t}(\phi'(a)\oplus\theta(a))v_{t}^{*}=\lim_{t\to\infty}\psi'(a)\oplus\theta(a)\\
&=\lim_{t\to\infty}\max_{g\in K}\|v_{t}(\bbm{u}'_{g}\oplus\bbm{y}_{g})v_{t}-\bbm{v}'_{g}\oplus\bbm{y}_{g}\|=0
\end{align*}
for all $a\in A$. Therefore $(\phi'\oplus\theta,\bbm{u}'\oplus\bbm{y})$ and $(\psi'\oplus\theta,\bbm{v}'\oplus\bbm{y})$ are strongly asymptotically unitarily equivalent. Hence, we obtain
\begin{align*}
(\phi,\bbm{u})&\sim_{\substack{\mathrm{strong}\\ \mathrm{asymp}}}(\Gk,\bbm{x})^{-1}\circ (\id{B}\otimes \mathbf{1}_{C(X)\otimes \oinf},\mathbf{1})\circ (\phi,\bbm{u})\\
&\sasymp(\Gk,\bbm{x})^{-1}\circ(\phi',\bbm{u}')\\
&\sasymp(\Gk,\bbm{x})^{-1}\circ (\phi'\oplus\theta,\bbm{u}'\oplus\bbm{y})\\
&\sasymp(\Gk,\bbm{x})^{-1}\circ (\psi'\oplus\theta,\bbm{v}'\oplus\bbm{y})\\
&\sasymp(\Gk,\bbm{x}^{-1})\circ (\psi',\bbm{v}')\\
&\sasymp(\Gk,\bbm{x}^{-1})\circ (\id{B}\otimes\mathbf{1}_{C(X)\otimes \oinf},\mathbf{1})\circ (\psi,\bbm{v})\\
&\sasymp(\psi,\bbm{v})
\end{align*}
and the proof is complete.
\end{proof}

\begin{theorem}[cf. {\cite[Theorem 5.8]{GS2}}]\label{thm:unique-unital}
Let $G$ be a compact group. Let $A$ and $B$ be unital Kirchberg algebras and $X$ be a compact connected metrizable space. Let $\Ga:G\act A$ be an action and $\Gb:G\act B$ be an isometrically shift-absorbing action. Let $(\phi,\bbm{u})$ and $(\psi,\bbm{v})$ be two proper cocycle embeddings from $(A,\Ga)$ to $(B\otimes C(X),\Gb\otimes\id{C(X)})$ such that $\phi(1_{A})=\psi(1_{A})=p\in (B\otimes C(X))^{\Gb\otimes \id{C(X)}}$, and suppose that the cocycles $\bbm{u}$ and $\bbm{v}$ take values in $\Ucal(\mathbf{1}+p(B\otimes C(X))p)$. Then $KK^{G}(\phi,\bbm{u})=KK^{G}(\psi,\bbm{v})$ if and only if $(\phi,\bbm{u})$ and $(\psi,\bbm{v})$ are asymptotically unitarily equivalent.
\end{theorem}
\begin{proof}
"if" part is clear.

Let us consider the canonical inclusion $\Gi:B\otimes C(X)\to B\otimes C(X)\otimes \Kbb$ via $\Gi(x)=x\otimes e_{1,1}$. Consider the $\Gb\otimes\id{C(X)\otimes\Kbb}$-cocycles $\bbm{u}'_{g}=\Gi(\bbm{u})+(\mathbf{1}-e_{1,1})$ and $\bbm{v}'_{g}=\Gi(\bbm{v})+(\mathbf{1}-e_{1,1})$. Applying \autoref{thm:ex-stable} to $A=0$, we obtain a norm-continuous $\Gb\otimes\id{C(X)\otimes\Kbb}$-cocycle $\bbm{x}:G\to \Ucal(1+B\otimes C(X)\otimes \Kbb)$ such that $[(\bbm{u}',\bbm{v}')]_{h}=[(\bbm{x},\one)]_{h}$.

Take isometries $r_{1},r_{2}\in \Mcal(\Kbb)\subset\Mcal(B\otimes C(X)\otimes \Kbb)^{\Gb\otimes\id{C(X)\otimes\Kbb}}$ satisfying $r_{1}r_{1}^{*}+r_{2}r_{2}^{*}=\mathbf{1}$ and $r_{1}e_{1,1}=e_{1,1}$. Then, we obtain $\bbm{u}'=\bbm{u}'\oplus_{r_{1},r_{2}}\one$ and
\[KK^{G}(\Gi\circ \phi,\bbm{u}')=KK^{G}(\Gi\circ\psi,\bbm{v}'\oplus_{r_{1},r_{2}}\bbm{x})\]
in $KK^{G}(\Ga,\Gb\otimes\id{C(X)}\otimes\id{\Kbb})$. By construction, we have
\[[(\bbm{u}',\bbm{v}\oplus\bbm{x})]_{h}=[(\bbm{u}'\oplus_{r_{1},r_{2}}\one,\bbm{v}'\oplus_{r_{1},r_{2}}\bbm{x})]_{h}=[(\bbm{u}',\bbm{v}')]_{h}-[(\bbm{x},\one)]_{h}=0\]
Then by \autoref{thm:unique-stable}, $(\Gi\circ\phi,\bbm{u}')$ and $(\Gi\circ \psi,\bbm{v}'\oplus_{r_{1},r_{2}}\bbm{x})$ are strongly asymptotically unitarily equivalent, and hence there exists a norm-continuous path $v:[0,\infty)\to\Ucal(\mathbf{1}+B\otimes C(X)\otimes \Kbb)$ such that
\begin{align*}
&\Gi\circ \psi(a)=\lim_{t\to\infty}v_{t}(\Gi\circ\phi)(a)v_{t}^{*}\text{ for all }a\in A\\
&\Gi\circ \psi(a)-v_{t}(\Gi\circ\phi)(a)v_{t}^{*}\in B\otimes C(X)\otimes \Kbb\text{ for all }a\in A\text{ and }t\geq 0\\
&\lim_{t\to\infty}\max_{g\in G}\|\bbm{v}'_{g}\oplus_{r_{1},r_{2}}\bbm{x}_{g}-v_{t}\bbm{u}'_{g}(\Gb_{g}\otimes\id{C(X)\otimes\Kbb})(v_{t})^{*}\|=0\\
&\bbm{v}'_{g}\oplus_{r_{1},r_{2}}\bbm{x}_{g}-v_{t}\bbm{u}'_{g}(\Gb_{g}\otimes\id{C(X)\otimes\Kbb})(v_{t})^{*}\in B\otimes C(X)\otimes \Kbb\text{ for all }g\in G\text{ and }t\geq 0
\end{align*}
By the first equality, as $t\to\infty$, $v_{t}$ approximately commutes with $p\otimes e_{1,1}$. Thus, after cutting away an initial segment of  $v$ if necessary, we have a norm-continuous path of unitaries $u:[0,\infty)\to \Ucal(B\otimes C(X))$ defined by
\[u_{t}=\Gi^{-1}((p\otimes e_{1,1})v_{t}(p\otimes e_{1,1})\left|(p\otimes e_{1,1})v_{t}(p\otimes e_{1,1})\right|^{-1})+(\mathbf{1}-p)  \]
Here we identify $\Gi$ with the isomorphism from $p(B\otimes C(X))p$ to $(p\otimes e_{1,1})(B\otimes C(X)\otimes\Kbb)(p\otimes e_{1,1})$.
We show that $u$ implements the asymptotic unitary equivalence between $(\phi,\bbm{u})$ and $(\psi,\bbm{v})$.
For all $a\in A$, 
\begin{align*}
&\lim_{t\to\infty}\|u_{t}\phi(a)u_{t}^{*}-\psi(a)\|\\
&=\lim_{t\to\infty}\|\Gi^{-1}((p\otimes e_{1,1})v_{t}(p\otimes e_{1,1})|(p\otimes e_{1,1})v_{t}(p\otimes e_{1,1})|^{-1}(\Gi\circ \phi(a))\\
&\qquad\qquad\times|(p\otimes e_{1,1})v_{t}(p\otimes e_{1,1})|^{-1}(p\otimes e_{1,1})v_{t}^{*}(p\otimes e_{1,1})-(\Gi\circ \psi)(a))\|\\
&=\lim_{t\to\infty}\|\Gi^{-1}((p\otimes e_{1,1})(v_{t}(\Gi\circ\phi)(a)v_{t}-(\Gi\circ\psi)(a))(p\otimes e_{1,1}))\|\\
&=0
\end{align*}
and
\begin{align*}
&\lim_{t\to\infty}\max_{g\in G}\|u_{t}\bbm{u}'_{g}(\Gb_{g}\otimes\id{C(X)\otimes\Kbb})(u_{t})^{*}-\bbm{v}_{g}\|\\
&=\lim_{t\to\infty}\max_{g\in G}\|\Gi^{-1}((p\otimes e_{1,1})v_{t}(p\otimes e_{1,1})|(p\otimes e_{1,1})v_{t}(p\otimes e_{1,1})|^{-1}(p\otimes e_{1,1})\bbm{u}'_{g}(p\otimes e_{1,1})\\
&\qquad\qquad\qquad\times |(p\otimes e_{1,1})v_{t}(p\otimes e_{1,1})|^{-1}(p\otimes e_{1,1})v_{t}^{*}(p\otimes e_{1,1})-(p\otimes e_{1,1})\bbm{v}'(p\otimes e_{1,1}))\|\\
&=\lim_{t\to\infty}\max_{g\in G}\|\Gi^{-1}((p\otimes e_{1,1})v_{t}\bbm{u}'_{g}(\Gb_{g}\otimes\id{C(X)\otimes \Kbb})(v_{t})^{*}(p\otimes e_{1,1})\\
&\qquad\qquad\qquad\qquad\qquad-(p\otimes e_{1,1})(\bbm{v}'_{g}\oplus_{r_{1},r_{2}}\bbm{x}_{g})(p\otimes e_{1,1}))\|\\
&=\lim_{t\to\infty}\max_{g\in G}\|(p\otimes e_{1,1})(v_{t}\bbm{u}'_{g}(\Gb_{g}\otimes\id{C(X)\otimes\Kbb}(v_{t})^{*}-\bbm{v}'_{g}\oplus_{r_{1},r_{2}}\bbm{x}_{g})(p\otimes e_{1,1})\|\\
&=0
\end{align*}
hold. The other conditions for asymptotic unitary equivalence are clear.
\end{proof}
\begin{corollary}\label{cor:compact-ex}
Suppose that $G$ is compact. Then for any $z\in KK^{G}(\Ga,\Gb\otimes \id{C(X)})$ with $[1_{A}]\otimes z=[p]$ in $K_{0}((B\otimes C(X))^{\Gb\otimes\id{C(X)}})$ for a nonzero projection $p\in B\otimes C(X)^{\Gb\otimes\id{C(X)}}$, there exists an equivariant $*$-homomorphism $\phi:A\to B\otimes C(X)$ such that $KK^{G}(\phi)=z$ and $\phi(1_{A})=p$. 
\end{corollary}
\begin{proof}
By the existence theorem, there is a cocycle embedding $(\psi,\bbm{v}):(A,\Ga)\to (B\otimes C(X),\Gb\otimes\id{C(X)})$ such that $KK^{G}(\psi,\bbm{v})=z$ and $\psi(\one_{A})=p$. We will show that $\bbm{v}$ is a coboundary.
\[KK^{G}((\Gl\mapsto \Gl p),\bbm{v})=[\one_{A}]\otimes KK^{G}(\psi,\bbm{v})=[1_{A}]\otimes z=KK^{G}((\Gl\mapsto \Gl p),\one)\]
By the uniqueness theorem, there is a continuous map $u:[0,\infty)\to \Ucal(B\otimes C(X))$ such that
\[\lim_{t\to\infty}\max_{g\in G}\|u_{t}\bbm{v}_{g}\Gb_{g}(u_{t}^{*})-1\|=0\]
This implies $\bbm{v}$ is an asymptotic coboundary, and hence $\bbm{v}$ is a coboundary because $G$ is compact. Then there exists $u\in \Ucal(C(X)\otimes B)$ such that $\bbm{v}_{g}=u\Gb_{g}(u^{*})$. Then $KK^{G}(\psi,\bbm{v})=KK^{G}(\ad{u^{*}}\circ \psi,\one)$ holds. Since $\ad{u^{*}}\circ \psi$ is equivariant, $u^{*}pu$ is in $(B\otimes C(X))^{\Gb\otimes\id{C(X)}}$. If $p=\one$, then $u^{*}pu=p$. If $p\neq \one$, then $p,\: \one-p,\: u^{*}pu,\:\one-u^{*}pu $ are full properly infinite projections. Because $[p]=[\one_{A}]\otimes KK^{G}(\psi,\bbm{v})=[\one_{A}]\otimes KK^{G}(\ad{u^{*}}\circ\psi,\one)=[u^{*}pu]$ in $K_{0}((B\otimes C(X))^{\Gb\otimes \id{C(X)}})$, $p$ and $u^{*}pu$ are unitarily equivalent in $(B\otimes C(X))^{\Gb\otimes\id{C(X)}}$. Thus there exists an equivariant $*$-homomorphism $\phi:A\to C(X)\otimes B$ such that $KK^{G}(\phi)=z$ and $\phi(\one_{A})=p$. 
\end{proof}

\section{Homotopy groups}
\subsection{Equivariant version of Dadarlat's theorem}
In \cite{D1}, the following bijection was established.
\begin{theorem}[{\cite[Theorem 4.6, 6.2]{D1}}]
Let $A$ be a unital Kirchberg algebra and let $X$ be a compact metrizable space.
Then there are bijections
\[[X,\mathrm{End}(A)]\to KK(\cna,SC(X)\otimes A)\]
and
\[[X,\mathrm{End}(A\otimes \Kbb)]\to KK(A,C(X)\otimes A),\]
where $\cna$ is the mapping cone of the unital inclusion $\nu_{A}:\Cbb\to A$. If moreover $(A,X)$ is a $KK$-continuous pair (\cite[Definition 5.3]{D1}), then there exists a bijection
\[[X,\mathrm{Aut}(A\otimes \Kbb)]\to KK(A,C(X)\otimes A)^{*}.\]
Here, $KK(A,C(X)\otimes A)^{*}=\{z\in (A,C(X)\otimes A)\:|\: (\mathrm{ev}_{x})_{*}(z)\in KK(A,A)\text{ is invertible for all }x\in X\}.$
\end{theorem}
In this section, we consider the equivariant analogue of this result.
\subsubsection{The homotopy set $[X,\mathrm{End}_{G}(A)]$}
\begin{notation}
For $G$-\cstar-algebra $(A,G,\Ga)$, we will denote by $\End{A}$ (resp. $\aut{A}$) the set of all equivariant (unital if $A$ is unital) $*$-endomorphisms (resp. $*$-automorphisms) on $A$ and equip it with the point-norm topology.

For topological spaces $X$ and $Y$, we will write $[X,Y]$ for the homotopy equivalence classes of continuous maps from $X$ to $Y$ and denote $[\phi]$ by the equivalence class of $\phi:X\to Y$.

For $\phi:X\to \End{A}$, we use the same letter $\phi$ for the $*$-homomorphism given by $A\ni a\mapsto (x\mapsto \phi_{x}(a))\in C(X)\otimes A$.
\end{notation}
\subsubsection*{Stable case}
\begin{theorem}\label{thm:stable-end}
Let $G$ be a compact group and $X$ a compact connected metrizable space. Let $A$ be a stable Kirchberg algebra and $\Ga:G\act A$ an isometrically shift-absorbing action. Then, the map from $[X,\End{A}]$ to $KK^{G}(A,C(X)\otimes A)$ given by $[\phi]\to KK^{G}(\phi,\mathbf{1})$ is a bijection.
\end{theorem}
\begin{proofbreak}
  To simplify notation, we write $KK^{G}(\phi)$ instead of $KK^{G}(\phi,\mathbf{1})$.

We begin by checking that the map is well-defined. If $[\phi]=[\psi]$, then $\phi$ and $\psi$ are homotopic. Thus there exists a homotopy between two Kasparov triples $(C(X)\otimes A,\phi,0)$ and $(C(X)\otimes A,\psi,0)$, and we obtain $KK^{G}(\phi)=KK^{G}(\psi)$.

Next, we prove injectivity of the map. Suppose $KK^{G}(\phi)=KK^{G}(\psi)$ for $[\phi],[\psi]\in [X,\End{A}]$. By \autoref{thm:unique-stable}, there is a path of unitaries $U:[0,\infty)\to \Ucal(\one+C(X)\otimes A)$ such that
\[\phi(a)=\lim_{t\to\infty}U_{t}\psi(a)U_{t}^{*}\text{  and  } \lim_{t\to\infty}\max_{g\in G}\|U_{t}-(\id{C(X)}\otimes \Ga_{g})(U_{t})\|=0.\]Since $G$ is compact, by cutting initial segments if necessary, we may assume that $U_{t}$ is $\id{C(X)}\otimes \Ga$-invariant for all $t$. By \cite[Proposition 6.3]{SZ2}, $\id{C(X)}\otimes \Ga$ is strongly stable. Hence, by \cite[Proposition 1.9]{SZ1}, we can take strictly continuous maps of isometries $S_{1}:[0,1)\to M(C(X)\otimes A)^{\id{C(X)}\otimes \Ga}$ and  $S_{2}:(0,1]\to M(C(X)\otimes A)^{\id{C(X)}\otimes \Ga}$ satisfying $S_{1}^{(0)}=\one$, $S_{2}^{(1)}=\one,$ and $S_{1
}^{(t)}S_{1}^{(t)*}+S_{2}^{(t)}S_{2}^{(t)*}=\one$ for all $t\in (0,1)$. Define $H:[0,1]\times C(X)\otimes A\to C(X)\otimes A$ by
\begin{align*}
H(t,a)=\begin{cases}
U_{0}\phi(a)U_{0}^{*} & t=0\\
\ad{(S_{1}^{(t)}U_{0}S_{1}^{(t)*}+S_{2}^{(t)}S_{2}^{(t)*})}\phi(a) & 0<t<1\\
\phi(a) & t=1
\end{cases}
\end{align*}
We verify the continuity by considering the cases $t=0$, $t=1$, and $0 < t < 1$ separately. The following claim is clear.
\begin{claim*}
Let $B$ be a C$^{\ast}$-algebra. Let $(z_{i})_{i}$ be a bounded net in $M(B)$ converging strictly to $z\in M(B)$ and $(b_{n})_{n}$ be a bounded net in $B$ converging to $b\in B$ in norm. Then $z_{i}b_{n}$ converges to $zb$ in norm.
\end{claim*}
First, we consider the case $t=0$. Let $t_{n}$ be a sequence converging to $0$ and let $b_{n}\in C(X)\otimes A$ be a sequence converging to $b\in C(X)\otimes A$. Since $S_{1}^{(t_{n})}S_{1}^{(t_{n})*}+S_{2}^{(t_{n})}S_{2}^{((t_{n}))*}=\one$ and $S_{1}^{(t_{n})}\to \one$ strictly as $t\to 0$, $S_{2}^{(t_{n})}S_{2}^{(t_{n})*}$ are bounded for all $n$ and converges strictly to $0$ as $t\to 0$. Thus
\begin{align*}
\|H(0,b)-H(t_{n},b_{n})\|&=\|U_{0}\phi(b)U_{0}^{*}-\mathrm{Ad}(S_{1}^{(t_{n})}U_{0}S_{1}^{(t_{n})*}+S_{2}^{(t_{n})}S_{2}^{(t_{n})*})(\phi(b_{i}))\|\to 0\quad(t\to 0)
\end{align*}
Next, we turn to the case $t=1$. It suffices to show that $S_{1}^{(t)}U_{0}S_{1}^{(t)*}$ converges strictly to $0$ as $t\to 1$. (The proof is the same as in the case $t=0$.) Since $S_{1}^{(t)}S_{1}^{(t)*}$ converges strictly to $0$  as $t\to 1$, for any $b\in C(X)\otimes A$ it holds that
\begin{align*}
\|S_{1}^{(t)}U_{0}S_{1}^{(t)*}b\|\leq \|S_{1}^{(t)*}b\|=\|S_{1}^{(t)*}S_{1}^{(t)}S_{1}^{(t)*}b\|\leq \|S_{1}^{(t)}S_{1}^{(t)*}b\|\to 0\quad(t\to 1)
\end{align*}
For the remaining case $t \in (0,1)$, the continuity follows from the claim. Thus $\phi$ and $\psi$ are homotopic, and this implies injectivity of the map.

It remains to show surjectivity of the map. Let $z\in KK^{G}(A,C(X)\otimes A)$ be any element. By \autoref{thm:ex-stable}, there exists an anchored proper cocycle embedding $(\phi,\bbm{u})$ such that $KK^{G}(\phi,\bbm{u})=z$. Consider the following commutative diagram in $KK^{G}$
\begin{center}
\begin{tikzpicture}[auto]
\node(1) at (-3,1.5) {$A$};
\node(2) at (7,1.5) {$C(X)\otimes A$};
\node (3) at (-3,-1.5) {$A\otimes \Kbb(L^{2}(G))$};
\node(4) at (7,-1.5) {$C(X)\otimes A\otimes \Kbb(L^{2}(G))$};
\draw[->] (1) to node {$z=KK^{G}(\phi,\bbm{u})$} (2);
\draw[->] (3) to node {$KK^{G}(\phi\otimes \id{\Kbb(L^{2}(G))},\bbm{u}\otimes 1)$} (4);
\draw[->] (1) to node {$KK^{G}(\id{A})\hat{\otimes} z'$} (3);
\draw[->] (2) to node {$KK^{G}(\id{C(X)\otimes A})\hat{\otimes}z'$} (4);
\end{tikzpicture}
\end{center}
where $z'$ is the $KK^{G}$-equivalence given by a minimal invariant projection in $(\Kbb(L^{2}(G)),\ad{\rho})$. Define a unitary $W\in \Mcal(C(X)\otimes A\otimes \Kbb(L^{2}(G)))=\BB(L^{2}(G,C(X)\otimes A))$ by
\[W(b\otimes \Gx)(g)=\bbm{u}_{g}b\Gx(g)\text{ for all } b\in C(X)\otimes A,\; \Gx\in C(G)\subset L^{2}(G).\] 
Then we have $W(\id{C(X)}\otimes \Ga_{g}\otimes \ad{\rho_{g}})(W^{*})=\bbm{u}_{g}\otimes \one$. This implies that $\bbm{u}\otimes \one$ is a coboundary. Thus there exists an equivariant $*$-homomorphism $\psi:A\otimes \Kbb(L^{2}(G))\to C(X)\otimes A\otimes \Kbb(L^{2}(G))$ such that $KK^{G}(\psi,\mathbf{1})=KK^{G}(\phi\otimes\id{\Kbb(L^{2}(G))},\bbm{u}\otimes \mathbf{1})$. Since $KK^{G}(\id{A})\hat{\otimes}z'$ is $KK^{G}$-equivalence, it follows from \cite[Corollary 6.4]{GS2} that there exists an equivariant $*$-isomorphism $\theta:A\to A\otimes \Kbb(L^{2}(G))$ such that $KK^{G}(\theta)=KK^{G}(\id{A})\hat{\otimes}z'$. Hence we obtain $KK^{G}((\id{C(X)}\otimes \theta^{-1})\circ\psi\circ \theta)=z$ and surjectivity follows.
\end{proofbreak}
\begin{remark}
We define an addition on $[X, \mathrm{End}_{G}(A)]$ using the pointwise Cuntz sum and a multiplication using pointwise composition. On the other hand, we define a multiplication on $KK^{G}(A, C(X)\otimes A)$ by the composition of the Kasparov product with the restriction to the diagonal map:
\[KK^{G}(A,C(X)\otimes A)\times KK^{G}(A,C(X)\otimes A)\to KK^{G}(A,C(X\times X)\otimes A)\to KK^{G}(A,C(X)\otimes A).\]
The map described in the preceding theorem is, in fact, a ring isomorphism with respect to these structures.
\end{remark}

\subsubsection*{Unital case}\label{subsubsec:unital}
Throughout this subsection, let $G$ be a compact group, $A$ be a unital Kirchberg algebra, and $X$ be a compact connected metrizable space. Let $\Ga: G \act A$ be an isometrically shift-absorbing action. 
To simplify notation, we set $B \coloneqq C(X)\otimes A$ and $\Gb \coloneqq \id{C(X)}\otimes \Ga$.
We denote by $\homg{A}{B}$ the set of all unital equivariant $*$-homomorphisms from $A$ to $B$, and by $[A,B]_{G,u}$ the set of their homotopy classes. Let $j_{A}: A \to B$ denote the canonical inclusion defined by $j_{A}(a) = \one_{C(X)}\otimes a$. 

Let $\nu_{A}: \Cbb \hookrightarrow A$ and $\nu_{B}: \Cbb \hookrightarrow B$ be the unital inclusions. We denote the mapping cones of $\nu_{A}$ and $\nu_{B}$ by $C_{\nu_{A}}$ and $C_{\nu_{B}}$, respectively. Recall that we have the following mapping cone sequence:
\begin{align}
  SA \xlongrightarrow{i} C_{\nu_{A}} \xlongrightarrow{\pi} \Cbb \xlongrightarrow{\nu_{A}} A.\label{eq:mapping-cone-seq}
\end{align}
Let $C_{\nu}j: C_{\nu_{A}} \to C_{\nu_{B}}$ be the map induced by $j_{A}$, defined by $C_{\nu}j = \id{C[0,1]} \otimes j_{A}|_{C_{\nu_{A}}}$. We now define a map $\Gc: [A,B]_{G,u} \to KK^{G}(C_{\nu_{A}}, SB)$ by$$ \Gc([\phi]) = [(C_{\nu}\phi, C_{\nu}j)], $$where $C_{\nu}\phi \coloneqq \id{C[0,1]} \otimes \phi|_{C_{\nu_{A}}}: C_{\nu_{A}} \to C_{\nu_{B}}$. In our discussion, we consider the following actions.

The group $\Ucal(B^{\Gb})/U_{0}(B^{\Gb})$ acts on $[A,B]_{G,u}$ by conjugation:\[ [u] \cdot [\phi] = [\mathrm{Ad}u \circ \phi].\] 
Also, the group $KK^{G}(\Cbb, SB)$ acts on $KK^{G}(C_{\nu_{A}}, SB)$ via the pullback of the projection $\pi: C_{\nu_{A}} \to \Cbb$:
\[g \cdot x = x + \pi^{*}(g). \]
\begin{remark}
Let $A$ be a unital Kirchberg algebra and let $\Ga$ be an isometrically shift-absorbing action of a compact group $G$ on $A$. Then, by \cite[Remark 4.6]{Mu}, the fixed point algebra $A^{\Ga}$ is also a unital Kirchberg algebra. In particular, we have $A^{\Ga}\cong A^{\Ga}\otimes \Ocal_{\infty}$. Since $(C(X)\otimes A)^{\id{C(X)}\otimes \Ga}=C(X)\otimes A^{\Ga}$, it follows that $B^{\Gb}\cong B^{\Gb}\otimes \Ocal_{\infty}$. By \autoref{prop:K1-oinf}, the natural map $\Ucal(B^{\Gb})/U_{0}(B^{\Gb})\to K_{1}(B^{\Gb})$ is an isomorphism.
\end{remark}
Consider the composite map
\[\partial:\Ucal(B^{\Gb})/\Ucal_{0}(B^{\Gb})\to K_{1}(B^{\Gb})\to K^{G}_{0}(SB)\to KK^{G}(\Cbb,SB).\]
For any $u\in \Ucal(B^{\Gb})$, there exists $\Go:[0,1]\to \Ucal(M_{2}(B^{\Gb}))$ such that $\Go_{0}=\left(\begin{smallmatrix} u & 0 \\ 0 & u^{*} \\ \end{smallmatrix}\right)$ and $\Go_{1}=\left(\begin{smallmatrix} \one & 0 \\ 0 & \one \\ \end{smallmatrix}\right)$. Let $p=\Go \left(\begin{smallmatrix} \one &0  \\ 0 & \one \\ \end{smallmatrix}\right)\Go^{*}$ and $p_{0}=\left(\begin{smallmatrix} \one & 0 \\ 0 &\one  \\ \end{smallmatrix}\right)$ be projections in $M_{2}(C[0,1]\otimes B^{\Gb})$. Define $\Gn, \Gn_{0}:\Cbb\to M_{2}(C[0,1]\otimes B^{\Gb})$ by $\Gn(\Gl)=\Gl p$ and $\Gn_{0}(\Gl)=\Gl p_{0}$. Then the pair $(\Gn,\Gn_{0})$ defines an element in $\mathbb{E}^{G}((\Cbb,\mathrm{id}),(SB,\id{S}\otimes \Gb))$. Through the map $\partial$, we identify $\Ucal(B^{\Gb})/\Ucal_{0}(B^{\Gb})$ with $KK^{G}(\Cbb,SB)$.

\begin{proposition}[{\cite[Proposition 3.8]{D1}}]
For any $v\in \Ucal(B^{\Gb})$ and $\phi\in\homg{A}{B}$, we have
\[\Gc([\mathrm{Ad}v\circ \phi])=\Gc([\phi])+\pi^{*}\partial([v]).\]
\end{proposition}
\begin{proof}
\begin{align*}
\Gc([\mathrm{Ad}v\circ \phi])&=[\ad{v}\circ \cna\phi,\cna  j]\\
&=[\ad{v}\circ \cna\phi,\ad{v}\circ \cna j]+[\ad{v}\circ \cna j,\cna j]\\
&=[\cna\phi,\cna j]+[\ad{v}\circ \cna j,\cna j]\\
&=\Gc([\phi])+[\ad{v}\circ \cna j,\cna j]
\end{align*}
So it suffices to show that $\displaystyle [\ad{v}\circ \cna j,\cna j]=\pi^{*}\partial ([v])$. Let $\Go:[0,1]\to \Ucal(M_{2}(B^{\Gb}))$ be a continuous path of unitaries such that $\Go_{0}=\tmat{v}{0}{0}{v^{*}}$ and $\Go_{1}=\tmat{\one}{0}{0}{\one}$.. We define two homotopies of Cuntz pairs $\Gg=(\Gg_{s})$ and $\Gd=(\Gd_{s})$ in $\Ebb^{G}(\cna,SB)$ as follows.
\begin{align*}
&s\mapsto \Gg_{s}(f)=\left[\Go\pmat{J(f(s-))}{0}{0}{0}\Go^{*},\pmat{J(f(s-))}{0}{0}{0}\right]\\
&s\mapsto \Gd_{s}(f)=\left[\Go_{(s-)}\pmat{J(f)}{0}{0}{0}\Go^{*}_{(s-)}, \pmat{J(f)}{0}{0}{0}\right]
\end{align*}
Since all $*$-homomorphisms on the right-hand side are equivariant, we may take the trivial cocycle $1$. Then $\Gg_{1}=\Gd_{1}$ and 
\begin{align*}
\Gg_{0}(f)&=\left[\Go\pmat{J(f)}{0}{0}{0}\Go^{*}, \pmat{J(f)}{0}{0}{0}\right]=\pi^{*}\partial([v])(f),\\
\Gd_{0}(f)&=\left[\pmat{vJ(f)v^{*}}{0}{0}{0},\pmat{J(f)}{0}{0}{0}\right]=\left[\ad{v}\circ J(f),J(f)\right]
\end{align*}
Thus $\pi^{*}\partial([v])=[\Gg_{0}]=[\Gd_{0}]=[\ad{v}\circ J,J]$.
\end{proof}
We will show that $\Gc$ preseves the orbits and stabilizers of the actions. First, we consider the orbits.\\
We define $T:[A,B]_{G,u}\to KK^{G}(A)$ by $T([\phi])=KK^{G*}(\phi)-KK^{G}(j)$ and consider the following sequence
\begin{align}
  [A,B]_{G,u}\xlongrightarrow{T}KK^{G}(A,B)\xlongrightarrow{\nu_{A}^{*}}KK^{G}(\Cbb,B).\label{eq:2}
\end{align}
We will show that we can identify the orbit space of $\Ucal(B^{\Gb})/\Ucal_{0}(B^{\Gb})\act [A,B]_{G,u}$ with $\ker \nu_{A}^{*}$.
\begin{proposition}[cf. {\cite[Lemma 3.1]{D1}}]\label{prop:T-inj}
For $\phi,\psi\in \homg{A}{B}$, the following conditions are equivalent.
\begin{enumerate}
\item $T([\phi])=T([\psi])$,
\item there exists a unitary $v\in \Ucal(B^{\Gb})$ such that $[\ad{v}\circ \phi]=[\psi]$.
\end{enumerate}
\end{proposition}
\begin{proof}
First, suppose $T([\phi])=T([\psi])$. Then we have $KK^{G}(\phi)=KK^{G}(\psi)$ and this implies that $\phi$ and $\psi$ are asymptotically unitarily equivalent. Hence there exists a continuous path $u:[0,\infty)\to \Ucal(B)$ such that
\[\lim_{t\to\infty}\ad{u_{t}}\circ \phi(a)=\psi(a),\quad\lim_{t\to\infty}\max_{g\in G}\|u_{t}-\Gb_{g}(u_{t})\|=0\]
for all $a\in A$. By cutting off an initial segments of $u$, we may assume $\max_{g\in G}\|u_{t}-\Gb_{g}(u_{t})\|<1$ for all $t$. Then $\int_{G}\Gb_{g}(u_{t})dg$ is invertible and the unitary part of its polar decomposition  defines a continuous path $v:[0,\infty)\to\Ucal(B^{\Gb})$ satisfying $\lim_{t\to\infty}\ad{v_{t}}\circ \phi(a)=\psi(a)$. Thus $v_{0}$ satisfies the second condition.

Suppose that there exists a unitary $v\in \Ucal(B^{\Gb})$ such that $[\ad{v}\circ \phi]=[\psi]$. It follows that $KK^{G}(\ad{v}\circ\phi)$ is equal to $KK^{G}(\psi)$. Since we have $KK^{G}(\phi)=KK^{G}(\ad{v}\circ\phi)$, it follows that 
\[T([\phi])=KK^{G}(\phi)-KK^{G}(j)=KK^{G}(\psi)-KK^{G}(j)=T([\psi]).\]
\end{proof}
\begin{proposition}[cf. {\cite[Lemma 3.3]{D1}}]
In the sequence \eqref{eq:2}, $\im{T}=\ker{\nu_{A}^{*}}$ holds.
\end{proposition}
\begin{proof}
First, we show the inclusion $\im{T}\subset \ker{\nu_{A}^{*}}$. Let $\phi\in\homg{A}{B}$. Then we have
\begin{align*}
\nu_{A}^{*}(T([\phi]))&=\nu_{A}^{*}(KK^{G}(\phi)-KK^{G}(j))\\
&=KK^{G}(\phi\circ \nu_{A})-KK^{G}(j\circ\nu_{A})\\
&=KK^{G}(\nu_{B})-KK^{G}(\nu_{B})=0.
\end{align*}
We turn to the inclusion $\ker{\nu_{A}^{*}}\subset \im{T}$. Let $x$ be any element in $\ker{\nu_{A}^{*}}$. Then $x+KK^{G}(j)$ induces a map from $K_{0}^{G}(A)$ to $K_{0}^{G}(B)$ which maps $[\one_{A}]$ to $[\one_{B}]$. By \autoref{cor:compact-ex}, there exists a unital equivariant $*$-homomorphism $\phi:A\to B$ such that $KK^{G}(\phi)=x+KK^{G}(j)$. Hence we have $x=T([\phi])$.
\end{proof}
Next we consider the orbit of $KK^{G}(\Cbb,SB)\act KK^{G}(\cna,SB)$. Consider the following exact sequence induced by the mapping cone sequence \eqref{eq:mapping-cone-seq}
\[KK^{G}(\Cbb,SB)\xlongrightarrow{\pi^{*}}KK^{G}(\cna,SB)\xlongrightarrow{i^{*}}KK^{G}(SA,SB)\xlongrightarrow{(S\nu_{A})^{*}}KK^{G}(S\Cbb,SB).\]
It follows from this exact sequence that the orbit space of $KK^{G}(\Cbb,SB)\act KK^{G}(\cna,SB)$ is equal to $\im{i^{*}}=\ker{(S\nu_{A})^{*}}$.
\vskip\baselineskip
Next we turn to the stabilizers.

\begin{lem}[cf. {\cite[Proposition 3.2]{D1}}]
For any $\phi\in\homg{A}{B}$ and any $x\in KK^{G}(A,SB)$, there exists an equivariant $*$-homomorphism $\Psi=(\Psi_{t}):A\to M_{2}(C[0,1]\otimes B)$ such that $\Psi_{0}=\Psi_{1}$, $[((\Psi,1),(\Psi_{0},1))]=x\in KK^{G}(A,SB)$ and $\Psi_{0}=\tmat{\phi_{0}}{0}{0}{0}:A\to M_{2}(B)$ where $\phi_{0}\in\homg{A}{B}$ with $[\phi_{0}]=[\phi]$ in $[A,B]_{G,u}$.
\end{lem}
\begin{proof}
Fix an element $x$ in $KK^{G}(A,SB)$. Using the split exact sequence
\[0\rightarrow KK^{G}(A,SB)\rightarrow KK^{G}(A,C(\Tbb)\otimes B)\rightarrow KK^{G}(A,B)\rightarrow 0\]
We can regard $x$ as an element of $KK^{G}(A,C(\Tbb)\otimes B)$. By \autoref{cor:compact-ex}, there exists an equivariant $*$-homomorphism $\Gg=(\Gg_{t}):A\to C[0,1]\otimes B$ such that $\Gg_{0}=\Gg_{1}$ and $KK^{G}(\Gg,1)=x$. We can arrange $\Gg_{0}(\one_{A})$ so that $\one_{B}-\Gg_{0}(1_{A})\in B^{\Gb}$ is full and properly infinite. Since $x\in\ker{(KK^{G}(A,C(\Tbb)\otimes B)\to KK^{G}(A,B))}$, we obtain $KK^{G}(\Gg_{0},1)=0$. The $*$-homomorphisms
\[\theta\coloneqq \pmat{\phi}{0}{0}{\Gg_{0}},\: \theta'\coloneqq \pmat{\phi}{0}{0}{0}:A\to M_{2}(B)\]
are equivariant and $KK^{G}(\theta,\one)=KK^{G}(\theta',\one)$ in $KK^{G}(A,B)$. Therefore $\left[\tmat{\one}{0}{0}{\Gg_{0}(\one)}\right]=\left[\tmat{\one}{0}{0}{0}\right]$ in $K_{0}(B^{\Gb})$. Because these projections and their orthogonal complements are full and properly infinite, $\tmat{\one}{0}{0}{\Gg_{0}(\one)}$ and $\tmat{\one}{0}{0}{0}$ are unitarily equivalent in $M_{2}(B^{\Gb})$. Thus there exists $\Go\in \Ucal(M_{2}(B^{\Gb}))$ such that $\Go\tmat{\one}{0}{0}{\Gg_{0}(\one)}\Go^{*}=\tmat{\one}{0}{0}{0}$. Then $\ad{\Go}\circ\theta(\one_{A})=\tmat{\one}{0}{0}{0}$ and hence we may regard $\ad{\Go}\circ\theta$ as a $*$-homomorphism into $B$, which we denoted by $\psi$. Then $\psi$ is unital and equivariant and 
\[KK^{G}(\psi,\one)=KK^{G}(\theta,\one)=KK^{G}(\phi,\one)+KK^{G}(\Gg_{0},\one)=KK^{G}(\phi,\one).\]
Then there exists $u_{0}\in \Ucal(B^{\Gb})$ such that $[\ad{u_{0}}\circ\psi]=[\phi]$. Let us set $\phi_{0}=\ad{u_{0}}\circ \psi$ and $u=\tmat{u_{0}}{0}{0}{\one}\in \Ucal(M_{2}(B^{\Gb}))$. We define $\Psi=(\Psi_{t})_{t\in[0,1]}:A\to M_{2}(C[0,1]\otimes B)$ by $\Psi_{t}=\ad{u\Go}\circ\tmat{\phi}{0}{0}{\Gg_{t}}$. Then $\Psi_{0}=\Psi_{1}$ and hence $(\Psi,\Psi_{0})\in \mathbb{E}^{G}(A,SB)$. Moreover we have
\[\Psi_{0}=\pmat{u_{0}}{0}{0}{\one}\pmat{\psi}{0}{0}{0}\pmat{u_{0}}{0}{0}{\one}^{*}=\pmat{\ad{u_{0}}\circ\psi}{0}{0}{0}=\pmat{\phi_{0}}{0}{0}{0}.\]
We conclude the proof by observing that 
\begin{align*}
\left[\left((\Psi,\mathbf{1}),(\Psi_{0},\mathbf{1})\right)\right]
&=\left[\left(\pmat{\phi}{0}{0}{\Gg},\pmat{\mathbf{1}}{0}{0}{\mathbf{1}}\right),\left(\pmat{\phi}{0}{0}{\Gg_{0}},\pmat{\mathbf{1}}{0}{0}{\mathbf{1}}\right)\right]\\
&=KK^{G}(\Gg)-KK^{G}(\Gg_{0})=KK^{G}(\Gg)=x.
\end{align*}
\end{proof}

\begin{proposition}[cf. {\cite[Proposition 3.4]{D1}}]\label{prop:stab-htpyset}
The stabilizer group of $[\phi]\in[A,B]_{G,u}$ can be identified with
\[\im{\left(Q\coloneqq \partial^{-1}\circ \nu_{A}^{*}:KK^{G}(A,SB)\to K_{1}(B^{\Gb})\right)}.\]
That is, for $v\in \Ucal(B^{\Gb})$ and $\phi\in\homg{A}{B}$, $[\ad{v}\circ\phi]=[\phi]$ in $[A,B]_{G,u}$ if and only if $[v]\in \im{Q}$.
\end{proposition}
\begin{proof}
First suppose that $[\ad{v}\circ\phi]=[\phi]$ for $\phi\in\homg{A}{B}$ and $v\in \Ucal(B^{\Gb})$. Then there exists a unital equivariant $*$-homomorphism $H:A\to C[0,1]\otimes B$ such that $H_{0}=\phi$ and $H_{1}=\ad{v}\circ\phi$. Let $w:[0,1]\to \Ucal(M_{2}(B^{\Gb}))$ be a continuous map such that $w_{0}=\tmat{v}{0}{0}{v^{*}}$ and $w_{1}=\tmat{\one}{0}{0}{\one}$. Define $H':A\to M_{2}(C[0,1]\otimes B)\subset M(SB\otimes \Kbb)$ by $H'_{t}=\ad{w_{t}}\circ\tmat{H_{t}}{0}{0}{0}$. Then, $H'_{t}:A\to M_{2}(B)$ is equivariant and 
\[H'_{0}=\ad{w_{0}}\circ\tmat{H_{0}}{0}{0}{0}=\tmat{v}{0}{0}{v^{*}}\tmat{\phi}{0}{0}{0}\tmat{v^{*}}{0}{0}{v}=\tmat{\ad{v}\circ \phi}{0}{0}{0}=H'_{1}.\]
Hence $\displaystyle \left((H',\one),(H'_{0},\one)\right)\in \mathbb{E}^{G}((A,\Ga),(SB,S\Gb))$. Then we have 
\begin{align*}
\nu_{A}^{*}[((H',\mathbf{1}),(H'_{0},\mathbf{1}))]&=[((H'\circ \nu_{A},\mathbf{1}),(H'_{0}\circ\nu_{A},\mathbf{1}))]\\
&=\left[\left((\Gl\mapsto\Gl w\pmat{\one}{0}{0}{0}w^{*},\pmat{\mathbf{1}}{0}{0}{\mathbf{1}}),(\Gl\mapsto \Gl\pmat{\one}{0}{0}{0},\pmat{\mathbf{1}}{0}{0}{\mathbf{1}})\right)\right]\\
&=\partial([v])
\end{align*}
It follows that $[v]$ is in $\im{Q}$.

Suppose that $[v]=Q(x)$ for some $x\in KK^{G}(A,SB)$. By the preceding lemma, there exists an equivariant $*$-homomorphism $\Psi=(\Psi_{t})_{t\in[0,1]}:A\to M_{2}(C[0,1]\otimes B)$ such that $\Psi_{0}=\Psi_{1}$, $x=[((\Psi,\one),(\Psi_{0},\one))]$ and $\Psi_{0}=\tmat{\phi_{0}}{0}{0}{0}:A\to M_{2}(B)$ for some $\phi_{0}\in\homg{A}{B}$ with $[\phi_{0}]=[\phi]$ in $[A,B]_{G,u}$. If we set $p=(p_{t})_{t\in [0,1]}=\Psi(\one)=(\Psi_{t}(\one))$, then $\nu_{A}^{*}\left(\left[(\Psi,\one),(\Psi_{0},\one)\right]\right)=\left[(\Gl\mapsto \Gl p,1),(\Gl\mapsto\Gl p_{0},1)\right]$ and $p_{0}=p_{1}=\tmat{\one}{0}{0}{0}$.

Let $w:[0,1]\to \Ucal(M_{2}(B^{\Gb}))$ be a continuous path such that $w_{1}=\tmat{\one}{0}{0}{\one}$ and $p_{t}=w_{t}p_{0}w_{t}^{*}$ for all $t\in[0,1]$. As $p_{0}=w_{0}p_{0}w_{0}^{*}$, $w_{0}=\tmat{v_{0}}{0}{0}{v_{0}'}$ for some $v_{0},\ v_{0}'\in \Ucal(B^{\Gb})$. By the desciption of $\partial$, we have $Q(x)=\partial^{-1}\circ\nu_{A}^{*}\left((\Psi,\Psi_{0})\right)=[v_{0}]$. Hence $v_{0}$ and $v$ are homotopic in $\Ucal(B^{\Gb})$.

Define an equivariant $*$-homomorphism $H=(H_{t})_{t\in[0,1]}:A\to M_{2}(C[0,1]\otimes B)$ by $H_{t}=\ad{w_{t}^{*}}\circ\Psi_{t}$. Then $H_{t}(1)=w_{t}^{*}p_{t}w_{t}=\tmat{\one}{0}{0}{0}=p_{0}$ for all $t\in[0,1]$. Hence We may regard $H=(H_{t})$ as a $*$-homomorphism into $C[0,1]\otimes B$, which we denote by $h=(h_{t})$. Then $h_{t}$ is unital and equivariant. For all $t\in [0,1]$ we have 
\begin{align*}
H_{0}&=\ad{w_{0}^{*}}\circ\Psi_{0}=\ad{\pmat{v_{0}}{0}{0}{v_{0}'}^{*}}\circ\pmat{\phi_{0}}{0}{0}{0}=\pmat{\ad{v_{0}^{*}}\circ\phi_{0}}{0}{0}{0}\\
H_{1}&=\ad{w_{1}^{*}}\circ\Psi_{\one}=\ad{\pmat{\one}{0}{0}{\one}}\circ\pmat{\phi_{0}}{0}{0}{0}=\pmat{\phi_{0}}{0}{0}{0}
\end{align*}
This implies $h_{0}=\ad{v_{0}^{*}}\circ\phi_{0}$ and $h_{1}=\phi_{0}$. 
In conclusion $[\phi]=[\phi_{0}]=[\ad{v_{0}^{*}}\circ\phi_{0}]=[\ad{v^{*}}\circ\phi]$ and hence $[\ad{v}\circ\phi]=[\phi]$.
\end{proof}

\begin{proposition}
The stabilizer group of $\Gc([\phi])$ is $\im{\left(\nu_{A}^{*}:KK^{G}(A,SB)\to KK^{G}(\Cbb,SB)\right)}$ 
\end{proposition}
\begin{proof}
This follows from the following exact sequence.
\[KK^{G}(A,SB)\xlongrightarrow{\nu_{A}^{*}}KK^{G}(\Cbb,SB)\xlongrightarrow{\pi^{*}}KK^{G}(\cna,SB).\]
\end{proof}
\begin{theorem}[cf. {\cite[Corollary 3.10]{D1}}]\label{thm:unital-end-htpy}
Let $G$ be a compact group and let $X$ be a compact connected metrizable space. Let $A$ be a unital Kirchberg algeba and $\Ga:G\act A$ be an isometrically shift-absorbing action. Then the map $\Gc:[A,C(X)\otimes A]_{G,u}\to KK^{G}(\cna,SB)$ is a bijection.
\end{theorem}
\begin{proof}
Note that we have the following commutative diagram.
\begin{center}
\resizebox{\textwidth}{!}{
\begin{tikzpicture}[font=\small]
\node (1) at (-7,1) {$KK^{G}(A,SB)$};
\node (2) at (-3.8,1) {$K_{1}^{G}(B)$};
\node (3) at (-0.5,1) {$[A,B]_{G,u}$};
\node (4) at (3.5,1) {$KK^{G}(A,B)$};
\node (5) at (7.5,1) {$KK^{G}(\Cbb,B)$};
\node (6) at (-7,-1) {$KK^{G}(A,SB)$};
\node (7) at (-3.8,-1) {$KK^{G}(\Cbb,SB)$};
\node (8) at (-0.5,-1) {$KK^{G}(\cna,SB)$};
\node (9) at (3.5,-1) {$KK^{G}(SA,SB)$};
\node (10) at (7.5,-1) {$KK^{G}(S\Cbb,SB)$};
\draw[->] (1) to node[above] {$Q$} (2);
\draw[->] (2) to node[above] {$I$} (3);
\draw[->] (3) to node[above] {$T$} (4);
\draw[->] (4) to node[above] {$\nu_{A}^{*}$} (5);
\draw[->] (6) to node[above] {$\nu_{A}^{*}$} (7);
\draw[->] (7) to node[above] {$\pi^{*}$} (8);
\draw[->] (8) to node[above] {$i^{*}$} (9);
\draw[->] (9) to node[above] {$(S\nu_{A})^{*}$} (10);
\draw[double, -] (1) to (6); 
\draw[->] (2) to node[right] {$\partial$} (7);
\draw[->] (3) to node[right] {$\Gc$} (8);
\draw[->] (4) to node[right] {$S$} (9);
\draw[->] (5) to node[right] {$S$} (10);
\end{tikzpicture}
}
\end{center}
Here $I$ is defined by $I([v])=[\ad{v}\circ j]$.

Suppose that $\Gc([\phi])=\Gc([\psi])$. Since $S:KK^{G}(A,B)\to KK^{G}(SA,SB)$ is bijection, $i^{*}\Gc([\phi])=i^{*}\Gc([\psi])$ implies that $T([\phi])=T([\psi])$. By \autoref{prop:T-inj}, there exists a unitary $\Ucal(B^{\Gb})$ such that $[\ad{v}\circ\phi]=[\psi]$. Then we have $\Gc([\phi])=\Gc([\psi])=\Gc([\ad{v}\circ \phi])=\Gc([\phi])+\pi^{*}\partial ([v])$. This implies $[v]$ is in the image of $Q$. Thus it follows from \autoref{prop:stab-htpyset} that $[\ad{v}\circ\phi]=[\phi]$, and hence we have $[\phi]=[\psi]$.

Let $z\in KK^{G}(\cna.SB)$ be any element. Then $i^{*}(z)$ is in the kernel of $(S\nu_{A})^{*}$. Since $S:KK^{G}(A,B)\to KK^{G}(SA,SB)$ induces a bijection $\im{T}=\ker{\nu_{A}^{*}}\to \ker{(S\nu_{A})^{*}}$, there exists $[\phi]\in [A,B]_{G,u}$ such that $ST([\phi])=i^{*}(z)$. Thus $\Gc([\phi])$ is in the orbit of $z$ and there exists a unitary $v\in \Ucal(B^{\Gb})$ such that $\pi^{*}\circ\partial ([v])=z-\Gc([\phi])$. Then $\Gc([\ad{v}\circ \phi])=z$ and this finishes the proof.
\end{proof}

\begin{remark}
Here we note some facts about homotopy sets. We refer to \cite{Wh-htpy} for the followings. Let $(X,x_{0})$ and $(Y,y_{0})$ be pointed spaces. We denote by $[X,Y]$ the homotopy classes of continuous maps from $X$ to $Y$. Also, we denote by $[(X,x_{0}),(Y,y_{0})]$ the homotopy classes of base-point preserving continuous maps from $X$ to $Y$. We suppose that the base point $x_{0}$ is non-degenerate, that is, $\{x_{0}\}\to X$ is a cofibration. Then there is a natural action of $\pi_{1}(Y,y_{0})$ on $[(X,x_{0}),(Y,y_{0})]$. If $Y$ is path connected, then the orbit space of this action is identified with $[X,Y]$. We will note the relation between $[(X,x_{0}),(Y,y_{0})]$, and $[X,Y]$, and their monoidal structure.

A pointed space $(Y,y_{0})$ is an \textit{$H$-space} if there is a continuous map $m:Y\times Y\to Y$ such that the $y\mapsto m(y,y_{0})$ and $y\mapsto m(y_{0},y)$ are homotopic to $\id{Y}$. It is known that the action of $\pi_{1}(Y,y_{0})$ on $[(X,x_{0}),\pty]$ is trivial for every $H$-space $(Y,y_{0})$. Thus if $Y$ is path connected, the natural map $[(X,x_{0}),(Y,y_{0})]\to [X,Y]$ is bijective. Moreover, if $m$ is homotopy associative, then $[(X,x_{0}),(Y,y_{0})]$ has a monoid structure. If $X$ is a CW-complex and $Y$ is path connected, then $[(X,x_{0}),(Y,y_{0})]\cong [X,Y]$ is a group.

Let $X\vee X$ be the wedge product. We regard $X\vee X\subset X\times X$ so that every element of $X\vee X$ is of the form $(x,x_{0})$ or $(x_{0},x')$ for $x,x'\in X$. Let $q_{1}=p_{1}|_{X\vee X}:X\vee X\to X$ and $q_{2}=p_{2}|_{X\vee X}:X\vee X\to X$, where $p_{1},p_{2}:X\times X\to X$ are the projections. A pointed space $(X,x_{0})$ is a \textit{co-$H$-space} if there exists a continuous map $\Gm:X\to X\vee X$ such that $q_{1}\circ\Gm$ and $q_{2}\circ\Gm$ are homotopic to $\id{X}$. If $(X,x_{0})$ is a co-$H$-space and $(Y,y_{0})$ is an $H$-space, then there are two multiplications on $[(X,x_{0}),(Y,y_{0})]$ induced by $m$ and $\Gm$. It is known that they coincide. Moreover, the multiplication is commutative and associative.
\end{remark}

Define a map
\[\Gc_{X}:[(X,x_{0}),(\mathrm{End}_{G}(A),\id{A})]\to KK^{G}(\cna, SC(X,x_{0})\otimes A)\]
in the similar way to $\Gc$ as in \autoref{thm:unital-end-htpy}. Then the following diagram commutes.
\begin{center}
\begin{tikzpicture}[auto]
\node (1) at (-3,1) {$[(X,x_{0}),(\mathrm{End}_{G}(A),\id{A})]$};
\node(2) at (3,1) {$KK^{G}(\cna,SC(X,x_{0})\otimes A)$};
\node (3) at (-3,-1) {$[A,C(X)\otimes A]_{G,u}$};
\node (4) at (3,-1) {$KK^{G}(\cna,SC(X)\otimes A)$};
\draw[->] (1) to node {$\Gc_{X}$} (2);
\draw[->] (1) to (3);
\draw[->] (2) to (4);
\draw[->] (3) to node {$\Gc$} (4);
\end{tikzpicture}
\end{center}
If $(X,x_{0})$ is a co-$H$-space, for $\phi,\psi:(X,x_{0})\to(\mathrm{End}_{G}(A),\id{A})$, $[\phi]\ast[\psi]=[(\phi\vee\psi)\circ\Gm]$ is the multiplication on $[(X,x_{0}),(\mathrm{End}_{G}(A),\id{A})]$ induced by $\Gm$.
\begin{proposition}[cf. {\cite[Proposition 4.4]{D1}}]\label{prop:unital-multi}
If $(X,x_{0})$ is a co-$H$-space, then $\Gc_{X}([\phi]\ast[\psi])=\Gc_{X}([\phi])+\Gc_{X}([\psi])$.
\end{proposition}
\begin{proof}
For a pointed map $f:(X,x_{0})\to (Y,y_{0})$, let $f^{*}:[\pty,(\mathrm{End}_{G}(A),\id{A})]\to[\ptx,(\mathrm{End}_{G}(A),\id{A})]$ and $f^{*}:KK^{G}(\cna,SC\pty\otimes A)\to KK^{G}(\cna,SC\ptx\otimes A)$ be induced maps. Then clearly $f^{*}\circ\Gc_{X}=\Gc_{X}\circ f^{*}$ holds. Thus for the co-multiplication $\Gm:\ptx\to (X\vee X,x_{0})$, we have the following commutative diagram.
\begin{center}
\begin{tikzpicture}[auto]
\node (1) at (-4,1) {$[(X\vee X,x_{0}),(\mathrm{End}_{G}(A),\id{A})]$};
\node (2) at (4,1) {$[\ptx,(\mathrm{End}_{G}(A),\id{A})]$};
\node (3) at (-4,-1) {$KK^{G}(\cna,SC(X\vee X,x_{0})\otimes A)$};
\node (4) at (4,-1) {$KK^{G}(\cna,SC\ptx\otimes A)$};
\draw[->] (1) to node {$\Gm^{*}$} (2);
\draw[->] (1) to node {$\Gc_{X\vee X}$} (3);
\draw[->] (2) to node {$\Gc_{X}$} (4);
\draw[->] (3) to node {$\Gm^{*}$} (4);
\end{tikzpicture}
\end{center}
Let $i_{1},i_{2}:X\to X\vee X$ be the maps defined by $i_{1}(x)=(x,x_{0})$ and $i_{2}(x)=(x_{0},x)$. Then we have $(\phi\vee\psi)\circ i_{1}=\phi$ and $(\phi\vee\psi)\circ i_{2}=\psi$. By combining with the fact that $C(X\vee X,x_{0})\cong C(X,x_{0})\oplus C\ptx$ and $q_{k}\circ i_{k}=\id{}$, we have $\Gc_{X\vee X}([\phi\vee\psi])=q_{1}^{*}\circ \Gc_{X}([\phi])+q_{2}\circ \Gc_{X}([\psi])$. Thus it follow that
\begin{align*}
\Gc_{X}([\phi]\ast[\psi])&=\Gc_{X}([(\phi\vee\psi)\circ \Gm])=\Gm^{*}\circ \Gc_{X\vee X}([\phi\vee\psi])\\
&=\Gm^{*}\circ q_{1}^{*}\circ\Gc_{X}([\phi])*+\Gm^{*}\circ q_{2}^{*}\circ\Gc_{X}([\psi])=\Gc_{X}([\phi])+\Gc_{X}([\psi])
\end{align*}
\end{proof}
For the stable case, let $\overline{\Gc}_{X}:[\ptx,(\mathrm{End}_{G}(A),\id{A})]\to KK^{G}(A,C(X,x_{0})\otimes A)$ be the map defined by $\overline{\Gc}_{X}([\phi])=[(\phi,\one),(j_{A},\one)]$, where $j_{A}:A\ni a\mapsto \one_{C(X)}\otimes a\in C(X)\otimes A$. In the similar way, we have 
\begin{proposition}[cf. {\cite[Proposition 4.5]{D1}}]\label{prop:stable-multi}
If $\ptx$ is a co-$H$-space, $\overline{\Gc}_{X}([\phi]\ast[\psi])=\overline{\Gc}_{X}([\phi])+\overline{\Gc}_{X}([\psi])$.
\end{proposition}
Let $\mathrm{End}_{G}(A)^{0}$ be the path component of $\id{A}$.
\begin{theorem}[cf. {\cite[Theorem 4.6]{D1}}]\label{thm:end-grp-str}
Let $G$ be a compact group and let $X$ be a compact path connected metrizable space and $x_{0}\in X$. Let $A$ be a Kirchberg algebra and let $\Ga:G\act A$ be an isometrically shift-absorbing action.\\
(i) If $A$ is stable, $\Gc_{X}:[X,\mathrm{End}_{G}(A)^{0}]\to KK^{G}(\cna,SC\ptx\otimes A)$ is a bijection.\\
(ii) If $A$ is unital, $\overline{\Gc}_{X}:[X,\mathrm{End}_{G}(A)^{0}]\to KK^{G}(A,C\ptx\otimes A)$ is a bijection.\\
If $\ptx$ is a co-$H$-space, these maps are group isomorphisms. 
\end{theorem}
\begin{proof}
(i) : Since $\mathrm{End}_{G}(A)^{0}$ is path connected, we can identify  $[\ptx,(\mathrm{End}_{G}(A)^{0},\id{A})]$ with $[X,\mathrm{End}_{G}(A)^{0}]$. Note that
\[[\ptx,(\mathrm{End}_{G}(A)^{0},\id{A})]\cong \{[\phi]\in [A,C(X)\otimes A]_{G,u}\:|\: [\phi_{x}]=[\id{A}]\in [A,A]_{G}\}\]
as topological spaces. Then we have the following commutative diagram.
\begin{center}
\resizebox{\textwidth}{!}{
\begin{tikzpicture}[font=\small]
\node (1) at (-6,1) {$[X,\mathrm{End}_{G}(A)^{0}]$};
\node(2) at  (0,1) {$[A,C(X)\otimes A]_{G,u}$};
\node (3) at (6,1) {$[A,A]_{G,u}$};
\node (4) at (-6,-1) {$KK^{G}(\cna,SC\ptx\otimes A)$};
\node (5) at (0,-1) {$KK^{G}(\cna,SC(X)\otimes A)$};
\node (6) at (6,-1) {$KK^{G}(\cna,SA)$};
\draw[->] (1) to (2);
\draw[->] (2) to node[above] {$(\mathrm{ev}_{x_{0}})_{*}$} (3);
\draw[->] (1) to node[right] {$\Gc_{X}$} (4);
\draw[->] (2) to node[right] {$ \Gc$} (5);
\draw[->] (3) to node[right] {$\Gc$} (6);
\draw[->] (4) to (5);
\draw[->] (5) to node[above] {$(\mathrm{ev}_{x_{0}})_{*}$} (6);
\end{tikzpicture}}
\end{center}
Since the middle and right vertical maps are bijective, it follows that $\Gc_{X}$ is also bijective.\\
(ii) : Consider the following diagram.
\begin{center}
  \resizebox{\textwidth}{!}{
\begin{tikzpicture}[font=\small]
\node (1) at (-6,1) {$[X,\mathrm{End}_{G}(A)^{0}]$};
\node(2) at  (0,1) {$[A,C(X)\otimes A]_{G,u}$};
\node (3) at (6,1) {$[A,A]_{G,u}$};
\node (4) at (-6,-1) {$KK^{G}(j_{A})+KK^{G}(A,C\ptx\otimes A)$};
\node (5) at (0,-1) {$KK^{G}(A,C(X)\otimes A)$};
\node (6) at (6,-1) {$KK^{G}(A,A)$};
\draw[->] (1) to (2);
\draw[->] (2) to node[above] {$(\mathrm{ev}_{x_{0}})_{*}$} (3);
\draw[->] (1) to node[right] {$KK^{G}(j_{A})+\overline{\Gc}_{X}$} (4);
\draw[->] (2) to node[right] {$ \Gk$} (5);
\draw[->] (3) to node[right] {$\Gk$} (6);
\draw[->] (4) to (5);
\draw[->] (5) to node[above] {$(\mathrm{ev}_{x_{0}})_{*}$} (6);
\end{tikzpicture}}
\end{center}
In the same way to the unital case, it follows that $KK^{G}(j_{A})+\overline{\Gc}_{X}$ is bijective, and hence $\overline{\Gc}_{X}$ is bijective.

If $\ptx$ is a co-$H$-space, it follows from \autoref{prop:unital-multi} and \autoref{prop:stable-multi} that $\Gc_{X}$ and $\overline{\Gc}_{X}$ are group isomorphisms. 
\end{proof}
\begin{remark}
Since $\mathrm{End}_{G}(A)^{0}$ is an $H$-space, the group structure induced by the co-multiplication $\Gm:X\to X\vee X$ and the structure induced by pointwise composition of endomorphisms coinside.
\end{remark}

\subsubsection{The homotopy set $[X,\mathrm{Aut}_{G}(A)]$}
We define \[\End{A}^{*}=\{\phi\in\End{A} \mid KK^{G}(\phi) \text{ is invertible in } KK^{G}(A,A)\}.\]

\begin{definition}[cf. {\cite[Definition 5.3]{D1}}] Let $A$ be a separable C$^{\ast}$-algebra and let $X$ be a compact metrizable space with trivial $G$-action. We say that the pair $(A,X)$ is $KK^{G}$-continuous if for any point $x\in X$ there exists a base of closed neighborhoods $(V_{n})_{n}$ of $x$ such that $V_{n}\supset V_{m}$ for $n<m$ and the natural map $\varinjlim KK^{G}(A,C(V_{n})\otimes A)\to KK^{G}(A,A)$ is injective.
\end{definition}

\begin{remark}
For every $A$ and $G$-action, if $X$ is locally contractible, then $(A,X)$ is a $KK^{G}$-continuous pair. If $(A,X)$ is a $KK^{G}$-continuous pair and $Y$ is a locally contractible space, then $(A,X\times Y)$ is also a $KK^{G}$-continuous pair.
\end{remark}
We use the following notation.
\begin{align*}
K^{G}_{A}(X)&=KK^{G}(A,C(X)\otimes A)\\
K^{G}_{A}(X,Y)&=KK^{G}(A,C(X,Y)\otimes A)\quad\text{for }Y\subset X:\text{closed}
\end{align*}
where $C(X,Y)=C_{0}(X\setminus Y)$. The composition of the Kasparov product
\[K_{A}^{G}(X)\times K^{G}_{A}(X)\to K^{G}_{A}(X)\]
with the restriction to the diagonal map
\[K^{G}_{A}(X\times X)\to K^{G}_{A}(X)\]
defines a cup product on $K^{G}_{A}(X)$ which makes $K^{G}_{A}(X)$ into a ring.
The multiplicative unit $\Gi_{A}$ of $K^{G}_{A}(X)$ is given by the $KK^{G}$-class of the $*$-homomorphism
\[j_{A}:A\ni a\mapsto 1_{C(X)}\otimes a\in C(X)\otimes A\]
Similarly, one has a cup product 
\[K^{G}_{A}(X,Y)\times K^{G}_{A}(X,Y')\to K^{G}_{A}(X,Y\cup Y')\]
which is compatible with the cup product on $K^{G}_{A}(X)$. The multiplication on $K^{G}_{A}(X)$ is described as follows.
\begin{center}
\begin{tikzpicture}[auto]
\node (1) at (-4,1) {$K^{G}_{A}(X)\otimes K^{G}_{A}(X)$};
\node (2) at (0,1) {$K^{G}_{A}(X\times X)$};
\node (3) at (4,1) {$K^{G}_{A}(X)$};
\node (4) at (-4,0.5) {$\vin$};
\node (5) at (0,0.5) {$\vin$};
\node (6) at (4,0.5) {$\vin$};
\node (7) at (-4,0) {$([\Phi],[\Psi])$};
\node (8) at (0,0) {$[(\id{C(X)}\otimes \Psi)\circ \Phi]$};
\node (9) at (4,0) {$\tilde{\Psi}\circ\Phi $};
\draw[->] (1) to (2);
\draw[->] (2) to (3);
\draw[|->] (7) to (8);
\draw[->] (8) to (9);
\end{tikzpicture}
\end{center}
where $\tilde{\Psi}:C(X)\otimes A\to C(X)\otimes A$ is the $C(X)$-linear extension of $\Psi$.
For $\Gs\in KK^{G}(A,C(X)\otimes A)$, we write $\Gs_{x}$ for the element $(ev_{x})_{*}(\Gs)\in KK^{G}(A,A)$.
\begin{proposition}[cf. {\cite[Proposition 5.6]{D1}}]\label{prop:invertibility}
If the pair $(A,X)$ is $KK^{G}$-continuous, then $\Gs$ is invertible in the ring $K^{G}_{A}(X)$ if and only if $\Gs_{x}\in KK^{G}(A,A)$ is invertible for all $x\in X$.
\end{proposition}
\begin{proof}
For each $x\in X$, the following sequence is split exact.
\[0\longrightarrow C(X,x)\otimes A\longrightarrow C(X)\otimes A\xlongrightarrow{ev_{x}} A\rightarrow 0\]
Thus we have a split exact sequence
\[0\longrightarrow K^{G}_{A}(X,x)\longrightarrow K^{G}_{A}(X)\longrightarrow K^{G}_{A}(x)=KK^{G}(A,A)\longrightarrow 0\]
For $\Gs\in K^{G}_{A}(X)$, let $\hat{\Gs}:X\to KK^{G}(A,A)$ be the map $\hat{\Gs}(x)=\Gs_{x}$. We claim that $\hat{\Gs}$ is locally constant. Assume that $\hat{\Gs}$ is not locally constant. Then there exists $x\in X$
such that the restriction of $\hat{\Gs}$ on any neighborhood of $x$ is not constant. Let $(V_{n})$ be a base of closed neighborhood of $x$ as in the definition on $KK^{G}$-continuous. Since $\hat{\Gs}$ is not locally constant, we can take $x_{n}\in V_{n}$ such that $\Gs_{x}\neq \Gs_{x_{n}}$ for all $n$. By replacing $\Gs$ by $\Gs-\Gs_{x}$, we may assume that $\Gs_{x}=0$. For every $n$, the element in $K_{A}^{G}(V_{n})$ corresponding to $\Gs$ is nonzero, and hence the element in $\varinjlim K_{A}^{G}(V_{n})$ corresponging to $\Gs$ is nonzero, but the image of this element under the map $\varinjlim K_{A}^{G}(V_{n})\to K_{A}^{G}(x)$ is zero. This contradicts to the fact that $\varinjlim K_{A}^{G}(V_{n})\to K_{A}^{G}(x)$ is injective. Thus the map $\hat{\Gs}$ is locally constant, and hence constant. Therefore we have a split exact sequence of rings
\[0\longrightarrow \ker{\rho}\longrightarrow K^{G}_{A}(X)\xlongrightarrow{\rho} \check{H}^{0}(X,KK^{G}(A,A))\longrightarrow 0 \]
where $\rho(\Gs)=\hat{\Gs}$ and $\check{H}^{0}(X,KK^{G}(A,A))=\{f:X\to KK^{G}(A,A)\:|\: \text{locally constant}\}$.

We will show that $\Ker{\rho}$ is a nil ideal, that is, for all $\Gs\in \Ker{\rho}$, there exists $m\geq 1$ such that $\Gs^{m}=0$. Fix $\Gs\in \Ker{\rho}=\bigcap_{x\in X}K^{G}_{A}(X,x)\subset K^{G}_{A}(X)$. For every $x\in X$, let $(V_{n})$ be a base of closed neighborhood of $x$ as in the definition of $KK^{G}$-continuous. Since the map $\varinjlim K^{G}_{A}(V_{n})\to K^{G}_{A}(x)$ is injective, it follows from the following commutative diagram that the natural map $\varinjlim K^{G}_{A}(X,V_{n})\to K^{G}_{A}(X,x)$ is surjective for every $x\in X$.
\begin{center}
\begin{tikzpicture}[auto]
\node (1) at (-3,1) {$\varinjlim K^{G}_{A}(X,V_{n})$};
\node (2) at (0,1) {$K^{G}_{A}(X)$};
\node (3) at (3,1) {$\varinjlim K^{G}_{A}(V_{n})$};
\node (4) at (-5,-1) {$0$};
\node (5) at (-3,-1) {$K^{G}_{A}(X,x)$};
\node (6) at (0,-1) {$K^{G}_{A}(X)$};
\node (7) at (3,-1) {$K^{G}_{A}(x)$};
\node (8) at (5,-1) {$0$};
\draw[->] (1) to (2);
\draw[->] (2) to (3);
\draw[->] (4) to (5);
\draw[->] (5) to (6);
\draw[->] (6) to (7);
\draw[->] (7) to (8);
\draw[->] (1) to (5);
\draw[double, double distance=2pt, -] (2) to (6);
\draw[->] (3) to (7);
\end{tikzpicture}
\end{center}
Since $X$ is compact, there exist closed sets $Y_{1},\cdots Y_{m}$ in $X$ and $\Gs_{k}\in K^{G}_{A}(X,Y_{k})$ for $k=1,\cdots m$ such that each $\Gs_{k}$ maps to $\Gs$ under $K^{G}_{A}(X,Y_{k})\to K^{G}_{A}(X)$ and $\bigcup_{k=1}^{m}Y_{k}=X$. It follows that $\Gs^{m}$ is equal to the image of $\Gs_{1}\cdots\Gs_{m}\in K^{G}_{A}(X,Y_{1}\cup\cdots\cup Y_{m})=0$. Since $\Ker{\rho}$ is a nil ideal and $\rho$ admits a multiplicative splitting, $\Gs\in K^{G}_{A}(X)$ is invertible if and only if $\rho(\Gs)=\hat{\Gs}$ is invertible.
\end{proof}

We first describe the intertwining argument, which will be used to prove \autoref{prop:im-inv}. The following lemma is a special case of \cite[Theorem 4.5]{SZ3} for the case of compact groups.
\begin{lem}[cf. {\cite[Theorem 4.5]{SZ3}}]\label{lem:intertwining}
Let $G$ be a compact group. Let $\Ga:G\act A$ and $\Gb:G\act B$ be actions on separable \cstar-algebras. Suppose that $\phi:A\to B$ and $\psi:B\to A$ are equivariant $*$-homomorphisms such that $\psi\circ\phi$ and $\phi\circ\psi$ are properly asymptotically unitarily equivalent to $\id{A}$ and $\id{B}$ respectively. Then there exist mutually inverse equivariant $*$-isomorphisms $\phi_{0}:A\to B$ and $\psi:B\to A$ such that $\phi_{0}$ and $\psi_{0}$ are properly asymptotically unitarily equivalent to $\phi$ and $\psi$ respectively.
\end{lem}
Let us set $\mathrm{Aut}_{C(X),G}(C(X)\otimes A)=\{\phi\in \mathrm{Aut}_{G}\:|\: \phi\text{ is }C(X)\text{-linear}\}$. For a $C(X)$-linear equivariant $*$-homomorphism $\phi:C(X)\otimes A\to C(X)\otimes A$, if $\psi:C(X)\otimes A\to C(X)\otimes A$ is an euivariant $*$-homomorphism that is properly asymptotically unitarily equivalent to $\phi$, then $\psi$ is also $C(X)$-linear since $C(X)\otimes \mathbf{1}$ is in the center of the unitization of $C(X)\otimes A$.
\begin{proposition}[cf. {\cite[Proposition 5.7]{D1}}]\label{prop:im-inv}
Let $A$ be a Kirchberg algebra and $\Ga$ be an isometrically shift-absorbing action of a compact group $G$. Let $X$ be a compact connected metrizable space. Suppose that $(A,X)$ is $KK^{G}$-continuous. Let $\Gs\in KK^{G}(A,C(X)\otimes A)$ such that $\Gs_{x}\in KK^{G}(A,A)^{-1}$ for all $x\in X$ and $\Gs\otimes [\mathbf{1}_{A}]=[\mathbf{1}_{C(X)\otimes A}]$ in $K_{0}^{G}(C(X)\otimes A)$ if $A$ is unital. Then there exists $\Phi_{0}\in \autx{C(X)\otimes A}$ such that $KK^{G}(\Phi_{0}|_{A})=\Gs$
\end{proposition}
\begin{proof}
First, we deal with the stable case.

By the preceding proposition, $\Gs\in K^{G}_{A}(X)$ is invertible. Thus there exists $\Gs'\in K^{G}_{A}(X)$ such that $\Gs\Gs'=\Gs'\Gs=1$ in $K^{G}_{A}$. By \autoref{thm:stable-end}, there are equivariant embeddings $\Phi,\Psi:A\to C(X)\otimes A$ such that $KK^{G}(\Phi)=\Gs,\;KK^{G}(\Psi)=\Gs'$. Then $KK^{G}(\tilde{\Phi}\circ\Psi)=KK^{G}(\tilde{\Psi}\circ\Phi)=KK^{G}(j_{A})$ holds where $j_{A}=\one_{C(X)}\otimes \id{A}:A\to C(X)\otimes A$. By the uniqueness theorem, $\tilde{\Phi}\circ \Psi$ and $\tilde{\Psi}\circ\Phi$ are strongly asymptotically unitarily equivalent to $j_{A}$. Therefore $\tilde{\Phi}\circ\tilde{\Psi}$ and $\tilde{\Psi}\circ\tilde{\Phi}$ are strongly asymptotically unitarily equivalent to $\tilde{j_{A}}=\id{C(X)\otimes A}$. By \autoref{lem:intertwining}, there exists $\Phi_{0}'\in\autx{C(X)\otimes A}$ such that $\tilde{\Phi}$ is asymptotically unitarily equivalent to $\Phi_{0}'$. Thus there exists a norm-continuous path $u:[0,\infty)\to\Ucal(\mathbf{1}+C(X)\otimes A)$ such that
\begin{align*}
&\tilde{\Phi}(a)=\lim_{t\to\infty}u_{t}\Phi_{0}'(a)u_{t}^{*}\text{ for all }a\in A,\\
&\lim_{t\to\infty}\max_{g\in G}\|u_{t}-(\id{C(X)}\otimes\Ga)(u_{t})\|=0.
\end{align*}
By cutting off an initial segment of $u$, if necessary, we may assume that $\|u_{t}-(\id{C(X)}\otimes \Ga)(u_{t})\|<1$ for all $t$.
By averaging unitaries, we obtain a norm-continuous path $v:[0,\infty)\to\Ucal(\mathbf{1}+(C(X)\otimes A)^{\id{A}\otimes \Ga})$ such that
\[\tilde{\Phi}(a)=\lim_{t\to\infty}v_{t}\Phi'_{0}(a)v_{t}^{*}\text{ for all }a\in A.\]
Set $\Phi_{0}=\ad{u_{0}}\circ \Phi_{0}'$. Then $\Phi_{0}$ is in $\autx{C(X)\otimes A}$ and $\tilde{\Phi}$ and $\Phi_{0}$ are homotopic in $\autx{C(X)\otimes A}$. Therefore $\tilde{\Phi}|_{A}=\Phi$ and $\Phi_{0}|_{A}$ are homotopic in $\mathrm{Hom}_{G}(A,C(X)\otimes A)$, and hence $\Gs=KK^{G}(\Phi)=KK^{G}(\Phi_{0}|_{A})$.\\
We now turn to the unital case. The proof is similar.

By the preceding proposition, there exists $\Gs'\in K^{G}_{A}(X)$ that is the inverse of $\Gs$. Since $\Gs$ maps $[\mathbf{1}_{A}]\in K_{0}^{G}(A)$ to $[\mathbf{1}_{C(X)\otimes A}]\in K_{0}^{G}(A)$, so does $\Gs'$. Thus, there exist unital equivariant embeddings $\Phi,\Psi:A\to C(X)\otimes A$ such that $KK^{G}(\Phi)=\Gs$ and $KK^{G}(\Psi)=\Gs'$. Then $KK^{G}(\tilde{\Phi}\circ\Psi)=KK^{G}(\tilde{\Psi}\circ\Phi)=KK^{G}(j_{A})$ holds and this implies $\tilde{\Phi}\circ\Psi$ and $\tilde{\Psi}\circ\Phi$ are asymptotically unitarily equivalent to $j_{A}$. Therefore $\tilde{\Phi}\circ\tilde{\Psi}$ and $\tilde{\Psi}\circ\tilde{\Phi}$ are asymptotically unitarily equivalent to $\id{C(X)\otimes A}$. By \autoref{lem:intertwining}, there exist unital equivariant $C(X)$-linear automorphisms $\Phi_{0}',\Psi_{0}':C(X)\otimes A\to C(X)\otimes A$ such that $\Phi_{0}'$ and $\Psi_{0}'$ are asymptotically unitarily equivalent to $\tilde{\Phi}$ and $\tilde{\Psi}$ respectively. By averaging unitaries, we obtain a unital equivariant $C(X)$-linear automorphism $\Phi_{0}:C(X)\otimes A\to C(X)\otimes A$ such that $\Phi_{0}$ and $\tilde{\Phi}$ are homotopic. Therefore $\Phi$ and $\Phi_{0}|_{A}$ are homotopic in $\homg{A}{C(X)\otimes A}$ and hence we obtain $\Gs=KK^{G}(\Phi_{0})$.
\end{proof}

\begin{corollary}[cf. {\cite[Proposition 5.8]{D1}}]\label{cor:auto-end}
Under the same assumptions as in the preceding proposition, the natural map $[X,\mathrm{Aut}_{G}(A)]\ni[\phi]\mapsto[\phi]\in [X,\End{A}^{*}]$ is a bijection, that is, $\Aut{G}{A}\to\mathrm{End}_{G}(A)^{*}$ is a weak homotopy equivalence.
\end{corollary}

\begin{proof}
First we prove surjectivity. Let $[\phi]$ be an any element in $[X,\End{A}^{*}]$. Then $KK^{G}(\phi)\otimes (ev_{x})_{*}$ is invertible in $KK^{G}(A,A)$ for every $x\in X$. By the preceding proposition, there exists an equivariant $*$-homomorphism $\psi:A\to C(X)\otimes A$ such that $ev_{x}\circ \psi$ is an automorphism for every $x\in X$ and $\phi$ and $\psi$ are homotopic in $\mathrm{Hom}_{G}(A,C(X)\otimes A)$. Therefor there exists $\psi:X\to \mathrm{Aut}_{G}(A)$ such that $[\phi]=[\psi]$.

Next we show injectivity. Suppose that $[\phi]=[\psi]$ in $[X,\End{A}^{*}]$ for $[\phi],[\psi]\in [X,\mathrm{Aut}_{G}(A)]$. Then there is a homotopy $H:X\times [0,1]\to \End{A}^{*}$ such that $H_{0}=\phi$ and $H_{1}=\psi$. Since $(A,X\times [0,1])$ is also a $KK^{G}$-continuous pair, there exists $H':X\times [0,1]\to\mathrm{Aut}_{G}(A)$ such that $[H']=[H]$ in $[X\times[0,1],\End{A}^{*}]$. Thus we get a homotopy between $\phi$ and $\psi$ in $\mathrm{Map}(X,\mathrm{Aut}_{G}(A))$ and injectivity follows.
\end{proof}
\begin{corollary}\label{cor:stbl-htpy-set}
Let $G$ be a compact group and let $A$ be a stable Kirchberg algebra with isometrically shift-absorbing action $\Ga$. Let $X$ be a compact metrizable space. Suppose that $(A,X)$ is $KK^{G}$-continuous. Set
\begin{align*}
&KK^{G}(A,C(X)\otimes A)^{*}\\
&=\{x\in KK^{G}(A,C(X)\otimes A)\:|\: (\mathrm{ev}_{x})_{*}(x)\text{ is invertible in }KK^{G}(A,A)\text{ for all }x\in X\}.
\end{align*}
Then $[X,\mathrm{Aut}_{G}(A)]\ni[\phi]\mapsto KK^{G}(\phi)\in KK^{G}(A,C(X)\otimes A)^{*}$ is a group isomorphism.
\end{corollary}
\begin{remark}\label{rem:inv}
Note that we can take a homotopy $H$ from $\Phi_{0}$ to $\tilde{\Psi}$ in the proof of \autoref{prop:im-inv} satisfying $H_{t}\in \mathrm{Aut}_{C(X),G}(C(X)\otimes A)$ for all $0\leq t<1$. Hence for $\phi:X\to \mathrm{Aut}_{G}(A)$ and $\psi:X\to\mathrm{End}_{G}(A)^{*}$ such that $[\phi]=[\psi]$ in $[X,\mathrm{End}_{G}(A)]$, there exists a homotopy $H:X\times [0,1]\to \mathrm{End}_{G}(A)^{*}$ such that $H_{0}=\phi$, $H_{1}=\psi$ and $H_{t}(x)\in \mathrm{Aut}_{G}(A)$ for all $x\in X$ and $0\leq t<1$.
\end{remark}
\begin{corollary}[cf. {\cite[Theorem 5.9]{D1}}]\label{cor:htpy-grps}
Let $G$ be a compact group and let $X$ be a path connected compact metrizable space and $x_{0}\in X$. Let $A$ be a Kirchberg algebra and let $\Ga:G\act A$ be an isometrically shift-absorbing action. Suppose that $(A,X)$ is a $KK^{G}$-continuous pair.\\
(i) : If $A$ is unital, 
\[\Gc_{X}:[X,\mathrm{Aut}_{G}(A)^{0}]\to KK^{G}(\cna,SC\ptx\otimes A)\]
is bijective.\\
(ii) : If $A$ is stable,
\[\overline{\Gc}_{X}:[X,\mathrm{Aut}_{G}(A)^{0}]\to KK^{G}(A,C\ptx\otimes A)\]
is bijective.
Moreover, if $\ptx$ is a co-$H$-space, they are group isomorphisms. In particular, we have 
\[\pi_{n}(\mathrm{Aut}_{G}(A))\cong\begin{cases}
KK^{G}(\cna, S^{n+1}A) & A\text{ is unital},\\
KK^{G}(A,S^{n}A) & A\text{ is stable}.
\end{cases}\]
for $n\geq 1$. For $n=0$, we have
\[\pi_{0}(\mathrm{Aut}_{G}(A))\cong \begin{cases}
(KK^{G}(\cna,SA)^{-1},\:\circ\:)& A\text{ is unital},\\
KK^{G}(A,A)^{-1}& A\text{ is stable}.
\end{cases}\]
Here, the group structure on $KK^{G}(A,A)^{-1}$ comes from the multiplication on $KK^{G}(A,A)$. The group structure on $KK^{G}(\cna,SA)$ is given by
\[x\circ y=x+y+y\otimes KK^{G}(\iota_{A})\otimes x\]
where $\Gi_{A}$ is the inclusion $SA\to \cna$.
\end{corollary}
\begin{proof}
The assertions other than $\pi_{0}$ follows from \autoref{thm:end-grp-str} and \autoref{cor:auto-end}. By \autoref{cor:stbl-htpy-set}, the stable case follows. For the unital case, since $[\mathrm{pt},\mathrm{Aut}_{G}(A)]=[\mathrm{pt},\mathrm{End}_{G}(A)^{*}]=[\mathrm{pt},\mathrm{End}_{G}]^{-1}$, it suffices to show that the bijection between $[\mathrm{pt},\mathrm{End}_{G}(A)]$ and $KK^{G}(\cna,SA)$ preserves the multiplicative strucures. Let $[\phi],[\psi]\in[\mathrm{pt},\mathrm{End}_{G}]$ be arbitrary elements. Then, the direct computation yields
\begin{align*}
[(C_{\nu}(\phi\circ\psi)&,C_{\nu} j)]\\
&=[(C_{\nu}(\phi\circ\psi),C_{\nu}\phi)]+[(C_{\nu}\phi,C_{\nu}j)]\\
&=[(C_{\nu}\psi,C_{\nu}j)]\otimes \{KK^{G}(\id{S})\hotimes KK^{G}(\phi)\}+[(C_{\nu}\phi,C_{\nu}j)]\\
&=[(C_{\nu}\psi,C{\nu}j)]\otimes \{KK^{G}(\id{S})\hotimes (KK^{G}(\phi)-KK^{G}(\id{C(X)\otimes A}))\}\\
&\qquad\qquad+[(C_{\nu}\phi,C_{\nu}j)]+[(C_{\nu}\psi,C_{\nu}j)]\\
&=[(C_{\nu}\psi,C_{\nu}j)]\otimes KK^{G}(\Gi_{A})\hotimes [(C_{\nu}\phi,C_{\nu}j)]+[(C_{\nu}\phi,C_{\nu}j)]+[(C_{\nu}\psi,C_{\nu}j)]\\
\end{align*}
\end{proof}
\begin{example}
Let $\pi:G\to \Ucal(\Cbb^{n})$ be a faithful representation of a compact group such that every irreducible representation of $G$ is equivalent to a subrepresentation of $\pi^{\otimes m}$ for some $m\geq 0$. Then we can construct a quasi-free action $\Ga^{\pi}:G\act \on$. In order to compute the homotopy group of $\mathrm{Aut}_{G}(\on)$, first we show the mapping cone of the unital inclusion $\Gv:\Cbb\to \on$ is $KK^{G}$-equivalent to $\Cbb$.

Let $0\to \Kbb\to E_{n}\xrightarrow{q}\on\to 0$ be the Toeplitz extension and consider the following commutative diagram.
\[
\begin{tikzcd}
S\on \arrow[r] \arrow[d, equal]
  & \Kbb \arrow[r]
  & E_n \arrow[r, "q"] \arrow[d, equal]
  & \on \arrow[d, equal] \\
S\on \arrow[r] \arrow[d, equal]
  & \mathrm{Cone}\, q \arrow[r]
  & E_n \arrow[r, "q"] \arrow[d,phantom, "\rotatebox{90}{$\cong$}"]
  & \on \arrow[d, equal] \\
S\on \arrow[r]
  & \mathrm{Cone}\, \nu \arrow[r]
  & \Cbb \arrow[r, "\nu"]
  & \on
\end{tikzcd}
\]
The Toeplitz extension has a completely positive, contractive equivariant cross section. Thus $\Kbb$ is $KK^{G}$-equivalent to $\mathrm{Cone}\ q$. (see \cite{MN}). Since $KK^{G}$ category is triangulated, $\mathrm{Cone}\ q$ and $\mathrm{Cone}\ \nu$ are $KK^{G}$-equivalent. Hence $\mathrm{Cone}\ \nu$ is $KK^{G}$-equivalent to $\Cbb$. By using this $KK^{G}$-equivalence, we have
\begin{align*}
\pi_{k}(\mathrm{Aut}_{G}(\on))\cong KK^{G}(\mathrm{Cone}\ \nu,S^{n+1}\on)\cong KK^{G}(\Cbb,S^{n+1}\on)\cong K^{G}_{n+1}(\on).
\end{align*} 
Using the Toeplitz extension and six-term exact sequence, we have 
\[0\to K^{G}_{1}(\on)\to R(G)\xrightarrow{1-[\pi]}R(G)\to K^{G}_{0}(\on)\to 0.\]
Hence we have 
\[\pi_{k}(\mathrm{Aut}_{G}(\on))\cong\begin{cases}
\coker (1-[\pi]) & k\text{ : odd},\\
\ker (1-[\pi]) & k\geq 2\text{ : even}.
\end{cases}\]
\end{example}

\begin{example}\label{ex:otw}
Let $p$ be a prime number and set $G=\mathbb{Z}/p\mathbb{Z}$. Let $\Gg:G\act \mathcal{O}_{2}$ be an outer action. Suppose that $(\otw,\Gg)$ belongs to the equivariant bootstrap class in $KK^{G}$ as in \cite{MN2}. Let $\nu:\Cbb\to\otw$ be the unital inclusion and let $C_{\nu}$ be its mapping cone. Since $\Cbb$ and $\otw$ belong to the equivariant bootstrap class, so does $\cn$. 

Recall that the representation ring $R(G)$ is isomorphic to $\Zbb[t]/(t^{p}-1)$. We also have an isomorphism $\zzp\cong R(G)/\Zbb N(t)$, where $\zeta_{p}$ is a primitive $p$-th root of unity and $N(t)=1+t+\cdots+t^{p-1}$ is the norm element. Write $R = \Zbb[\zeta_{p}, 1/p]$ and $\GG_{*}=K_{*}^{G}(\otw)$ for $*=0,1$. By \cite[Theorem 7.2]{Me}, $\GG_{0}$ and $\GG_{1}$ are not only $R(G)$-modules but also $\Zbb[\zeta_{p}]$-modules. Furthermore, \cite[Lemma 4.4]{Iz2} shows that they are uniquely $p$-divisible, and hence they are $R$-modules. Note that it follows from \cite[Theorem 5.2]{MN2} and \cite[Corollary 6.4]{GS2} that $(\Gamma_{0},\Gamma_{1},[1_{\otw}])$ is a complete invariant of an outer action on $\otw$ that belongs to the equivariant bootstrap class in $KK^{G}$ up to conjugacy. 

We will compute $KK^{G}_{*}(\cn,\otw)\cong KK^{G}_{*}(\cn\otimes M_{p^{\infty}},\otw)$ (\cite[Proposition 4.3]{MN2}). We will achieve this by applying the Universal Coefficient Theorem given in \cite[Theorem 5.2]{MN2}.
For a $G$-\cstar algebra $B$, let 
\begin{align*}
  &F_{*}^{\{e\}}(B)\coloneqq K_{*}(B),\\
  &F_{*}^{G}(B)\coloneqq \{x\in KK^{G}_{*}(C(G/H),B)\:|\: N(t)\cdot x=0\}.
\end{align*}
It is known that $F^{G}_{*}$ is a homological functor on the full subcategory of $KK^{G}$ consisting of objects that have uniquely $p$-divisible $K$-theory. Observe that $F_{*}^{\{e\}}(\otw)=0$ for $*=0,1$. Therefore, the computation reduces to the following exact sequence : 
\[\mathrm{Ext}^{1}_{R}(F^{G}_{*}(\cn\otimes M_{p^{\infty}}),F_{*-1}^{G}(\otw))\rightarrowtail KK_{*}^{G}(\cn,\otw)\twoheadrightarrow \mathrm{Hom}_{R}(F_{*}^{G}(\cn\otimes M_{p^{\infty}}),F_{*}^{G}(\otw)).\]
The even and odd parts of a $\Zbb/2\Zbb$-graded $R$-module of the form $F_{*}(A)$ can be considerd as individual $R$-modules. There exist uniquely $p$-divisible $C^*$-algebras in the bootstrap class that realize them, and the direct sum of these algebras is $KK^{G}$-equivalent to the original $C^*$-algebra. Thanks to this decomposition, the exact sequence splits unnaturally, as in the usual UCT.\footnote{The author thanks Ralf Meyer for explaining this argument.}

First, we compute $F^{G}_{*}(\cn\otimes M_{p^{\infty}})$. 
Because $F_{*}^{G}(\otw)=K_{*}^{G}(\otw)$ is uniquely $p$-divisible, $K_{*}^{G}(\otw\otimes M_{p^{\infty}})=\GG_{*}$ holds. Applying the functor $KK^{G}(\Cbb,-)$ to the mapping cone sequence $C_{\nu}\otimes M_{p^{\infty}}\to M_{p^{\infty}}\to\mathcal{O}_{2}\otimes M_{p^{\infty}}$ yields the following exact sequence:
\[0\to \GG_{1}\to K_{0}^{G}(\cn\otimes M_{p^{\infty}})\to \Zbb[1/p]\oplus\Zbb[\zeta_{p},1/p]\to \GG_{0}\to K_{1}^{G}(\cn\otimes M_{p^{\infty}})\to 0.\]
Since $N(t)$ acts as $0$ on any $\Zbb[\zeta_{p}]$-module, we have
\[0\to \GG_{1}\to F_{0}^{G}(\cn\otimes M_{p^{\infty}})\to R\to \GG_{0}\to F_{1}^{G}(\cn\otimes M_{p^{\infty}})\to 0.\]
Here, the map $R\to \GG_{0}$ is a $R$-module map sending $1$ to $[\one_{\otw}]$. Let $J$ denote the kernel of this map. Then we have $F_{1}^{G}(\cn\otimes M_{p^{\infty}})\cong\GG_{0}/\langle[\one_{\otw}]\rangle$, where $\langle[\one_{\otw}]\rangle$ is the $R$-module generated by $[\one_{\otw}]$. We also obtain the following exact sequence of $R$-modules.
\[0\to \GG_{1}\to F_{0}^{G}(\cn\otimes M_{p^{\infty}})\to J\to 0.\]
Since $R$ is a Dedekind domain, $J$ is a finitely generated projective $R$-module by \cite[Example 15.71.3]{Sp}. Consequently, this exact sequence splits, yielding $F_{0}^{G}(\cn\otimes M_{p^{\infty}})\cong \GG_{1}\oplus J$. Putting everything together, we obtain:
\begin{align*}
&F_{0}^{G}(\otw)=\Gamma_{0},\quad F_{1}^{G}(\otw)=\Gamma_{1},\\
&F_{0}^{G}(\cn\otimes M_{p^{\infty}})\cong \Gamma_{1}\oplus J,\quad F_{1}^{G}(\cn\otimes M_{p^{\infty}})\cong \Gamma_{0}/\langle[\one_{\otw}]\rangle.
\end{align*}
Next, we turn to the computation of $KK^{G}_{*}(\cn,\otw)$. Since $J$ is a finitely generated projective $R$-module, it follows from \cite[Lemma 10.78.9]{Sp} that
\[\mathrm{Ext}_{R}^{1}(J,M)=0,\quad \Hom_{R}(J,M)\cong \Hom_{R}(J,R)\otimes_{R}M\]
for any $R$-module $M$. Suppose that $J$ is nonzero. Then, since $J$ is an invertible fractional ideal of $R$, there is an isomorphism $\Hom_{R}(J,R)\cong J^{-1}$. Here, $J^{-1} \coloneqq \{x\in K\:|\: xJ\subset R\}$ denotes the inverse fractional ideal of $J$, where $K$ is the field of fractions of $R$. Since $J=\mathrm{Ann}_{\mathbb{Z}[\zeta_{p},1/p]}[\one_{\otw}]\coloneqq \{x\in \Zbb[\zeta_{p},1/p]\:|\: x[\one_{\otw}]=0\}$ is finitely generated, we have $J^{-1}\cong (\mathrm{Ann}_{\zzp}[\one_{\otw}])^{-1}\otimes_{\zzp}\mathbb{Z}[\zeta_{p},1/p]$ by \cite[Proposition 3.14]{AM2}. This implies $J^{-1}\otimes_{R}\GG_{*}\cong (\mathrm{Ann}_{\zzp}[\one_{\otw}])^{-1}\otimes_{\zzp}\GG_{*}$. By \autoref{cor:htpy-grps}, we obtain
\begin{align*}
\pi_{2n-1}(\mathrm{Aut}_{G}(\otw))&\cong KK^{G}_{0}(\cn,\otw)\\
&\cong \Hom_{R}(\Gamma_{1},\Gamma_{0})\oplus((\mathrm{Ann}_{\zzp}[\one_{\otw}])^{-1}\otimes_{\zzp}\Gamma_{0})\oplus \Hom_{R}(\Gamma_{0}/\langle[\one_{\otw}]\rangle,\Gamma_{1})\\
&\oplus\Ext_{R}(\Gamma_{1},\Gamma_{1})\oplus\Ext_{R}(\Gamma_{0}/\langle[\one_{\otw}]\rangle,\Gamma_{0}),\\
\pi_{2n}(\mathrm{Aut}_{G}(\otw))&\cong KK^{G}_{1}(\cn,\otw)\\
&\cong\Hom(\Gamma_{1},\Gamma_{1})\oplus((\mathrm{Ann}_{\zzp}[\one_{\otw}])^{-1}\otimes_{\zzp}\Gamma_{1})\oplus\Hom_{R}(\Gamma_{0}/\langle[\one_{\otw}]\rangle,\Gamma_{0})\\
&\oplus\Ext_{R}(\Gamma_{1},\Gamma_{0})\oplus \Ext_{R}(\Gamma_{0}/\langle[\one_{\otw}]\rangle,\Gamma_{1})
\end{align*}
for all $n\geq 1$. In the case $J=0$, the term $(\mathrm{Ann}_{\zzp}[\one_{\otw}])^{-1}\otimes_{\zzp}\Gamma_{*}$ vanishes.

\end{example}
\begin{example}\label{ex:oinf}
Let $p$ be a prime number and set $G=\Zbb/p\Zbb$. Let $\Gg:G\act \oinf$ be an outer action that belongs to the equivariant bootstrup class. Let $\cn$ be the mapping cone of the unital inclusion $\nu:\Cbb\to \oinf$. Write 
\[\widetilde{\Gamma_{0}}=K_{0}^{G}(\oinf),\quad \GG_{0}\coloneqq \widetilde{\GG_{0}}/\langle [\one_{\oinf}]  \rangle,\quad\Gamma_{1}=K^{G}_{1}(\oinf).\]
By considering the six-term exact sequence of equivariant $K$-theory, we have the following exact sequence : 
\begin{center}
\begin{tikzpicture}[auto]
\node(1) at (-5,1) {$   K_{0}^{G}(\cn)  $};
\node(2) at (0,1) {$   K_{0}^{G}(\Cbb)\cong R(G)  $};
\node(3) at (5,1) {$  K_{0}^{G}(\oinf)\cong \widetilde{\Gamma_{0}}   $};
\node(4) at (5,-1) {$  K_{1}^{G}(\cn)   $};
\node(5) at (0,-1) {$   K_{1}^{G}(\Cbb)=0  $};
\node(6) at (-5,-1) {$  K_{1}^{G}(\oinf)\cong \Gamma_{1}   $};
\draw[->] (1) to  (2);
\draw[->] (2) to node {$\nu_{*}$} (3);
\draw[->] (3) to  (4);
\draw[->] (4) to  (5);
\draw[->] (5) to  (6);
\draw[->] (6) to (1);
\end{tikzpicture}
\end{center}
The map $\nu_{*}$ sends $1\in R(G)$ to $[\one_{\oinf}]$. Hence we have
\[K_{0}^{G}(\cn)\cong \Gamma_{1},\quad K_{1}^{G}(\cn)\cong \GG_{0}.\]
Since $\nu:\Cbb\to \oinf$ induces $KK$-equivalence between $\Cbb$ and $\oinf$, we have $KK^{G}_{*}(C(G),\oinf)\cong K_{*}(\cn)=0$ for $*=0,1$. 
Since the multiplication map $N(t):K_{*}^{G}(\cn)\to K_{*}^{G}(\cn)$ factors through $K_{*}(\cn)=0$ by \cite[Lemma 11.15]{Ko}, $\GG_{0}$ and $\GG_{1}$ are $\zzp$-modules. Moreover, it follows from \cite[Corollary 5.3]{IO} that $\GG_{0}$ and $\GG_{1}$ are uniquely $p$-divisible, and hence they are $\mathbb{Z}[\zeta_{p},1/p]$-modules.

We will compute $KK^{G}_{*}(\cn,\oinf)$. To achieve this, we will employ K\"{o}hler's Universal Coefficient Theorem \cite{Ko}. 

Before stating the exact sequence, we briefly recall the necessary ingredients of the UCT. Let $\alpha$ denote the action of $G$ on $C(G)$ given by left translation, and let $u: \Cbb \to C(G)$ be the unital embedding. We denote by $\cu$ the mapping cone of $u$. The UCT is formulated over the $\Zbb/2\Zbb$-graded ring
\[\Rcal\coloneqq KK^{G}_{*}(\Cbb\oplus C(G)\oplus\cu,\Cbb\oplus C(G)\oplus\cu).\]
For any $G$-\cstar-algebra $B$, we define the $\Rcal$-module
\begin{align*}
  EK^{G}(B)&=KK^{G}_{*}(\Cbb\oplus C(G)\oplus\cu,B)\\
  &\cong K^{G}_{*}(B)\oplus K_{*}(B)\oplus KK^{G}_{*}(\cu,B).
\end{align*}
Here we use isomorphisms $KK^{G}_{*}(\Cbb,B)\cong K_{*}^{G}(B)$ and $KK^{G}_{*}(C(G),B)\cong K_{*}(B)$.
Applying the UCT to our setting yields the following natural short exact sequence :
\begin{align*}
  \Ext_{\Rcal}(EK^{G}(S\cn),EK^{G}(\oinf))\rightarrowtail KK_{*}^{G}(\cn,\oinf)\twoheadrightarrow \Hom_{\Rcal}(EK^{G}(\cn),EK^{G}(\oinf)).
\end{align*}
First, we compute $EK^{G}(\oinf)$. We know $KK^{G}_{0}(C(G),\oinf)\cong\Zbb$ and $KK^{G}_{1}(C(G),\oinf)=0$. Applying the functor $KK^{G}(-,\oinf)$ to the mapping cone sequence $\cu\to\Cbb\to C(G)$ yields the following exact sequence :
\begin{center}
\begin{tikzpicture}[auto]
\node(1) at (-5,1) {$   KK^{G}_{0}(\cu,\oinf)  $};
\node(2) at (0,1) {$   KK^{G}_{0}(\Cbb,\oinf)\cong\widetilde{\Gamma_{0}}  $};
\node(3) at (5,1) {$    KK^{G}_{0}(C(G),\oinf)\cong\mathbb{Z} $};
\node(4) at (5,-1) {$   KK^{G}_{1}(\cu,\oinf)  $};
\node(5) at (0,-1) {$   KK^{G}_{1}(\Cbb,\oinf)\cong\Gamma_{1}  $};
\node(6) at (-5,-1) {$   KK^{G}_{1}(C(G),\oinf)= 0  $};
\draw[->] (1) to  (6);
\draw[->] (2) to  (1);
\draw[->] (3) to node {$u^{*}$} (2);
\draw[->] (4) to  (3);
\draw[->] (5) to  (4);
\draw[->] (6) to  (5);
\end{tikzpicture}
\end{center}
Consider the map $u^{*}:KK^{G}_{0}(C(G),\oinf)\to KK^{G}_{0}(\Cbb,\oinf)$. Because $[(L^{2}(G,\oinf),m,0)]$ is a generator of $KK^{G}_{0}(C(G),\oinf)$, $u^{*}$ maps $1\in \Zbb$ to $N(t)[\one_{\oinf}]\in \widetilde{\Gamma_{0}}$. Hence $u^{*}$ is injective and we obtain 
\[KK^{G}_{0}(\cu,\oinf)\cong \widetilde{\Gamma_{0}}/\langle N(t)[\one_{\oinf}]\rangle,\quad KK^{G}_{1}(\cu,\oinf)\cong \Gamma_{1}.\]

Next, we compute $EK^{G}(\cn)$. We already have
\[K_{*}(\cn)=0\quad (\text{$*$}=0,1),\quad K_{0}^{G}(\cn)\cong \GG_{1},\quad K_{1}^{G}(\cn)\cong \GG_{0}.\]
Applying the functor $KK^{G}(-,\cn)$ to the mapping cone sequence $\cu\to\Cbb\to C(G)$ and using these results yields
\[KK^{G}_{0}(\cu,\cn)\cong \Gamma_{1},\quad KK^{G}_{0}(\cu,\cn)\cong \GG_{0}.\]
Summarizing the calculations so far, we have the following :
\begin{align*}
&EK^{G}(\oinf)\cong (\widetilde{\Gamma_{0}}\oplus\Zbb\oplus\widetilde{\Gamma_{0}}/\langle N(t)[\one_{\oinf}]\rangle,\;\Gamma_{1}\oplus 0 \oplus \Gamma_{1}),\\
&EK^{G}(\cn)\cong (\Gamma_{1}\oplus 0\oplus \Gamma_{1},\;\GG_{0}\oplus 0\oplus\GG_{0}).
\end{align*}
To further investigate these structures, we use \cite[Proposition 10.5]{Me}. Our calculation yields that 
\[0\to EK^{G}(\Cbb)\to EK^{G}(\oinf)\to EK^{G}(\cn)\to 0\]
is an exact sequence of exact $\Rcal$-modules. This extension appears in \cite[Theorem 9.1]{Me}, and $EK^{G}(\Cbb)$ is the exact $\Rcal$-module given in \cite[Example 9.2]{Me}. Note that the element $s\in KK^{G}_{0}(C(G),C(G))\cong \Zbb[s]/(s^{p}-1)$ acts as the identity on $KK^{G}_{0}(C(G),\Cbb)\cong KK^{G}_{0}(C(G),\oinf)$. Furthermore, the map $\Ga_{10}=[(L^{2}(G),m,0)]\otimes -:KK^{G}_{*}(\Cbb,\oinf)\to KK^{G}_{*}(C(G),\oinf)$ is surjective, because it sends $[\one_{\oinf}]\in K_{0}^{G}(\oinf)\cong KK^{G}_{0}(\Cbb,\oinf)$ to $[\one_{\oinf}]\in K_{0}(\oinf)\cong KK^{G}_{0}(C(G),\oinf)$. Consequently, it follows that both $EK^{G}(\oinf)$ and $EK^{G}(\cn)$ decompose into direct sums of their even and odd parts as $\Rcal$-modules. While it is unknown whether K\"{o}hler's UCT exact sequence splits in general, it does so in our case because the modules decompose into direct sums of their even and odd parts. 

Finally, we compute the $\mathrm{Hom}$-groups and the $\mathrm{Ext}$-groups to determine $KK^{G}_{*}(\cn,\oinf)$. To simplify notation, when we write $\begin{pmatrix} M_{0} \\ M_{1} \\ M_{2} \end{pmatrix}$, we mean the $\mathcal{R}$-module whose even part is given by this and whose odd part is $0$. Furthermore, we write 
\begin{align*}
EK^{G}_{0}=\begin{pmatrix}\widetilde{\GG_{0}}\\ \Zbb\\ \widetilde{\GG_{0}}/\langle N(t)[\one_{\oinf}]\rangle\end{pmatrix}, EK^{G}_{1}=\begin{pmatrix}\GG_{1}\\ 0\\ \GG_{1}\end{pmatrix}, EK^{G}_{i}(\cn)=\begin{pmatrix}\GG_{1-i}\\ 0\\ \GG_{1-i}\end{pmatrix}
\end{align*}
for $i=0,1$. It is clear that
\begin{align*}
\mathrm{Hom}_{\Rcal}(EK^{G}_{i}(\cn),EK^{G}_{0}(\oinf))&\cong\mathrm{Hom}_{\zzp}(\GG_{1-i},\widetilde{\GG_{0}}/\langle  N(t)[\one_{\oinf}] \rangle)\\
\mathrm{Hom}_{\Rcal}(EK^{G}_{i}(\cn),EK^{G}_{1}(\oinf))&\cong \mathrm{Hom}_{\zzp}(\GG_{1-i},\GG_{1})
\end{align*}
as groups for $i=0,1$.\\
We now compute $\Ext_{\Rcal}(EK^{G}_{i}(\cn),EK^{G}_{k}(\oinf))$. Since the argument is similar for other cases, we only provide the details for the case $i=1$. Let 
\[0\to P'\xrightarrow{\Gi'}\oplus_{i}\zzp\to \GG_{0}\]
be a projective resolution of $\GG_{0}$ over $\zzp$. Since $\GG_{0}$ is a $\zzp$-module, this projective resolution induces the following resolution over $\Rcal$.
\[0\to \widetilde{P}\coloneqq\begin{pmatrix}
P\\ \oplus_{i}\mathbb{Z}\\ P'\\
\end{pmatrix}\to \oplus_{i} EK^{G}(\Cbb)=\oplus_{i}\begin{pmatrix}
R(G)\\ \Zbb\\ \zzp\\
\end{pmatrix}\to EK^{G}_{1}=\begin{pmatrix}
\GG_{0}\\ 0 \\ \GG_{0}
\end{pmatrix},\]
where $P$ is the kernel of the map $\oplus_{i}R(G)\to\oplus_{i}\zzp\to M$. By \cite[Theorem 12.4]{Ko}, $\widetilde{P}$ and $\oplus_{i} EK^{G}_{0}(\Cbb)$ are projective $\Rcal$-modules. Focusing on the first components, it is obvious that $\oplus_{i}R(G)$ is a projective $R(G)$-module, and by \cite[Corollary 8.5]{Me}, $P$ is also a projective $R(G)$-module. Hence we obtain a projective resolution of $\GG_{0}$ over $R(G)$ : 
\[0\to P\to \oplus_{j}R(G)\to \GG_{0}.\]
Since the $\Rcal$-module structure implies 
\begin{align*}
\Hom_{\Rcal}(\widetilde{P},EK^{G}_{1}(\oinf))&\cong\Hom_{\zzp}(P',\GG_{1}),\\
\Hom_{\Rcal}(\oplus_{i}EK^{G}(\Cbb),EK^{G}_{1}(\oinf))&\cong\Hom_{\zzp}(\oplus_{i}\zzp,\GG_{1}),
\end{align*}
we obtain
\begin{align*}
\Ext_{\Rcal}(EK^{G}_{1}(\cn),EK^{G}_{1}(\oinf))\cong \Ext_{\zzp}(\GG_{0},\GG_{1}).
\end{align*}
\begin{claim*}
\[\Ext_{\Rcal}(EK^{G}_{1}(\cn),EK^{G}_{0}(\oinf))\cong \Ext_{\zzp}(\GG_{0},\widetilde{\GG_{0}}/\langle N(t)[\one_{\oinf}]  \rangle)\quad\text{as groups}.\]
\end{claim*}
\begin{proof}
Let $\Gi:P\to \oplus_{i}R(G)$ and $\Gi':P'\to \oplus_{i}\zzp$ be inclusion maps and let $\pi_{P}:P\to P'$, $\pi:\oplus_{i}R(G)\to \oplus_{i}\zzp$ and $\pi_{0}:\widetilde{\GG_{0}}\to\widetilde{\GG_{0}}/\langle N(t)[\one_{\oinf}]\rangle$ be the quotient maps. Since the $\Rcal$-module structure implies 
\begin{align*}
&\Hom_{\Rcal}(\widetilde{P},EK^{G}_{0}(\oinf))\cong\{\phi\in \Hom_{R(G)}(P,\widetilde{\GG_{0}})\:|\:  \phi(\ker\pi_{P})\subset \ker\pi_{0}\}\eqqcolon M,\\
&\Hom_{\Rcal}(\oplus_{i}EK^{G}(\Cbb),EK^{G}_{0}(\oinf))\cong\Hom_{R(G)}(\oplus_{i}R(G),\widetilde{\GG_{0}}),
\end{align*}
we have
\[\Ext_{\Rcal}(EK^{G}_{1}(\cn),EK^{G}_{0}(\oinf))\cong M/(\Gi)^{*}\left(\Hom_{R(G)}(\oplus_{i}R(G),\widetilde{\GG_{0}})\right).\]

We now show that the right-hand side is isomorphic to $\Ext_{\zzp}(\GG_{0},\widetilde{\GG_{0}}/\langle N(t)[\one_{\oinf}]\rangle)$. Let $\phi\in M$ be an arbitrary element. Because the kernerl of $\pi_{0}\circ \phi:P\to\widetilde{\GG_{0}}/\langle N(t)[\one_{\oinf}]\rangle$ contains $\ker\pi_{P}$, we obtain a $\zzp$-module map $\overline{\phi}:P'\to \widetilde{\GG_{0}}/\langle N(t)[\one_{\oinf}]\rangle)$ satisfying $\pi_{0}\circ \phi=\overline{\phi}\circ \pi_{P}$. Define $\Phi:\Hom_{R(G)}(P,\GG_{0})\to \Hom_{\zzp}(P',\widetilde{\GG_{0}}/\langle N(t)[\one_{\oinf}]\rangle)$ by $\Phi(\phi)=\overline{\phi}$. Then $\Phi$ is a group homomorphism and it is easy to see that if $\phi\in\im\Gi^{*}$, then $\overline{\phi}\in\im (\Gi')^{*}$. Thus, $\Phi$ induces a group homomorphism
\[\overline{\Phi}:M/\im\Gi^{*}\to \Hom_{\zzp}(P',\widetilde{\GG_{0}}/\langle N(t)[\one_{\oinf}]\rangle)/\im(\Gi')^{*}=\Ext_{\zzp}(\GG_{0},\widetilde{\GG_{0}}/\langle N(t)[\one_{\oinf}]\rangle).\]
Let $f\in M$ be any element satisfying $\overline{\Phi}([f])=[\overline{f}]=0$. Then there exists $g:\oplus_{i}\zzp\to\widetilde{\GG_{0}}/\langle N(t)[\one_{\oinf}]\rangle$ such that $g\circ\Gi'=\overline{f}$. Let $\one_{j}\in\oplus_{i}\zzp$ be the element with $1$ in the $j$-th component and $0$ elsewhere. By choosing a lift for each $g(\one_{j})$, we obtain $\tilde{g}:\oplus_{i}R(G)\to\GG_{0}$. Then we have $\Phi(f-\tilde{g}\circ \Gi)=0$. This implies that the image of $h\coloneqq f-\tilde{g}\circ\Gi$ is contained in $\Zbb N(t)[\one_{\oinf}]$. By the $\Rcal$-module structure of $\widetilde{P}$ and $\oplus_{i}EK^{G}(\Cbb)$, we have the following commutative diagram : 
\begin{center}
\begin{tikzpicture}[auto]
\node (1) at (-3,1) {$\Hom_{R(G)}(\oplus_{i}R(G),\Zbb N(t)[\one_{\oinf}])$};
\node (2) at (3,1) {$\Hom_{R(G)}(P,\Zbb N(t)[\one_{\oinf}])$};
\node (3) at (-3,-1) {$\Hom_{\Zbb}(\oplus_{i}\Zbb,\Zbb)$};
\node (4) at (3,-1) {$\Hom_{\Zbb}(\oplus_{i}\Zbb,\Zbb)$};
\draw[->] (1) to node {$\Gi^{*}$}  (2);
\draw[->] (3) to node {$\id{}$}  (4);
\draw[double, double distance=2pt, -] (1) to (3);
\draw[double, double distance=2pt, -] (2) to (4);
\end{tikzpicture}
\end{center}
Therefore $h$ is in the image of $\Gi^{*}$, so does $f=h+\tilde{g}\circ \Gi$. Hence $\overline{\Phi}$ is injective.

Let $g\in\Hom_{\zzp}(P',\widetilde{\GG_{0}}/\langle  N(t)[\one_{\oinf}] \rangle)$ be any element. Then $g\circ\pi_{P}:P\to \widetilde{\GG_{0}}/\langle  N(t)[\one_{\oinf}] \rangle$ is a $R(G)$-module map. Since $P$ is a projective $R(G)$-module, there exists a $R(G)$-module map $\tilde{g}:P\to \GG_{0}$ such that $\pi_{0}\circ \tilde{g}=g\circ\pi_{P}$. By construction, we have $\tilde{g}(\ker\pi_{P})\subset \ker\pi_{0}$ and $\Phi(\tilde{g})=g$. Hence $\overline{\Phi}$ is surjective, and we can conclude that $\overline{\Phi}$ is a group isomorphism.
\end{proof}
Then we have 
\begin{align*}
\pi_{2n-1}(\mathrm{Aut}_{G}(\oinf))&\cong KK^{G}_{0}(\cn,\oinf)\\ 
&\cong \Hom_{\zzp}(\GG_{1},\widetilde{\GG_{0}}/\langle  N(t)[\one_{\oinf}] \rangle)\oplus\Hom_{\zzp}(\GG_{0},\GG_{1})\\
&\quad \oplus\Ext_{\zzp}(\GG_{1},\GG_{1})\oplus\Ext_{\zzp}(\GG_{0},\widetilde{\GG_{0}}/\langle  N(t)[\one_{\oinf}] \rangle)
\end{align*}
and
\begin{align*}
\pi_{2n}(\mathrm{Aut}_{G}(\oinf))&\cong KK^{G}_{1}(\cn,\oinf)\\
&\cong \Hom_{\zzp}(\GG_{1},\GG_{1})\oplus\Hom_{\zzp}(\GG_{0},\widetilde{\GG_{0}}/\langle  N(t)[\one_{\oinf}] \rangle)\\
&\quad \oplus\Ext_{\zzp}(\GG_{1},\widetilde{\GG_{0}}/\langle  N(t)[\one_{\oinf}] \rangle)\oplus\Ext_{\zzp}(\GG_{0},\GG_{1})
\end{align*}
for all $n\geq 1$.
\end{example}

\section{Equivariant version of Dadarlat-Pennig theory for Kirchberg algebras}
In this section, we develop an equivariant analogue of the Dadarlat-Pennig theory. 
\begin{theorem}[{\cite[Theorem A]{DP}}]
  Let $X$ be acompact metrizable space and let $A$ be a strongly self-absrobing \cstar-algebra. The set $\mathcal{B}un_{X}(A\otimes \Kbb)$ of isomorphism classes of locally trivial fields over $X$ with fiber $A\otimes \Kbb$ becomes an abelian group under the operation of tensor product. Moreover, this group is isomorphic to $E_{A}^{1}(X)$, the first group of a generalized connective cohomology theory $E_{A}^{*}(X)$ defined by the infinite loop space $B\Aut{}{A\otimes \Kbb}$.
\end{theorem}
  Our goal is to extend this result to the equivariant setting. In doing so, our approach closely follows the methods developed by Dadarlat and Pennig in \cite{DP} and Evans and Pennig in \cite{EP}. 
\subsection{Strongly self-absorbing actions}
\begin{definition}[{\cite[Definition 3.1]{SZ1}}]
Let $D$ be a unital, separable \cstar-algebra. Let $\Gg$ be an action of a compact group $G$ on $D$. We say that $\Gg$ is a \textit{strongly self-absorbing action} if there exist an equivariant isomorphism $\phi:D\to D\otimes D$ and a sequence $(v_{n})$ in $\Ucal((D\otimes D)^{\Gg\otimes \Gg})$ such that $x\otimes \mathbf{1}_{D}=\lim_{n\to\infty}\ad{v_{n}}\circ \phi(x)$ for all $x\in D$.
\end{definition}
For $\Ge>0$, let 
\begin{align*}
D^{\Gg}_{\Ge}&=\{x\in D\:|\: \max_{g\in G}\|\Gg_{g}(x)-x\|\leq \Ge\},\qquad\Ucal(D^{\Gg}_{\Ge})=\Ucal(D)\cap D^{\Gg}_{\Ge}.
\end{align*}

\begin{proposition}\label{prop:ssa-htpy}
Let $G$ be a compact group and let $D$ be a unital, separable \cstar-algebra. Let $\Gg:G\act D$ be a strongly self-absorbing action such that $D^{\Gg}$ is $K^{1}$-injective. Then for any two unital equivariant $*$-homomorphisms $\phi_{1},\phi_{2}:D\to D$, there exists a continuous path of unitaries $w:[1,\infty)\to\Ucal(D^{\Gg})$ such that 
\[w_{1}=\mathbf{1}_{D},\quad \phi_{2}(x)=\lim_{t\to\infty}\ad{w_{t}}\circ\phi_{1}(x)\text{ for all }x\in D.\]
\end{proposition}
\begin{proof}
By \cite[Theorem 3.15]{SZ1}, it suffices to show that $\Gg$ is unitarily regular (\cite[Definition 2.17]{SZ1}), that is, for every $\Ge>0$, there exists $\Gd>0$ such that $uvu^{*}v^{*}$ is homotopic to $\oned$ in $\Ucal(D^{\Gg}_{\Ge})$ for every $u,v\in \Ucal(D^{\Gg}_{\Gd})$. 

Fix any $\Ge>0$. We may assume $\Gd<\frac{1}{4}$. For $u\in \Ucal(D^{\Gg}_{\Gd})$, let $x=\int_{G}\Gg_{g}(u)dg$. Since 
\[\|x-u\|\leq \int_{G}\|\Gg_{g}(u)-u\|dg<\Gd <1,\]
$x$ is invertible. Hence $u_{1}=x|x|^{-1}$ is a $G$-invariant unitary. 
Then we have the following inequalities.
\begin{align*}
  \|x\|&\leq \int_{G}\|\Gg_{g}(u)\|dg\leq 1,\\
\|x^{*}x-\oned\|&\leq \|x^{*}-u^{*}\|\|x\|+\|u^{*}\|\|x-u\|\leq 2\Gd,\\
\||x|^{-1}-\oned\|&=\sup\{|t^{-\frac{1}{2}}-1|\:|\: t\in \Gs(x^{*}x)\}\leq \sup\{|t^{-\frac{1}{2}}-1|\:|\: t\in[1-2\Gd,1]\}.
\end{align*}
Hence we have
\begin{align*}
\|u_{1}-u\|&=\|x|x|^{-1}-u\|\\
&\leq\|x-u\|+\|x\|\||x|^{-1}-\oned\|\\
&\leq \Gd+(1-2\Gd)^{-\frac{1}{2}}-1\\
&\leq \Gd+ 1+2\Gd-1=3\Gd.
\end{align*}
The last inequality follows from the fact that $(1-t)^{-\frac{1}{2}}\leq 1+t$ for $0\leq t\leq \frac{1}{2}$. Since $\|uu_{1}^{*}-\oned\|\leq 1$, $-1\notin\Gs(uu_{1}^{*})$. Define $\phi:\Gs(uu_{1}^{*})\to\Rbb$ by $\phi(e^{it})=t$ and define $f_{s}:\Gs(uu_{1}^{*})\to \Tbb$ by $f_{s}(z)=e^{is\phi(z)}$ for $0\leq s\leq 1$. Continuity of the continuous functional calculus implies that there exists $\Gd_{1}>0$ such that if $a,b\in D$ satisfies $\Gs(a),\ \Gs(b)\subset \Gs(uu_{1}^{*})$ and $\|a-b\|\leq \Gd_{1}$, then $\max_{0\leq s\leq 1}\|f_{s}(a)-f_{s}(b)\|<\frac{\Ge}{4}$. Since $\max_{g\in G}\|\Gg_{g}(uu_{1}^{*})-uu_{1}^{*}\|\leq 3\Gd$, by replacing $\Gd$ with $\min\{\Gd,\frac{\Gd_{1}}{3}\}$ if necessary, we have 
\[\max_{g\in G,\ s\in[0,1]}\|\Gg_{g}(f_{s}(uu_{1}^{*}))-f_{s}(uu_{1}^{*})\|<\frac{\Ge}{4}.\]
Hence the path $s\mapsto f_{s}(uu_{1}^{*})$ between $uu_{1}^{*}$ and $\oned$ takes values in $D^{\Gg}_{\frac{\Ge}{4}}$, and we can conclude that $u$ and $u_{1}$ are homotopic in $\Ucal(D^{\Gg}_{\frac{\Ge}{4}})$. By a similar construction, for $v\in \Ucal(D^{\Gg}_{\Gd})$, we have $v_{1}\in \Ucal(D^{\Gg})$ that is homotopic to $v$ in $\Ucal(D^{\Gg}_{\frac{\Ge}{4}})$. Then $uvu^{*}v^{*}$ is homotopic to $u_{1}v_{1}u_{1}^{*}v_{1}^{*}$ in $\Ucal(D^{\Gg}_{\Ge})$. Since $D^{\Gg}$ is $K_{1}$-injective and $[u_{1}v_{1}u_{1}^{*}v_{1}^{*}]_{1}=0$ in $K_{1}(D^{\Gg})$, we can conclude that $u_{1}v_{1}u_{1}^{*}v_{1}^{*}$ is homotopic to $\oned$ in $\Ucal(D^{\Gg})$. This finishes the proof.
\end{proof}
\begin{remark}
Let $\Gg$ be an \isa action of compact group $G$ on a Kirchberg algebra $D$. Then, it follow from \cite[Remark 4.6]{Mu} that $D^{\Gg}$ is a Kirchberg algebra. Hence $D^{\Gg}$ is $K_{1}$-injective.
\end{remark}
\begin{example}\label{ex:ssaisa-triv}
Let $G$ be a compact group. Consider the action $\Gg:G\act \oinf$ as in \autoref{def:qf}. Then its infinite tensor product $(\oinf,\Gg^{\infty})$ is isometrically shift-absorbing and strongly self-absorbing by \cite[Corollary 6.8]{SZ2}. In fact, it follows from \cite[Corollary 7.3]{Iz} that it is not necessary to take the infinite tensor product.
\end{example}
\begin{example}\label{ex:model-action}
Let $G$ be a compact group. Let $x=\sum_{i=1}^{N}d_{i}[\pi_{i}]\in R(G)$ ($d_{i}\in\Zbb\setminus\{0\}$) and set $d=\sum_{i=1}^{N}d_{i}\dim \pi_{i}$. Then we can take projections $p_{i}\in M_{d}\otimes \oinf$ such that $[p_{i}]_{0}=d_{i}\dim \pi_{i}$ in $K_{0}(M_{d}\otimes \oinf)$ and $\sum_{i=1}^{N}p_{i}=\one_{M_{d}\otimes\oinf}$.

For each $i$, the corner $p_{i}(M_{d}\otimes\oinf)p_{i}$ is isomorphic to $M_{|d_{i}|\dim \pi_{i}}\otimes \oinf$. Then there is a unitary representation $U^{(i)}:G\to\Ucal(p_{i}(M_{d}\otimes\oinf)p_{i})$ such that $(p_{i}(M_{d}\otimes\oinf)p_{i},\, \ad U^{(i)})$ is conjugate to $(M_{|d_{i}|\dim \pi_{i}}\otimes\oinf,\, \ad (\oplus_{|d_{i}|}\pi_{i})\otimes\id{\oinf})$. For each $g\in G$, define $U_{g}=\sum_{i=1}^{N}U^{(i)}_{g}$. Then $\ad U$ defines an inner action on $M_{d}\otimes\oinf$ of $G$ and its infinite tensor product action on $(M_{d}\otimes \oinf)^{\otimes\infty}\cong M_{d^{\infty}}\otimes\oinf$ is strongly self-absorbing by \cite[Proposition 6.3]{SZ1}.

We compute its equivariant $K$-groups. Since the action on $M_{d}\otimes\oinf$ constructed as above is inner, it is cocycle conjugate to the trivial action. Hence $(M_{d}\otimes\oinf,\:\ad U)$ is $KK^{G}$-equivalent to $(\Cbb,\:\id{\Cbb})$. This implies
\[K_{0}^{G}(M_{d}\otimes \oinf)\cong R(G),\qquad K_{1}^{G}(M_{d}\otimes\oinf)=0.\]
To compute $K_{0}^{G}(M_{d^{\infty}}\otimes\oinf)$, we determine the connecting map $K_{0}^{G}((M_{d}\otimes\oinf)^{\otimes n})\to K_{0}^{G}((M_{d}\otimes\oinf)^{\otimes n+1})$. Note that $K_{0}^{G}((M_{d}\otimes\oinf)^{\otimes n})\cong R(G)$ for every $n$. Let $[p]_{0}\in K_{0}^{G}((M_{d}\otimes\oinf)^{\otimes n})$. Then we have
\begin{align*}
[p\otimes \one_{M_{d}\otimes \oinf}]_{0}=\sum_{i=1}^{N}[p\otimes p_{i}]_{0}=\sum_{i=1}^{N}d_{i}[\pi_{i}][p]_{0}=x[p]_{0}.
\end{align*}
Thus the connecting map is multiplication by $x$, and $K_{0}^{G}(M_{d^{\infty}}\otimes\oinf)$ is isomorphic to $R(G)[x^{-1}]$. Since the action as in \autoref{ex:ssaisa-triv} is $KK^{G}$-equivalent to $(\Cbb,\id{})$, by tensoring this action, we have a strongly self-absorbing and isometrically shift-absorbing action on $M_{d^{\infty}}\otimes \oinf$ such that 
\[K_{0}^{G}(M_{d^{\infty}}\otimes\oinf)\cong R(G)[x^{-1}],\quad K_{1}^{G}(M_{d^{\infty}}\otimes\oinf)=0.\]
Let $S$ be any multiplicative subset of $R(G)$. For each $x\in S$, we can construct a strongly self-absorbing and isometrically shift-absorbing action $\Gg^{(x)}:G\act D^{(x)}$ such that $K_{0}^{G}(D^{(x)})=R(G)[x^{-1}]$ and $K_{1}^{G}(D^{(x)})=0$. Set $(D,\Gg)=(\otimes_{x\in S}D^{(x)},\otimes_{x\in S}\Gg^{(x)})$. Then we have a strongly self-absorbing and isometrically shift-absrobing action on a unital Kirchberg algebra $D$ such that
\[K_{0}^{G}(D)=R(G)[S^{-1}],\quad K_{1}^{G}(D)=0.\]
Note that $K^{G}_{0}(D)$ is a copmplete invariant of $(D,\Gg)$ constructed as above up to conjugacy \cite[Proposition 3.5]{IO}. 
\end{example}

\begin{example}
Let $G$ be a finite group and let $D$ be a strongly self-absorbing \cstar-algebra. Let $\Gg$ be the Bernoulli action of $G$ on $D^{\otimes|G|}$. Then $\Gg$ is a strongly self-absorbing action.
\begin{proof}
We denote the $k$-th copy of $D$ in $D^{\otimes|G|}$ by $D_{k}$. Let 
\[\phi_{0}: D_{1}^{\otimes 2}\otimes D_{2}^{\otimes 2}\otimes\cdots\otimes D_{|G|}^{\otimes 2}\to D^{\otimes |G|}\otimes D^{\otimes|G|}\]
 be the isomorphism defined by permuting the tensor factors. Then, $\phi_{0}$ is an equivariant isomorphism between $(D_{1}^{\otimes 2}\otimes D_{2}^{\otimes 2}\otimes\cdots\otimes D_{|G|}^{\otimes 2},\Gb)$ and $(D^{\otimes |G|}\otimes D^{\otimes|G|},\Gg\otimes\Gg)$, where $\Gb$ is the Bernoulli action. Since $D$ is a strongly self-absorbing \cstar-algebra, there exists an isomorphism $\phi:D\to D\otimes D$ and sequence of unitaries $(u_{n})\subset \Ucal(D\otimes D)$ such that
\[\lim_{n\to\infty}\ad{u_{n}}\circ \phi(d)=d\otimes \oned\qquad\forall\:d\in D.\]
Define an equivariant isomorphism $\psi:D^{\otimes|G|}\to D^{\otimes|G|}\otimes D^{\otimes |G|}$ and sequence of unitaries $(v_{n})\subset \Ucal(D^{\otimes|G|}\otimes D^{\otimes|G|})$ by 
\[\psi=\phi_{0}\circ \phi^{\otimes |G|}\quad\text{and}\quad v_{n}=\phi_{0}(u_{n}^{\otimes|G|}).\]
Then for every $g\in G$, we have
\begin{align*}
(\Gg\otimes\Gg)_{g}(v_{n})=(\Gg\otimes\Gg)_{g}(\phi_{0}(u_{n}^{\otimes |G|}))=\phi_{0}(\Gb_{g}(u_{n}^{\otimes |G|}))=\phi_{0}(u_{n}^{\otimes |G|})=v_{n}.
\end{align*}
And for every $z=(x_{1}\otimes \cdots x_{|G|})\in D^{\otimes |G|}$, we have
\begin{align*}
\|\ad{v_{n}}&\circ\psi(z)-z\otimes\oned\|\\
&=\|\phi_{0}(u_{n}^{\otimes |G|}(\phi(x_{1})\otimes\cdots\otimes \phi(x_{|G|}))(u_{n}^{\otimes|G|})^{*})-(x_{1}\otimes\cdots x_{|G|})\otimes \one_{D^{\otimes |G|}}\|\\
&=\|(u_{n}\phi(x_{1})u_{n}^{*})\otimes\cdots\otimes (u_{n}\phi(x_{|G|})u_{n}^{*})-((x_{1}\otimes\oned)\otimes\cdots \otimes (x_|G|\otimes\oned))\|\\
&\to 0\quad(n\to\infty).
\end{align*}
Hence $\Gg$ is a strongly self-absorbing action.
\end{proof}
It follows from \cite[Theorem 2.1]{Iz3} and \cite[Lemma 2.4]{CSESKJN} that the Bernoulli actions on $\oinf$ and the Jiang-Su algebra $\Zcal$ are $KK^{G}$-equivalent to the trivial action on $\Cbb$.
\end{example}

Let $G$ be a compact group and let $A$ be a \cstar-algebra with $G$-action $\Ga$. Define $\times :K_{0}^{G}(A)\times K_{0}^{G}(A)\to K_{0}^{G}(A\otimes A)$ as follows. Let $V,W$ be finite-dimensional representation spaces of $G$ and let $p\in A\otimes \BB(V)$ and $q\in A\otimes \BB(W)$ be any $G$-invariant projections. By abuse of notation, we use $p\otimes q$ to denote the image of $p\otimes q$ under the isomorphism $(A\otimes\BB(V))\otimes(A\otimes \BB(W))\cong (A\otimes A)\otimes \BB(V\otimes W)$. Then $[p]_{0}\times[q]_{0}$ is given by $[p\otimes q]_{0}$.
\begin{proposition}
Let $G$ be a compact group and let $D$ be a unital, separable \cstar-algebra with a \ssa action $\Gg$ such that $D^{\Gg}$ is $K_{1}$-injective. Then $K_{0}^{G}(D)$ is a commutative ring with multiplication given by the composition 
\[K_{0}^{G}(D)\times K_{0}^{G}(D)\xlongrightarrow{\times} K_{0}^{G}(D\otimes D)\xlongrightarrow{\psi_{*}}K_{0}^{G}(D)\]
where $\psi:D\otimes D\to D$ is an equivariant $*$-isomorphism.  (Multiplication is independent of the choice of $\psi$.)
\end{proposition}
\begin{proof}
By using \autoref{prop:ssa-htpy}, it follows that any two equivariant $*$-isomorphisms $\phi_{1},\phi_{2}:D\otimes D\to D$ are homotopic. Thus the multiplication is independent of the choice of an equivariant $*$-isomorphism. 

Associativity follows from the fact that $\psi\circ (\psi\otimes \id{D})$ and $\psi\circ (\id{D}\otimes\psi)$ are homotopic. Since the tensor flip automorphism is homotopic to the identity, it is commutative. Since $\psi^{-1}$ and $\id{D}\otimes \oned$ are homotopic, we have
\[[p]_{0}\cdot[\oned]_{0}=[\psi\circ(\id{D}\otimes\oned)(p)]_{0}=[p]_{0}.\]
Commutativity and this implies $[\oned]_{0}$ is a unit. The distributive law follows from the fact that $[p\otimes\mathrm{diag}(q,r)]_{0}=[\mathrm{diag}(p\otimes q,p\otimes r)]$ for $p\in D\otimes \BB(V_{1}),q\in D\otimes \BB(V_{2})$ and $r\in D\otimes (V_{3})$.
\end{proof}
\begin{lem}
Let $G$ be a compact group and let $D$ be a unital, separable \cstar-algebra with a \ssa action $\Gg$ such that $D^{\Gg}$ is $K_{1}$-injective. Then the ring $KK^{G}(D,D)$ is commutative.
\end{lem}
\begin{proof}
Let $\Gv:\Cbb\to D$ be a unital inclusion and let $\phi:D\otimes D\to D$ be an equivariant $*$-isomorphism. We will denote by $\hotimes $ the exterior tensor product in $KK^{G}$. Since $\phi^{-1}\sim_{h}\Gv\otimes\id{D}\sim_{h}\id{D}\otimes \Gv$, we have
\[KK^{G}(\phi^{-1})=KK^{G}(\nu\otimes \id{D})=KK^{G}(\id{D}\otimes \nu).\]
For any $x\in KK^{G}(D,D)$, we have
\begin{align*}
KK^{G}(\nu\otimes\id{D})&\otimes \{x\hotimes KK^{G}(\id{D})\}\otimes KK^{G}(\phi)\\
&=KK^{G}(\id{D}\otimes \nu)\otimes \{x\hotimes KK^{G}(\id{D})\}\otimes KK^{G}(\phi)\\
&=\{x\hotimes KK^{G}(\id{\Cbb})\}\otimes KK^{G}(\id{D}\otimes \nu)\otimes KK^{G}(\phi)\\
&=x.
\end{align*}
Let $x,y\in KK^{G}(D,D)$ be arbitrary elements. Then we have
\begin{align*}
x\otimes y&=x\otimes KK^{G}(\nu\otimes \id{D})\otimes \{y\hotimes KK^{G}(\id{D})\}\otimes KK^{G}(\phi)\\
&=x\otimes \{(KK^{G}(\nu)\otimes y)\hotimes KK^{G}(\id{D})\}\otimes KK^{G}(\phi)\\
&=\{KK^{G}(\id{\Cbb})\hotimes x\}\otimes \{(KK^{G}(\nu)\otimes y)\hotimes KK^{G}(\id{D})\}\otimes KK^{G}(\phi)\\
&=\{(KK^{G}(\nu)\otimes y)\hotimes KK^{G}(\id{D})\}\otimes \{KK^{G}(\id{D})\hotimes x\}\otimes KK^{G}(\phi)\\
&=KK^{G}(\nu\otimes \id{D})\otimes \{y\hotimes KK^{G}(\id{D})\}\otimes KK^{G}(\phi)\\
&\qquad\otimes KK^{G}(\nu\otimes \id{D})\otimes \{KK^{G}(\id{D})\hotimes x\}\otimes KK^{G}(\phi)\\
&=y\otimes x.
\end{align*}
\end{proof}
\begin{proposition}[cf. {\cite[Theorem 3.4]{DW}}]
Let $G$ be a compact group and let $D$ be a unital, separable \cstar-algebra with a \ssa action $\Gg$ such that $D^{\Gg}$ is $K_{1}$-injective. Let $A$ be a separable $G$-\cstar-algebra. Let $\phi:D\otimes D\to D$ be an equivariant $*$-isomorphism. Then 
\begin{align*}
KK^{G}(\Gv)\otimes -:KK^{G}(D,D\otimes A)\ni x\mapsto KK^{G}(\nu)\otimes x\in KK^{G}(\Cbb,D\otimes A)
\end{align*}
and
\begin{align*}
\{KK^{G}&(\id{D})\hotimes -\}\otimes KK^{G}(\phi\otimes \id{A}):\\
&KK^{G}(\Cbb,D\otimes A)\ni y\mapsto\{KK^{G}(\id{D})\hotimes y\}\otimes KK^{G}(\phi\otimes \id{A})\in KK^{G}(D,D\otimes A)
\end{align*}
are mutually inverse group isomorphisms.
\end{proposition}
\begin{proof}
Let $x\in KK^{G}(D,D\otimes A)$ be arbitrary element. Then we have
\begin{align*}
\{KK^{G}(\id{D})&\hotimes (KK^{G}(\nu)\otimes x)\}\otimes KK^{G}(\phi\otimes \id{A})\\
&=KK^{G}(\id{D}\otimes \nu)\otimes\{KK^{G}(\id{D})\hotimes x\}\otimes KK^{G}(\phi\otimes \id{A})\\
&=KK^{G}(\nu\otimes \id{D})\otimes \{KK^{G}(\id{D})\hotimes x\}\otimes KK^{G}(\phi\otimes \id{A})\\
&=x\otimes KK^{G}(\nu\otimes \id{D\otimes A})\otimes KK^{G}(\phi\otimes \id{A})\\
&=x.
\end{align*}
For any $y\in KK^{G}(\Cbb,D\otimes A)$, we have
\begin{align*}
KK^{G}(\nu)&\otimes \{KK^{G}(\id{D})\hotimes y\}\otimes KK^{G}(\phi\otimes \id{A})\\
&=KK^{G}(\nu\otimes\id{\Cbb})\otimes \{KK^{G}(\id{D})\hotimes y\}\otimes KK^{G}(\phi\otimes \id{A})\\
&=y\otimes KK^{G}(\nu\otimes \id{D\otimes A})\otimes KK^{G}(\phi\otimes \id{A})\\
&=y.
\end{align*}
Thus these maps are bijective. It is clear that they preserve additions.
\end{proof}

Let $G$ be a compact group and let $D$ be a unital separable \cstar-algebra with a strongly self-absorbing action $\Gg$ such that $D^{\Gg}$ is $K_{1}$-injective. Let $X$ be a compact metrizable space. Define multiplications on $KK^{G}(\Cbb,D\otimes C(X))$ and $KK^{G}(D,D\otimes C(X))$ as follows. Let $\GD:C(X)\otimes C(X)\to C(X)$ be the diagonal map and let $\Gs:D\otimes C(X)\otimes D\otimes C(X)\to D\otimes D\otimes C(X)\otimes C(X)$ be the map that swaps the second and third tensor factors. Let $\phi:D\otimes D\to D$ be an equivariant $*$-isomorphism and let $\Gv:\Cbb\to D$ be the unital inclusion. For $x,y\in KK^{G}(\Cbb,D\otimes C(X))$ and $\Gx,\Gn\in KK^{G}(D,D\otimes C(X))$, multiplications are given by
\begin{align*}
x\cdot y&=\{KK^{G}(\id{\Cbb})\hotimes x\}\otimes \{y\hotimes KK^{G}(\id{D\otimes C(X)})\}\otimes KK^{G}(\Gs)\otimes \{KK^{G}(\phi\otimes \GD)\},\\
\Gx\cdot\Gn&=\Gx\otimes\{\Gn\hotimes KK^{G}(\id{C(X)})\}\otimes \{KK^{G}(\id{D})\hotimes KK^{G}(\GD)\}.
\end{align*}
\begin{proposition}\label{prop:multi-str}
The map 
\[\{KK^{G}(\id{D})\hotimes -\}\otimes KK^{G}(\phi\otimes \id{C(X)}):KK^{G}(\Cbb,D\otimes C(X))\to KK^{G}(D,D\otimes C(X))\]
preserves the multiplications.
\end{proposition}
\begin{proof}
  Notice that we have 
  \[KK^{G}(\phi\otimes\id{C(X)})\otimes KK^{G}(\nu\otimes\id{D\otimes C(X)})=KK^{G}(\id{D\otimes D\otimes C(X)}).\]
Let $x,y\in KK^{G}(\Cbb,D\otimes C(X))$ be arbitrary elements. Then we have
\begin{align*}
&\{KK^{G}(\id{D})\hotimes(x\cdot y)\}\otimes KK^{G}(\phi\otimes \id{C(X)})\\
&=\{KK^{G}(\id{D})\hotimes (\{KK^{G}(\id{\mathbb{C}})\hotimes x\}\otimes \{y\hotimes KK^{G}(\id{D\otimes C(X)})\}\otimes KK^{G}(\Gs)\otimes KK^{G}(\phi\otimes \GD))\}\\
&\qquad \qquad \otimes KK^{G}(\phi\otimes\id{C(X)})\\
&=\{KK^{G}(\id{D})\hotimes x\}\otimes \{KK^{G}(\id{D})\hotimes(\{y\hotimes KK^{G}(\id{D\otimes C(X)})\}\otimes KK^{G}(\Gs)\otimes KK^{G}(\phi\otimes \GD))\}\\
&\qquad\qquad\otimes KK^{G}(\phi\otimes \id{C(X)})\\
&=\{KK^{G}(\id{D})\hotimes x\}\otimes KK^{G}(\phi\otimes \id{C(X)})\\
&\qquad\qquad \otimes KK^{G}(\nu\otimes\id{D\otimes C(X)})\\
&\qquad\qquad\otimes \{KK^{G}(\id{D})\hotimes(\{y\hotimes KK^{G}(\id{D\otimes C(X)})\}\otimes KK^{G}(\Gs)\otimes KK^{G}(\phi\otimes \GD))\}\\
&\qquad\qquad\otimes KK^{G}(\phi\otimes \id{C(X)})
\end{align*}
Here, we compute the middle two lines of the last equation.
\begin{align*}
&KK^{G}(\nu\otimes\id{D\otimes C(X)})\\
&\qquad\qquad\otimes \{KK^{G}(\id{D})\hotimes(\{y\hotimes KK^{G}(\id{D\otimes C(X)})\}\otimes KK^{G}(\Gs)\otimes KK^{G}(\phi\otimes \GD))\}\\
&=\{KK^{G}(\id{\Cbb})\hotimes y \hotimes KK^{G}(\id{D\otimes C(X)})\}\otimes \{KK^{G}(\nu)\hotimes KK^{G}(\id{D\otimes C(X)\otimes D\otimes C(X)})\}\\
&\qquad\qquad\otimes \{KK^{G}(\id{D})\hotimes KK^{G}(\Gs)\}\otimes \{KK^{G}(\id{D})\hotimes KK^{G}(\phi\otimes \GD)\}\\
&=\{KK^{G}(\id{\Cbb})\hotimes y \hotimes KK^{G}(\id{D\otimes C(X)})\}\otimes \{KK^{G}(\id{\Cbb})\hotimes KK^{G}(\Gs)\}\\
&\qquad\qquad \otimes \{KK^{G}(\nu)\hotimes KK^{G}(\id{D\otimes D\otimes C(X)\otimes C(X)})\}\otimes \{KK^{G}(\id{D})\hotimes KK^{G}(\phi\otimes\GD)\}\\
&=\{KK^{G}(\id{\Cbb})\hotimes y \hotimes KK^{G}(\id{D\otimes C(X)})\}\otimes \{KK^{G}(\id{\Cbb})\hotimes KK^{G}(\Gs)\}\\
&\qquad\qquad \{(KK^{G}(\nu\otimes \id{D\otimes D})\otimes KK^{G}(\id{D}\otimes \phi))\hotimes (KK^{G}(\id{C(X)\otimes C(X)})\otimes KK^{G}(\GD))\}\\
&=\{KK^{G}(\id{\Cbb})\hotimes y \hotimes KK^{G}(\id{D\otimes C(X)})\}\otimes \{KK^{G}(\id{\Cbb})\hotimes KK^{G}(\Gs)\}\\
&\qquad\qquad \otimes\{KK^{G}(\id{D\otimes D})\hotimes KK^{G}(\GD)\}.
\end{align*}
The last equality follows from $(\id{D}\otimes \phi)\circ (\nu\otimes \id{D\otimes D})\sim_{h}\id{D\otimes D}$. Let $\Gs_{D}:D\otimes D\to D\otimes D$ be the flip automorphism. Notice that we have $\Gs_{D}\sim_{h}\id{D\otimes D}$.
Then we obtain 
\begin{align*}
&\{KK^{G}(\id{D})\hotimes(x\cdot y)\}\otimes KK^{G}(\phi\otimes \id{C(X)})\\
&=\{KK^{G}(\id{D})\hotimes x\}\otimes KK^{G}(\phi\otimes \id{C(X)})\\
&\qquad\qquad\otimes \{KK^{G}(\id{\Cbb})\hotimes y \hotimes KK^{G}(\id{D\otimes C(X)})\}\otimes \{KK^{G}(\id{\Cbb})\hotimes KK^{G}(\Gs)\}\\
&\qquad\qquad \otimes\{KK^{G}(\id{D\otimes D})\hotimes KK^{G}(\GD)\}\otimes KK^{G}(\phi\otimes\id{C(X)})\\
&=\{KK^{G}(\id{D})\hotimes x\}\otimes KK^{G}(\phi\otimes \id{C(X)})\\
&\qquad\qquad \{KK^{G}(\id{D})\hotimes y\hotimes KK^{G}(\id{C(X)})\}\otimes KK^{G}(\Gs_{D}\otimes\id{C(X)\otimes C(X)})\\
&\qquad\qquad\otimes KK^{G}(\phi\otimes \id{C(X)\otimes C(X)})\otimes KK^{G}(\id{D}\otimes \GD)\\
&=\{KK^{G}(\id{D})\hotimes x\}\otimes KK^{G}(\phi\otimes \id{C(X)})\\
&\qquad\qquad\otimes \{(\{KK^{G}(\id{D})\hotimes y\}\otimes KK^{G}(\phi\otimes \id{C(X)}))\hotimes KK^{G}(\id{C(X)})\}\\
&\qquad\qquad\otimes KK^{G}(\id{D}\otimes \GD)
\end{align*}
Therefore this map preserves multiplications.
\end{proof}

\begin{corollary}\label{cor:ring-isom}
The map 
\[\{KK^{G}(\id{D})\hotimes -\}\otimes KK^{G}(\phi):KK^{G}(\Cbb,D)\to KK^{G}(D,D)\]
is a ring isomorphism.
\end{corollary}

\subsection{The topology of $\Aut{G}{\dst}$}
In what follows, let $D$ be a unital separable \cstar-algebra, and let $\Gg$ be a strongly self-absorbing action on $D$ such that $D^{\Gg}$ is $K_{1}$-injective. 
\subsubsection{Contractibility of the automorphism group}
Let $A$ be a \cstar-algebra and let $\Ga:G\act A$ be an action of a compact group $G$. Let $e\in A^{\Ga}$ be a projection. Let 
\[\mathrm{End}_{G,e}(A)=\{\phi\in\mathrm{End}_{G}(A)\:|\: \phi(e)=e\},\quad\mathrm{Aut}_{G,e}(A)=\{\phi\in \mathrm{Aut}_{G}(A)\:|\: \phi(e)=e\}.\]
We define equivariant $*$-homomorphisms $l,r:A\to A\otimes A$ by $l(a)=a\otimes e$ and $r(a)=e\otimes a$.
\begin{lem}[cf. {\cite[Lemma 2.2]{DP}}]\label{lem:contra}
Suppose that there exists a continuous map $\Psi:[0,1]\to \mathrm{Hom}_{G}(A,A\otimes A)$ such that $\Psi_{0}=l$, $\Psi_{1}=r$, $\Psi_{t}(e)=e\otimes e$ and $\Psi_{t}$ is a $*$-isomorphism for all $t\in(0,1)$. Then $\mathrm{Aut}_{G,e}(A)$ is a contractible space.
\end{lem}
\begin{proof}
The proof is identical to that of \cite[Lemma 2.2]{DP}.
\end{proof}
\begin{theorem}[cf. {\cite[Theorem 2.3]{DP}}]\label{thm:aut-contractible}
Let $G$ be a compact group and let $D$ be a unital separable \cstar-algebra. Let $\Gg:G\act D$ be a strongly self-absorbing action such that $D^{\Gg}$ is $K_{1}$-injective. Then $\mathrm{Aut}_{G}(D)$ is contractible.
\end{theorem}
\begin{proof}
Let $l,r:D\to D\otimes D$ be the equivariant $*$-homomorphisms defined by $l(d)=d\otimes\mathbf{1}_{D}$ and $r(d)=\mathbf{1}_{D}\otimes d$. Let $\psi:(D,\Gg)\to(D\otimes D,\Gg\otimes\Gg)$ be an equivariant $*$-isomorphsim. By \autoref{prop:ssa-htpy}, there exists a norm-continuous path $u:(0,1]\to\Ucal((D\otimes D)^{\Gg\otimes\Gg})$ with $u_{1}=\one_{D\otimes D}$ such that
\[l(d)=\lim_{t\to 0}\ad{u_{t}}\circ\psi(d)\quad\text{for all }d\in D.\]
Define $\psi^{(l)}:(0,1]\to \mathrm{Isom}_{G}(D, D\otimes D)$ by $\psi^{(l)}_{t}=\ad{u_{t}}\circ\psi$. Likewise, there exists a norm-continuous path $v:[0,1)\to\Ucal((D\otimes D)^{\Gg\otimes \Gg})$ with $v_{0}=\one_{D\otimes D}$ such that
\[r(d)=\lim_{t\to 1}\ad{v_{t}}\circ\psi(d)\quad\text{for all }d\in D.\]
Define $\psi^{(r)}:[0,1)\to \mathrm{Isom}_{G}(D,D\otimes D)$ by $\psi^{(r)}_{t}=\ad{v_{t}}\circ\psi$. Then We define $\Psi:[0,1]\to\mathrm{Hom}_{G}(D,D\otimes D)$ by 
\[\Psi_{t}=\begin{cases}
l & t=0,\\
\psi^{(l)}_{1-2t} & 0<t\leq \frac{1}{2},\\
\psi^{(r)}_{2t-1} & \frac{1}{2}<t<1,\\
r & t=1.
\end{cases}\]
This $\Psi$ satisfies the conditions as in \autoref{lem:contra} with $e=\one_{D}$. Therefore $\mathrm{Aut}_{G}(D)$ is contractible.
\end{proof}
Let $H=L^{2}(G)\otimes H_{0}$, where $H_{0}$ is a separable infinite-dimensional Hilbert space. In what follows, whenever we simply write $\mathbb{K}$, we assume that it is equipped with the trivial action. On the other hand, $\Kbb(H)$ means that $H$ is as defined above and the algebra is equipped with the $G$-action $\ad\rho\otimes \id{}$. Let $\Tbb=\{z\in \Cbb\:|\: |z|=1\}$
\begin{lem}[cf. {\cite[Theorem 2.4]{EP}}]\label{lem:rkone}
Let $G$ be a compact group. Let $e=e_{1}\otimes e_{0}\in \Kbb(H)$ be a rank-one projection, where $e_{1}$ is the projection onto the subspace spanned by constant functions, and $e_{0}$ is a rank-one projection in $\Kbb(H_{0})$. Then we have a homeomorphism 
\[\mathrm{Aut}_{G,e}(\Kbb(H))\cong \{W\in \Ucal(H)\;|\; U|_{eH}=\one|_{eH},\; (\rho_{g}\otimes\one{})(W)=W(\rho_{g}\otimes \one{})\text{ for all }g\in G\}\eqqcolon L\]
where $\mathrm{Aut}_{G,e}(\Kbb(H))$ is equipped with the point norm topology and $K$ is equipped with the strong topology. Moreover, $L$ is contractible.
\end{lem}
\begin{proof}
First, note that the strict topology on $\Ucal(H)\subset \Mcal(\Kbb(H))$ is the same as the strong topology. Hence 
\[\ad:L\ni W\mapsto \ad W\in\mathrm{Aut}_{G,e}(\Kbb(H))\]
is a continuous map. We will construct the inverse map of this. Let $\phi\in \mathrm{Aut}_{G,e}(\Kbb(H))$ be arbitrary element. Then there exists a unitary $W\in \Ucal(H)$ such that $\phi=\ad W$. Since $\phi$ fixes $e$, we have $eW=We=eWe=\Gl_{W}e$ for some $\Gl_{W}\in \Tbb$. We now observe that $\overline{\Gl_{W}}W$ satisfies $\overline{\Gl_{W}}W|_{eH}=\one|_{eH}$ and commutes with $\rho_{g}\otimes \one$ for all $g\in G$. Indeed, for any $\Gx\in H$, it is easy to see that $\overline{\Gl_{W}}We\Gx=\overline{\Gl_{W}}\Gl_{W}e\Gx=e\Gx$. Furthermore, for any $T\in \Kbb(H)$ and $g\in G$, we have$$[(\rho_{g}\otimes\one)^{*}\Gl_{W}W^{*}(\rho_{g}\otimes\one)\overline{\Gl_{W}}W,T]=0.$$This implies that there exists $z_{g}\in \Tbb$ such that $(\rho_{g}\otimes\one)\overline{\Gl_{W}}W=z_{g}\overline{\Gl_{W}}W(\rho_{g}\otimes \one)$. Applying both sides of this equation to $e\Gx$ for an arbitrary $\Gx\in H$, we obtain$$e\Gx=(\rho_{g}\otimes\one)\overline{\Gl_{W}}We\Gx=z_{g}\overline{\Gl_{W}}W(\rho_{g}\otimes \one)e\Gx=z_{g}e\Gx.$$Thus, we conclude that $z_{g}=1$, which yields $(\rho_{g}\otimes\one)\overline{\Gl_{W}}W=\overline{\Gl_{W}}W(\rho_{g}\otimes \one)$. Since $\overline{\Gl_{W}}W$ is invariant when $W$ is multiplied by an element in $\{z\in \Cbb\:|\: |z|=1\}$. Hence the map
\[\mathrm{Aut}_{G,e}(\Kbb(H))\ni \ad W\mapsto \overline{\Gl_{W}}W\in L\]
is well-defined. For any $W\in \{V\in\Ucal(H)\:|\: VeV=e\}$ and $\Gx\in H$, we have $We\Gx=\Gl_{W}e\Gx$. Hence the above map is continuous. It is clear that this map is the inverse of $\ad:L\ni W\mapsto \ad W\in \mathrm{Aut}_{G,e}(\Kbb(H))$.

To see contrabtibility of $L$, we use the following decomposition : 
\begin{align*}
  L&\cong \Ucal(\Cbb e_{1}\otimes (1-e_{0})H_{0})\times \prod_{\pi\in \hat{G},\pi\neq \id{}}\Ucal(\BB(V_{\pi}^{\otimes\mathrm{dim}\pi}\otimes H_{0})\cap (\pi^{\otimes\mathrm{dim}\pi}(G)\otimes \one)')\\
  &\cong \Ucal(\Cbb e_{1}\otimes (1-e_{0})H_{0})\times \prod_{\pi\in \hat{G},\pi\neq \id{}}\Ucal(V_{\pi}\otimes H_{0}).
\end{align*}
Since each component is contrabtible, we can conclude that $L$ is contractible.
\end{proof}

\begin{lem}[cf. {\cite[Lemma 2.5]{EP}}]
Let $H$ and $e=e_{1}\otimes e_{0}$ be as above. Let $G$ acts on $\Kbb(H)\otimes \Kbb(H)$ diagonally. Then there exists a point norm continuous path
\[\gamma:[0,1]\to\mathrm{Hom}_{G}(\Kbb(H),\Kbb(H)\otimes\Kbb(H))\]
such that
\begin{itemize}[nosep]
\item $\Gg_{0}(T)=T\otimes e$ and $\Gg_{1}(T)=e\otimes T$,
\item $\Gg_{t}$ is a euqivariant isomorphism for all $t\in (0,1)$,
\item $\Gg_{t}(e)=e\otimes e$ for all $t\in [0,1]$.
\end{itemize}
\end{lem}
\begin{proof}
Choose an isomorphism $\psi:\Kbb(H_{0})\to \Kbb(L^{2}(G)\otimes H_{0}\otimes H_{0})$ such that  $\psi(e_{0})=e_{1}\otimes e_{0}\otimes e_{0}$. By the proof of \cite[Theorem 2.5]{DP}, there is a continuous path
\[\Gg^{(1)}:[0,1]\to\mathrm{Hom}(\Kbb(H_{0}),\Kbb(L^{2}(G)\otimes H_{0}\otimes H_{0}))\]
such that
\begin{itemize}[nosep]
\item $\Gg^{(1)}_{0}(T)=e_{1}\otimes T\otimes e_{0}$, $\Gg^{(1)}_{1}=\psi$,
\item $\Gg^{(1)}_{t}$ is an isomorphism for all $t\in (0,1]$,
\item $\Gg^{(1)}_{t}(e_{0})=e_{1}\otimes e_{0}\otimes e_{0}$ for all $t\in [0,1]$.
\end{itemize}
By tensoring $(\Kbb(L^{2}(G),\ad\rho))$ and letting $\overline{\Gg^{(1)}}_{t}=\id{\Kbb(L^{2}(G))}\otimes \Gg^{(1)}_{t}$, we have a continuous path
\[\overline{\Gg^{(1)}}:[0,1]\to \mathrm{Hom_{G}}(\Kbb(H),\ad\rho\otimes\id{}),(\Kbb(L^{2}(G)\otimes L^{2}(G)\otimes H_{0}\otimes H_{0}),\ad\rho\otimes\id{\Kbb(L^{2}(G)\otimes H_{0}\otimes H_{0})})\]
such that
\begin{itemize}[nosep]
\item $\overline{\Gg^{(1)}}_{0}(T\otimes S)=T\otimes e_{1}\otimes S\otimes e_{0}$,
\item $\overline{\Gg^{(1)}}_{t}$ is an equivariant isomorphism for all $t\in (0,1]$,
\item $\overline{\Gg^{(1)}}_{t}(e_{1}\otimes e_{0})=e_{1}\otimes e_{1}\otimes e_{0}\otimes e_{0}$ for all $t\in [0,1]$. 
\end{itemize}
Define a unitary operator $W:L^{2}(G\times G)\to L^{2}(G\times G)$ by $(Wf)(x,y)=f(x,yx^{-1})$. Then $W(\rho_{g}\otimes \one)=(\rho_{g}\otimes\rho_{g})W$ holds. Hence $\ad W:(\Kbb(L^{2}(G\times G)),\ad\rho\otimes \id{})\to (\Kbb(L^{2}(G\otimes G)),\ad\rho\otimes\ad\rho)$ is an equivariant isomorphism.  We claim that $W(T\otimes e_{1})W^{*}=T\otimes e_{1}$ holds for all $T\in \Kbb(L^{2}(G))$. First, we observe $W(\one\otimes e_{1})=(\one\otimes e_{1})W=(\one\otimes e_{1})$. Let $\Gx\in L^{2}(G\times G)$ be any elements. Then
\[((\one\otimes e_{1})Wf)(x,y)=\int_{G}(Wf)(x,y)dy=\int_{G}f(x,yx^{-1})dy=\int_{G}f(x,y)dy=((\one\otimes e_{1})f)(x,y)\]
and
\[(W(\one\otimes e_{1})f)(x,y)=((\one\otimes e_{1})f)(x,yx^{-1})=\int_{G}f(x,y)dy=((\one\otimes e_{1})f)(x,y).\]
Therefore we have
\[W(T\otimes e_{1})W^{*}=W(T\otimes\one)(\one\otimes e_{1})W^{*}=W(T\otimes \one)(\one\otimes e_{1})=W(\one\otimes e_{1})(T\otimes \one)=T\otimes e_{1}.\]
By composing $\overline{\Gg^{(1)}}$, $\ad W$ and the tensor flip map, we obtain a continuous path 
\[\widetilde{\Gg^{(1)}}:[0,1]\to\mathrm{Hom}_{G}(\Kbb(H),\ad\rho\otimes\id{}),(\Kbb(H)\otimes \Kbb(H),\ad\rho\otimes\id{}\otimes\ad\rho\otimes\id{})\]
such that
\begin{itemize}[nosep]
\item $\widetilde{\Gg^{(1)}}_{0}(T)=T\otimes e$,
\item $\widetilde{\Gg^{(1)}}_{t}$ is an equivariant isomorphism for all $t\in (0,1]$,
\item $\widetilde{\Gg^{(1)}}_{t}(e)=e\otimes e$ for all $t\in [0,1]$.
\end{itemize}
Define a unitary operator $V:L^{2}(G\times G)\to L^{2}(G\times G)$ by $(Vf)(x,y)=f(xy^{-1},y)$. By a similar construction using $V$, we obtain a continuous path
\[\widetilde{\Gg^{(2)}}:[0,1]\to \mathrm{Hom}_{G}(\Kbb(H),\ad\rho\otimes\id{}),(\Kbb(H)\otimes \Kbb(H),\ad\rho\otimes\id{}\otimes\ad\rho\otimes\id{})\]
such that 
\begin{itemize}[nosep]
\item $\widetilde{\Gg^{(2)}}_{1}(T)=e\otimes T$,
\item $\widetilde{\Gg^{(2)}}_{t}$ is an equivariant isomorphism for all $t\in [0,1)$,
\item $\widetilde{\Gg^{(2)}}_{t}(e)=e\otimes e$ for all $t\in [0,1]$.
\end{itemize}
Since $\mathrm{Aut}_{G,e}(\Kbb(H))$ is contractible, there exists a continuous path $\Phi:[0,1]\to \mathrm{Aut}_{G,e}(\Kbb(H))$ from $\left(\widetilde{\Gg^{(2)}}_{0}\right)^{-1}\circ \widetilde{\Gg^{(1)}}_{1}$ to $\id{\Kbb(H)}$. Define $\widetilde{\Gg^{(3)}}_{t}=\widetilde{\Gg^{(2)}}_{0}\circ \Phi_{t}$. Then, $\widetilde{\Gg^{(3)}}_{t}$ is an equivariant isomorphism and $\widetilde{\Gg^{(3)}}_{t}(e)=e\otimes e$ for all $t\in [0,1]$. By concatenating the paths $\widetilde{\Gg^{(1)}}$, $\widetilde{\Gg^{(3)}}$, and $\widetilde{\Gg^{(2)}}$, we obtain the desired path.
\end{proof}

\begin{theorem}[cf. {\cite[Theorem 2.5]{DP}}]\label{thm:contractible}
Let $G$ be a compact group and let $D$ be a unital separable \cstar-algebra. Let $\Gg:G\act D$ be a \ssa action such that $D^{\Gg}$ is $K_{1}$-injective. Let $e\in\Kbb$ be a rank-one projection.  Then $\mathrm{Aut}_{G,\one_{D}\otimes e}((D\otimes \Kbb,\Gg\otimes\id{\Kbb}))$ and $\mathrm{Aut}_{G,\one_{D}\otimes e}((D\otimes\Kbb(H),\Gg\otimes\ad\rho\otimes\id{}))$ are contractible.
\end{theorem}
\begin{proof}
We begin by showing that $\mathrm{Aut}_{G,\one_{D}\otimes e}((D\otimes \Kbb,\Gg\otimes\id{\Kbb}))$ is contractible.
Let $l_{\Kbb},r_{\Kbb}:\Kbb\to\Kbb\otimes \Kbb$ be given by $l_{\Kbb}(x)=x\otimes e$ and $r_{\Kbb}(x)=e\otimes x$. By \cite[Lemma 2.4]{DP} and the same argument as in the proof of the preceding theorem, there exists a continuous map $\Psi^{(\Kbb)}:[0,1]\to\mathrm{Hom}_{G}(\Kbb,\Kbb\otimes\Kbb)$ such that $\Psi^{(\Kbb)}_{0}=l_{\Kbb}$, $\Psi^{(\Kbb)}_{1}=l_{\Kbb}$, $\Psi^{(\Kbb)}_{t}(e)=e\otimes e$ and $\Psi_{t}^{(\Kbb)}$ is a $*$-isomorphism for all $t\in(0,1)$. 

Let $l_{D},r_{D}:D\otimes D\otimes D$ be given by $l_{D}(d)=d\otimes\one_{D}$ and $r_{D}(d)=\one_{D}\otimes d$. By the proof of \autoref{thm:aut-contractible}, there exists a continuous map $\Psi^{(D)}:[0,1]\to \mathrm{Hom}_{G}(D,D\otimes D)$ such that $\Psi^{(D)}_{0}=l_{D}$, $\Psi^{(D)}_{1}=l_{D}$ and $\Psi^{(D)}$ is a $*$-isomorphism for all $t\in(0,1)$. 

Let $\Gs: D\otimes D\otimes\Kbb\otimes\Kbb \to (D\otimes\Kbb)\otimes( D\otimes \Kbb)$ be the map that swaps the second and third tensor factors.
Let $r,l:D\otimes \Kbb\to (D\otimes \Kbb)^{\otimes 2}$ be given by $l(x)=x\otimes (\one_{D}\otimes e)$ and $r(x)=(\one_{D}\otimes e)\otimes x$ for all $x\in D\otimes \Kbb$. Then $l=\Gs\circ(l_{D}\otimes l_{\Kbb})$ and $r=\Gs\circ (r_{D}\otimes r_{\Kbb})$ hold. Define 
\begin{align}\label{def:psi}
  \tag{$*$}
\Psi:[0,1]\to\mathrm{Hom}_{G}(D\otimes \Kbb,(D\otimes \Kbb)^{\otimes 2})
\end{align}
by $\Gs\circ(\Psi^{(D)}\otimes \Psi^{(\Kbb)})$. Then we have $\Psi_{0}=l$, $\Psi_{1}=r$, $\Psi_{t}(\one_{D}\otimes e)=(\one_{D}\otimes e)\otimes (\one_{D}\otimes e)$ and $\Psi_{t}$ is a $*$-isomorphism for all $t\in (0,1)$. By \autoref{lem:contra}, $\mathrm{Aut}_{G,\one_{D}\otimes e}(D\otimes \Kbb)$ is contractible.

The proof that $\mathrm{Aut}_{G,\one_{D}\otimes e}((D\otimes\Kbb(H),\Gg\otimes\ad\rho\otimes\id{}))$ is contractible proceeds exactly as in \cite[Theorem 2.6]{EP}.
\end{proof}
\begin{lem}
Let $G$ be a compact group and $A$ be a \cstar-algebra. Let $\Ga:G\act A$ be an action. Then $\Ga\otimes \ad\rho$ on $A\otimes \Kbb(L^{2}(G))$ is saturated.
\end{lem}
\begin{proof}
It suffices to show that $\mathrm{span}\{\int_{G}(a\Ga_{g}(b)\otimes T\rho_{g}S\rho_{g})u_{g}dg\:|\:a,b\in A,\; T,S\in\Kbb(L^{2}(G)) \}$ is dense in $(A\otimes\Kbb(L^{2}(G)))\rtimes_{\Ga\otimes\ad\rho}G$, where $u_{g}$ is the implimenting unitary representation. Since 
\[(A\otimes \Kbb(L^{2}(G)))\rtimes_{\Ga\otimes \ad\rho}G\ni \int_{G}x_{g}u_{g}dg\mapsto \int_{G}x_{g}(\one\otimes\rho_{g}^{*})u_{g}dg\in (A\otimes \Kbb(L^{2}(G)))\rtimes_{\Ga\otimes\id{}}G\]
gives an isomorphism, it suffices to show that $\mathrm{span}\{\int_{G}(a\Ga_{g}(b)\otimes T\rho_{g}S)u_{g}dg\:|\: a,b\in A,\; T,S\in\Kbb(L^{2}(G))\}$ is dense in $A\rtimes_{\Ga\otimes\id{}}G$. Let $\Gx_{1},\Gx_{2},\Gn_{1},\Gn_{2}\in L^{2}(G)$ be any elements. Then we have 
\begin{align*}
\int_{G}(a\Ga_{g}(b)^{*}\otimes \theta_{\Gx_{1},\Gx_{2}}\rho_{g}\theta_{\Gn_{1},\Gn_{2}})u_{g}dg&=\int_{G}(a\Ga_{g}(b)\otimes \langle  \Gx_{2},\rho_{g}\Gn_{1} \rangle \theta_{\Gx_{1},\Gn_{1}})u_{g}dg\\
&=\left(\int_{G}\langle  \Gx_{2},\rho_{g}\Gn_{1} \rangle a\Ga_{g}(b)u_{g}dg\right)\otimes \theta_{\Gx_{1},\Gn_{2}}.
\end{align*}
By the Peter-Weyl theorem, $\mathrm{span}\{(g\mapsto \langle \Gx,\rho_{g}\Gn \rangle)\in C(G)\:|\: \Gx,\Gn\in L^{2}(G)\}$ is dense in $C(G)$ with respect to the norm topology. Hence, $\mathrm{span}\{(g\mapsto\langle \Gx,\rho_{g}\Gn \rangle a\Ga_{g}(b))\:|\:a,b\in A, \Gx,\Gn\in L^{2}(G) \}$ is dense in $C(G,A)$ with respect to the norm topology, and it follows that$$ \mathrm{span}\left\{ \left. \left(\int_{G}\langle \Gx_{2},\rho_{g}\Gn_{1} \rangle a\Ga_{g}(b)u_{g}\,dg\right)\otimes \theta_{\Gx_{1},\Gn_{2}} \;\right|\; a,b\in A,\;\Gx_{i},\Gn_{i}\in L^{2}(G) \right\} $$is dense in $(A\rtimes_{\Ga} G)\otimes \Kbb(L^{2}(G))$. Therefore, $\Ga\otimes\ad\rho$ is saturated.
\end{proof}

\begin{lem}[cf. {\cite[Lemma 2.9]{EP}}]\label{lem:onee-inj}
Let $G$ be a compact group and let $D$ be a unital separable \cstar-algebra. Let $\Gg:G\act D$ be a strongly self-absorbing action such that $D^{\Gg}$ is $K_{1}$-injective. \\
(i)\; If $(D\otimes \Kbb)^{\Gg\otimes \id{}}$ has the cancellation property, then
\[\Theta:\pi_{0}(\mathrm{Aut}_{G}(\dst))\ni[\phi]\mapsto [\phi(\oned\otimes e)]_{0}\in K_{0}(D^{\Gg})_{+}\]
is an injective group homomorphism, where $e\in \Kbb$ is a rank-one projection.\\
(ii)\; If $(D\otimes \Kbb(H))\rtimes_{\Ga\otimes \ad\rho\otimes \id{}}G$ has the cancellation property, then
\[\Theta:\pi_{0}(\mathrm{Aut}_{G}(D\otimes \Kbb(H)))\ni[\phi]\mapsto [\phi(\oned\otimes e)]_{0}\in K_{0}^{G}(D)^{\times}_{+}\]
is an injective group homomorphism, where $H$ and $e$ are as in \autoref{lem:rkone}.
\end{lem}
\begin{proof}
We first deal with (i). Let $\Psi:[0,1]\to\Hom_{G}(\dst,(\dst)^{\otimes 2})$ be the path from \eqref{def:psi} and let $\psi=\Psi_{\frac{1}{2}}$. Define a multiplication on $\pi_{0}(\mathrm{Aut}_{G}(\dst))$ by
\[[\phi_{1}]\star[\phi_{2}]=[\psi^{-1}\circ(\phi_{1}\otimes \phi_{2})\circ\psi].\]
For each $\phi\in\mathrm{Aut}_{G}(\dst)$, define a continuous path $H_{\phi}:[0,1]\to\mathrm{Aut}_{G}(\dst)$ by
\[H_{\phi}(t)=\begin{cases}
(\Psi_{\frac{1}{2}t})^{-1}\circ (\phi\otimes \id{\dst})\circ\Psi_{\frac{1}{2}t} &t\in (0,1],\\
\phi& t=0.
\end{cases}\]
This path shows that $[\phi]\star[\id{\dst}]=[\phi]$. By replacing $\phi\otimes\id{\dst}$ with $\id{\dst}\otimes \phi$ in this path, we obtain $[\id{\dst}]\star[\phi]=[\phi]$. Furthermore, we have
\[([\phi_{1}]\star[\phi_{2}])\circ([\psi_{1}]\star[\psi_{2}])=([\phi_{1}]\circ[\psi_{1}])\star([\phi_{2}]\circ[\psi_{2}]).\]
Hence the Eckmann-Hilton argument implies that $\star$ and $\circ$ agree and are both associative and commutative. For any $\phi_{1},\phi_{2}\in\mathrm{Aut}_{G}(\dst)$, we have
\begin{align*}
\Theta([\phi_{1}]\star[\phi_{2}])&=[\psi^{-1}\circ(\phi_{1}\otimes\phi_{2})\circ\psi(\onee)]_{0}\\
&=[\psi^{-1}(\phi_{1}(\onee)\otimes\phi_{2}(\onee))]_{0}\\
&=[\phi_{1}(\onee)]_{0}\cdot[\phi_{2}(\onee)]_{0}=\Theta([\phi_{1}])\Theta([\phi_{2}]).
\end{align*}
Hence $\Theta$ is a group morphism. In particular, we have $\Theta([\phi])\cdot\Theta([\phi^{-1}])=\Theta([\id{\dst}])=[\onee]_{0}$. Hence the image of $\Theta$ is contained in $K_{0}(\dstf)^{\times}_{+}$.

Let $[\phi]\in \ker\Theta$. Then $[\phi(\onee)]_{0}=[\onee]_{0}$. Since $\dstf$ has the cancellation property, there exists a unitary $u\in\Ucal(\Mcal(\dstf))$ such that $u\phi(\onee)u^{*}=\onee$. There is a strictly continuous path from $\onee$ to $u$ in $\Ucal(\Mcal(\dstf))$ because $\Ucal(\Mcal(\dstf))\cong\Ucal(\Mcal(D^{\Gg}\otimes\Kbb))$ is contractible. Hence we have $[\phi]=[\ad u\circ\phi]$. Since $\ad u \circ\phi(\onee)=\onee$, it follows from \autoref{thm:contractible} that there is a continuous path between $\one{\dst}$ and $\ad u\circ \phi$. Therefore $\Theta$ is injective.

As for (ii), the argument is identical to that of \cite[Lemma 2.9]{EP}.
\end{proof}

\begin{proposition}
Let $G$ be a compact group and let $(D,\Gg)=(\mathrm{End}(V_{\pi})^{\otimes\infty},\ad\pi^{\otimes \infty})$. Then $\Theta:\pi_{0}(\mathrm{Aut}_{G}(D\otimes \Kbb(H)))\to K_{0}^{G}(D)_{+}^{\times}$ is a group isomorphism.
\end{proposition}
\begin{proof}
It suffices to show that $\Theta$ is surjective. Note that $K_{0}^{G}(D)\cong R(G)[1/[\pi]]$. Let $x\in (K_{0}^{G}(D))_{+}^{\times}$ be an arbitrary element. Then $x=r/[\pi]^{k}$ for some $r\in R(G)_{+}$ and $k\in \Zbb_{\geq 0}$. Since $x$ is invertible, there exists $s\in R(G)_{+}$ such that $rs=[\pi]^{l}$ for some $l\in \Zbb_{\geq 0}$. Let $V_{r}$ and $V_{s}$ be the representations corresponding to $r$ and $s$, respectively. Then we have 
$V_{r}\otimes V_{s}\cong V_{\pi}^{\otimes l}$. Let
\[\psi:D\to (\mathrm{End}(V_{r})\otimes \mathrm{End}(V_{s}))^{\otimes\infty}\otimes \mathrm{End}(V_{r})\cong  D\otimes\mathrm{End}(V_{r}) \]
be the equivariant $*$-isomorphism that shifts the tensor factors. Since $H$ contains all representations with infinite multiplicity, we obtain an equivariant $*$-isomorphism
\[\phi:\mathrm{End}(V_{r})\otimes \Kbb(H) \to \Kbb(H).\]
We define $\Ga^{(r)}$ to be $(\phi\otimes\id{D})\circ(\id{\Kbb(H)}\otimes \psi)\in \mathrm{Aut}_{G}(D\otimes \Kbb(H))$. By construction, $[\Ga^{(r)}(\onee)]_{0}=r\in K_{0}^{G}(D)$. Similarly, we can construct $\Ga^{(\pi)}\in\mathrm{Aut}_{G}(D\otimes \Kbb(H))$ such that $[\Ga^{(\pi)}(\onee)]_{0}=[\pi]\in K_{0}^{G}(D)$. Then we have $\rho([(\Ga^{(\pi)})^{-k}\circ \Ga^{(r)}])=x$, which completes the proof.
\end{proof}

\begin{lem}\label{lem:pi-zero-auto}
Let $G$ be a compact group and let $D$ be a unital Kirchberg algebra. Let $\Gg$ be an isometrically shift-absorbing and strongly self-absorbing action. Then
\[\Theta:\pi_{0}(\Aut{G}{\dst})\ni[\phi]\mapsto [\phi(\onee)]_{0}\in K_{0}^{G}(D)^{\times}\]
is a group isomorphism.
\end{lem}
\begin{proof}
First we show the map
\[\pi_{0}(\Aut{G}{\dst})\to KK^{G}(\dst,\dst)^{\times},\quad[\phi]\mapsto KK^{G}(\phi)\]
is a group isomorphism. Surjectivity follows from \cite[Corollary 6.4]{GS2}. By combining \cite[Theorem 5.7]{GS2} with the proof of injectivity in \autoref{thm:stable-end}, we obtain injectivity. Clearly this map preserves the group structure.

Next we claim that $KK^{G}(D,D)$ is isomorphic to $KK^{G}(\dst,\dst)$ as rings via $x\mapsto x\hotimes \id{\Kbb}$. Let $e\in \Kbb$ be a rank-one projection and let $\Gk:D\to D\otimes \Kbb$ be the inclusion into the corner associated with $\onee$. Then $KK^{G}(\Gk)\in KK^{G}(D,\dst)$ is an isomorphism. Moreover, we have $x\otimes KK^{G}(\id{D}\otimes\Gk)=KK^{G}(\id{D}\otimes \Gk)\otimes (x\hotimes \id{K})$ for all $x\in KK^{G}(D,D)$. Thus the claim follows.

Let $\nu:\Cbb\to D$ be a unital inclusion. Then the following diagram commutes.
\begin{center}
\begin{tikzpicture}[auto]
\node (1) at (-4,1) {$\pi_{0}(\Aut{G}{\dst})$};
\node (2) at (4,1) {$KK^{G}(\dst,\dst)^{\times}$};
\node (3) at (-4,-1) {$K_{0}^{G}(D)^{\times}\cong KK^{G}(\Cbb,D)^{\times}$};
\node (4) at (4,-1) {$KK^{G}(D,D)^{\times}$};
\draw[->] (1) to node[above] {$\cong$}(2);
\draw[->] (1) to node {$\Theta$} (3);
\draw[->] (4) to node {$KK^{G}(\nu)\otimes -$} (3);
\draw[->] (2) to node[right] {$\left(-\hotimes KK^{G}(\id{\Kbb})\right)^{-1}$} (4);
\end{tikzpicture}
\end{center}
Since all maps other than $\Theta$ are isomorphisms, it follows that $\Theta$ is also an isomorphism.
\end{proof}

\subsubsection{The homotopy type of $\mathrm{Aut}_{G}(D)^{0}$}
Let $\mathrm{Aut}_{G}(D\otimes \Kbb)^{0}$ be the path-component of $\id{D\otimes\Kbb}$. For a \cstar-algebra $A$, let us denote by $\mathrm{Proj}(A)$ the set of all projections in $A$, equipped with the norm topology. For a projection $p_{0}\in A$, let $\mathrm{Proj}_{p_{0}}(A)$ be the path-component of $p_{0}$.

All assertions in this section remain valid if $\mathbb{K}$ is replaced by $\mathbb{K}(H)$.
\begin{proposition}[cf.{\cite[Lemma 2.8]{DP}}]\label{prop:dp28}
Let $G$ be a compact group and let $D$ be a unital separable \cstar-algebra. Let $\Gg:G\act D$ be a strongly self-absorbing action. Let $e\in \Kbb$ be a rank-one projection. Then the map 
\[\pi:\mathrm{Aut}_{G}(D\otimes \Kbb)^{0}\to\mathrm{Proj}_{\one_{D}\otimes e}((D\otimes \Kbb)^{\Gg\otimes\id{\Kbb}}),\qquad \phi\mapsto\phi(\one_{D}\otimes e)\]
gives rise to a principal $\mathrm{Aut}_{G,\one_{D}\otimes e}(D\otimes \Kbb)$-bundle over $\mathrm{Proj}_{\one_{D}\otimes e}((D\otimes\Kbb)^{\Gg\otimes \id{\Kbb}})$. 
\end{proposition} 
\begin{proof}
First, we show surjectivity. Note that for any $u\in\Ucal_{0}(\Mcal(D\otimes \Kbb)^{\Gg\otimes \id{\Kbb}})$, we have $\ad{u}\in \mathrm{Aut}_{G}(D\otimes \Kbb)^{0}$. Let $p\in \mathrm{Proj}_{\one_{D}\otimes e}((D\otimes \Kbb)^{\Gg\otimes \id{\Kbb}})$ be any projection. Since $p$ is homotopic to $\one_{D}\otimes e$, there exists a unitary $u\in \Ucal_{0}(\Mcal(D\otimes \Kbb)^{\Gg\otimes \id{\Kbb}})$ such that $u(\one_{D}\otimes e)u^{*}=p$. Thus we have $\pi(\ad{u})=p$ and sujectivity follows. 

Let $p_{0}\in\mathrm{Proj}_{\one_{D}\otimes e}((D\otimes\Kbb)^{\Gg\otimes \id{\Kbb}})$ and $W$ be its open neighborhood given by 
\[W=\{p\in \mathrm{Proj}_{\one_{D}\otimes e}((D\otimes\Kbb)^{\Gg\otimes \id{\Kbb}})\:|\: \|p-p_{0}\|<1\}.\]
If $p\in W$, then $x_{p}=p_{0}p+(\one-p_{0})(\one-p)$ is an invertible element of $\Mcal(D\otimes \Kbb)^{\Gg\otimes \id{\Kbb}}$. Thus it follows that $u_{p}=x_{p}|x_{p}|^{-1}$ is a unitary in $\Ucal_{0}(\Mcal(D\otimes \Kbb)^{\Gg\otimes \id{\Kbb}})$ and the map $p\mapsto u_{p}$ is norm-continuous. Take a unitary $v\in\Ucal_{0}(\Mcal(D\otimes \Kbb)^{\Gg\otimes \id{\Kbb}})$ such that $p_{0}=v(\one_{D}\otimes e)v^{*}$. Then the map 
\[\Gs_{p_{0}}:W\to \mathrm{Aut}_{G}(D\otimes \Kbb)^{0},\quad p\mapsto \ad{u_{p}^{*}v}\]
is a continuous section of $\pi$ over $W$ and the map 
\[\Gk_{W}:W\times\mathrm{Aut}_{G,\one_{D}\otimes e}(D\otimes \Kbb)\to \pi^{-1}(W)\quad\Gk_{W}(x,\phi)=\Gs_{p_{0}}(x)\circ\phi \] 
is a local trivialization with inverse 
\[\tau_{W}:\pi^{-1}(W)\to W\times \mathrm{Aut}_{G,\one_{D}\otimes e}(D\otimes \Kbb),\quad \tau_{W}(\phi)=(\phi(\one_{D}\otimes e),\Gs_{p_{0}}(\phi(\one_{D}\otimes e))^{-1}\circ\phi).\]
\end{proof}

\begin{proposition}[cf.{\cite[Proposition 3.2]{EP}}]\label{prop:proj-bu}
Let $G$ be a compact group and let $D$ be a unital separable \cstar-algebra. Let $\Gg:G\act D$ be a strongly self-absorbing action. For simplicity of notation, we write $A$ instead of $\dstf$ and let 
\[\Ucal_{\one_{D}\otimes e}(\Mcal(A))=\{u\in \Ucal(\Mcal(A))\:|\: \ad{u}(\one_{D}\otimes e)=\one_{D}\otimes e\}\]
equipped with the norm topology. Then the map
\[\ad{}:\Ucal(\Mcal(A))\to\mathrm{Proj}_{\one_{D}\otimes e}(A),\quad u\mapsto \ad{u}(\one_{D}\otimes e)\]
gives rise to a principal $\Ucal_{\one_{D}\otimes e}(\Mcal(A))$-bundle. Furthermore, the inclusion 
\[\Ucal(D^{\Gg})\to \Ucal_{\one_{D}\otimes e}(\Mcal(A)),\quad v\mapsto v\otimes e+(\one-\one\otimes e)\]
is a homotopy equivalence and hence we have a homotopy equivalence $\mathrm{Proj}_{\one_{D}\otimes e}(A)\simeq B\Ucal(D^{\Gg})$.
\end{proposition}
\begin{proof}
First, note that $\Ucal(\Mcal(A))$ is contractible by the main theorem in \cite{CH}. Local triviality follows from the same argument as in the proof of the preceding proposition. 

We show that $\Ucal(D^{\Gg})\to \Ucal_{\one_{D}\otimes e}(\Mcal(A))$ is a homotopy equivalence. Note that 
\[\Ucal_{\one_{D}\otimes e}(A)\cong \Ucal((\one_{D}\otimes e)\Mcal(A)(\one_{D}\otimes e))\times \Ucal((\one_{D}\otimes e^{\perp})\Mcal(A)(\one_{D}\otimes e^{\perp})).\]
Since $\one_{D}\otimes e\in A$, we obtain $(\one_{D}\otimes e)\Mcal(A)(\one_{D}\otimes e)=D^{\Gg}$. By \cite[Corollary 1.2.37]{AM}, we have $(\one_{D}\otimes e^{\perp})\Mcal(A)(\one_{D}\otimes e^{\perp})=\Mcal((\one_{D}\otimes e)A(\one_{D}\otimes e))$. Since $e$ is a rank-one projection, we have $(\one_{D}\otimes e^{\perp})A(\one_{D}\otimes e^{\perp})=(D\otimes e^{\perp}\Kbb e^{\perp})^{\Gg\otimes \id{\Kbb}}\cong A$. This implies that $\Ucal((\one_{D}\otimes e^{\perp})\Mcal(A)(\one_{D}\otimes e^{\perp}))\cong \Ucal(\Mcal(A))$ is contractible. Therefore the inclusion $\Ucal(D^{\Gg})\to \Ucal_{\one_{D}\otimes e}(\Mcal(A))$ is a homotopy equivalence. The contractibility of $\Ucal(\Mcal(A))$ implies that $\mathrm{Proj}_{\one_{D}\otimes e}(A)$ is homtopy equivalent to $B\Ucal_{\one_{D}\otimes e}(A)$. Thus we have a homotopy equivalence between $\mathrm{Proj}_{\one_{D}\otimes e}(A)$ and $B\Ucal(D^{\Gg})$.
\end{proof}
\begin{proposition}[cf. {\cite[Proposition 3.3]{EP}}]\label{prop:aut-proj}
Let $G$ be a compact group and let $D$ be a unital separable \cstar-algebra. Let $\Gg:G\act D$ be a strongly self-absorbing action such that $D^{\Gg}$ is $K_{1}$-injective. Then, the map
\[\pi:\Aut{G}{\dst}^{0}\to\mathrm{Proj}_{\one_{D}\otimes e}(\dstf),\quad \phi\mapsto\phi(\one_{D}\otimes e)\]
as in \autoref{prop:dp28} indeuces a homotopy equivalence $\Aut{G}{\dst}^{0}\simeq\mathrm{Proj}_{\onee}(\dstf)$.
\end{proposition}
\begin{proof}
Since $\mathrm{Proj}_{\onee}(\dstf)$ is a Banach manifold, it has the homotopy type of a CW-complex. Since 
\[\Aut{G,1\otimes e}{\dst}\to\Aut{G}{\dst}^{0}\to\mathrm{Proj}_{\onee}(\dstf)\]
is a Hurewicz fibration, it follows from \cite[Theorem 2]{Sc} that $\Aut{G}{\dst}^{0}$ has the homotopy type of a CW-complex. Since $\Aut{G,\onee}{\dst}$ is contractible, Whitead's theorem implies $\Aut{G}{\dst}^{0}\simeq\mathrm{Proj}_{\onee}(\dstf)$.
\end{proof}

\begin{corollary}\label{cor:htpy-auto}
Let $G$ be a compact group and let $D$ be a unital separable \cstar-algebra. Let $\Gg:G\act D$ be a strongly self-absorbig action such that $D^{\Gg}$ is $K_{1}$-injective. Then we have
\begin{align*}
\pi_{n}(\mathrm{Aut}_{G}(D\otimes \Kbb))\cong\pi_{n}(\mathrm{Aut}_{G}(D\otimes\Kbb(H)))\cong \pi_{n-1}(\Ucal(D^{\Gg})),
\end{align*}
for $n\geq 1$. 
If $D$ is unital Kirchberg algebra and $\Gg$ is an \isa and \ssa action, then we have
\[\pi_{n}(\mathrm{Aut}_{G}(D\otimes \Kbb))\cong\pi_{n}(\mathrm{Aut}_{G}(D\otimes\Kbb(H)))\cong\begin{cases}
K_{0}^{G}(D)^{\times} & \text{if }n=0,\\
K_{n}^{G}(D) & \text{if } n\geq 1. 
\end{cases}\]
\end{corollary}
\begin{proof}
By \autoref{prop:proj-bu} and \autoref{prop:aut-proj}, we have
\[\pi_{n}(\mathrm{Aut}_{G}(D\otimes\Kbb))\cong \pi_{n}(B\Ucal(D^{\Gg}))\cong \pi_{n-1}(\Ucal(D^{\Gg})).\]

If $D$ is unital Kirchberg algebra and $\Gg$ is an isometrically shift-absorbing and strongly self-absorbing action, we have computed $\pi_{0}(\mathrm{Aut}_{G}(D\otimes \Kbb)$ in \autoref{lem:pi-zero-auto}. It follows from \cite[Theorem 3.1]{Wi}, \cite[Theorem 3]{Ji} and the fact that $D^{\Gg}$ absorbs $\oinf$ tensorially that $\pi_{n-1}(\Ucal(D^{\Gg}))\cong K_{n}^{G}(D)$ for $n\geq 1$. 

The same holds for $\pi_{n}(\mathrm{Aut}_{G}(D\otimes \Kbb(H)))$.
\end{proof}

\subsubsection{The topological group $\Aut{G}{\dst}^{0}$ is well-pointed}
All assertions in this section remain valid if $\mathbb{K}$ is replaced by $\mathbb{K}(H)$.
\begin{definition}
Let $X$ be a topological space and let $A\subset X$ be a closed subspace. The pair $(X,A)$ is called a neighborhood deformation retract (or NDR-pair for short) if there is a map $u:X\to[0,1]$ such that $u^{-1}(0)=A$ and a homotopy $H:X\times[0,1]\to X$ such that $H(x,0)=x$ for all $x\in X$, $H(a,t)=a$ for all $a\in A$ and $t\in[0,1]$, and $H(x,1)\in A$ if $u(x)<1$. A pointed topological space $X$ with basepoint $x_{0}\in X$ is said to be well-pointed if the pair $(X,x_{0})$ is an NDR-pair.
\end{definition}
For a neighborhood $V$ of $x_{0}$, we say that $\{x_{0}\}$ is a deformation retract of $V$ if there is a continuous map $h:V\times[0,1]\to V$ such that $h(x,0)=x$, $h(x_{0},t)=x_{0}$ and $h(x,1)=x_{0}$ for all $x\in V$ and $t\in[0,1]$. The following lemma is contained in \cite[Theorem 2]{St}.
\begin{lem}\label{lem:NDR-pair}
Let $(X,x_{0})$ be a pointed topological space together with a continuous map $v:X\to[0,1]$ such that $x_{0}=v^{-1}(0)$ and $\{x_{0}\}$ is a deformation retract of $V=\{x\in X\:|\:v(x)<1\}$. Then $(X,x_{0})$ is an NDR-pair.
\end{lem}
\begin{proposition}[cf. {\cite[Proposition 2.26]{DP}}]\label{prop:well-pointed}
Let $G$ be a compact group and let $D$ be a unital separable \cstar-algebra. Let $\Gg:G\act D$ be a strongly self-absorbing action. The topological space $(\Aut{G}{\dst}^{0},\id{\dst})$ is well-pointed.
\end{proposition}
\begin{proof}
Let $e\in\Kbb$ be a rank-one projection. Let 
\[W=\{p\in\mathrm{Proj}_{\onee}(\dstf)\:|\:\|p-\onee\|<\frac{1}{2}\}.\]
Let 
\[\pi:\Aut{G}{\dst}^{0}\to\mathrm{Proj}_{\onee}(\dstf),\quad \phi\mapsto\phi(\onee).\]
First we prove that $\pi^{-1}(W)$ deformation retracts to $\id{\dst}\in \Aut{G}{\dst}^{0}$. There exists a homeomorphism
\[\pi^{-1}(W)\to W\times \Aut{G,\onee}{\dst}\]
such that $\id{\dst}$ maps to $(\onee,\id{\dst})$. Hence it suffices to show that $W\times \Aut{G,\onee}{\dst}$ deformation retracts to $(\onee,\id{\dst})$. Let $\Gc$ be the characteristic function of $(\frac{1}{2},1]$. Then $h(p,t)=\Gc((1-t)p+t(\onee))$ is a deformation retraction of $W$ into $p_{0}$. We have shown that $\Aut{G,\onee}{\dst}$ deformation retracts to $\id{\dst}$. By combining these homotopies, we have a deformation retraction of $\pi^{-1}(W)$ into $\id{\dst}$. Let $d$ be a metric of $\Aut{G}{\dst}^{0}$. Then the map $v:\Aut{G}{\dst}^{0}\to [0,1]$ defined by
\[v(\phi)=\max\left\{\min\left\{d(\phi,\id{\dst}),\frac{1}{2}\right\},\min\left\{1,2\|\phi(\onee)-\onee\|\right\}\right\}\]
and $V\coloneqq \pi^{-1}(W)$ satisfy the condition of the precedeing lemma.
\end{proof}

\subsection{Classification of $(\dst,\Gg\otimes\id{\Kbb})$-bundles}
In this section, we classify the isomorphism classes of locally trivial fiber bundles with fiber $\dst$, where each fiber is equipped with the $G$-action $\Gg\otimes\id{\Kbb}$. The invariants appearing in this classification arise from a generalized cohomology theory associated with the infinite loop space structure of $B\Aut{G}{\dst}$. To this end, we begin by discussing this infinite loop space structure.

\begin{definition}
A sequence of pointed topological spaces $(E_{k})_{k=0}^{\infty}$ together with weak homotopy equivalences $\Gs_{k}:E_{k}\to\GO E_{k+1}$ is called an $\GO$-spectrum, where $\GO Z$ denotes the based loop space of the pointed space $Z$. The zeroth space of an $\GO$-spectrum is called an infinite loop space.
\end{definition}
Let $(E_{k})_{k}$ be an $\GO$-spectrum. For every pointed finite CW-complex $X$, let 
\[\tilde{h}^{k}=[X,E_{k}]_{+},\]
where the right-hand side denotes based homotopy classes of basepoint-preserving continuous maps and we define $E_{k}=\GO^{-k}E_{0}$ for $k<0$. Since $[X,E_{k}]_{+}=[X,\GO^{2}E_{k+2}]_{+}$, it has a natural abelian group structure.
Then $\tilde{h}$ defines a reduced cohomology theory on the category of pointed CW-complexes. The coefficients $\check{h}^{*}$ of this cohomology theory are defined to be $\check{h}^{k}=\tilde{h}^{k}(S^{0})$. By the loop-suspension adjunction, we have $\check{h}^{k}=\pi_{-k}(E_{0})$.

To see that $\Aut{G}{\dst}$ has an infinite loop space structure, we use the theory of commutative $\Ical$-monoids, as in \cite{EP}.
\vskip\baselineskip
Let $\Ical$ be the category with objects given by $\Bn=\{1,\cdots,n\}$ for $n\in\Zbb_{\geq 0}$ ($\mathbf{0}=\emptyset$) and morphisms given by injective maps between the sets. A covariant functor from $\Ical$ to the category $\Tcal op_{*}$ of based topological spaces is said to be an \textit{$\Ical$-space.}
Define a bifunctor $\sqcup:\Ical\times\Ical\to\Ical$ as follows. For $\Bm,\Bn\in\mathrm{obj}(\Ical)$, $\Bm\sqcup\Bn=\{1,\cdots,n+m\}$. For $f:\Bm\to\Bm'$ and $g:\Bn\to\Bn'$, $f\sqcup g:\Bm\sqcup\Bn\to\Bm'\sqcup\Bn'$ acts by identifying $\Bm$ (resp. $\Bm'$) with the first $m$ (resp. $m'$) elements of $\{1,\cdots,n+m\}$ (resp. $\{1,\cdots,n'+m'\}$) and $\Bn$ (resp. $\Bn'$) with the last $n$ (resp. $n'$) elements. Then $(\Ical,\sqcup,\mathbf{0})$ forms a strict monoidal category. For any $m,n\in\Zbb_{\geq 0}$, define $\tau_{m,n}:\Bm\sqcup\Bn\to\Bn\sqcup\Bm$ by permuting the $(n,m)$-blocks. Then this family of isomorphisms gives $\Ical$ the structure of a symmetric monoidal category.

An $\Ical$-space $X$ is called an $\Ical$-monoid if it comes equipped with a natural transformation 
\[\Gm_{m,n}:X(\Bm)\times X(\Bn)\to X(\Bm\sqcup \Bn)\]
which is associative in the sense that $\Gm_{l,m+n}\circ(\id{X{(\Bl)}}\times\Gm_{m,n})=\Gl_{l+m,n}\circ(\Gm_{l,m}\times \id{X(\Bn)})$ for all $\Bl,\Bm,\Bn\in\mathrm{obj}(\Ical)$ and unital in the sense that the diagrams
\begin{center}
\begin{tikzpicture}[auto]
\node (1) at (-5,1) {$X(\mathbf{0})\times X(\Bn)$};
\node (2) at (-1.5,1) {$X(\Bn)$};
\node (3) at (-5,-1) {$X(\Bn)$};
\node(4) at (1.5,1) {$X(\Bn)\times X(\mathbf{0})$};
\node (5) at (5,1) {$X(\Bn)$};
\node(6) at (1.5,-1) {$X(\Bn)$};
\draw[->] (3) to (1);
\draw[->] (1) to node {$\Gm_{0,n}$} (2);
\draw[->] (3) to node[below right] {$\id{X_{(\Bn)}}$} (2);
\draw[->] (6) to (4);
\draw[->] (4) to node {$\Gm_{n,0}$} (5);
\draw[->] (6) to node[below right] {$\id{X(\Bn)}$} (5);
\end{tikzpicture}
\end{center}
commute for every $\Bn\in\mathrm{obj}(\Ical)$, where the two upwards arrows are the inclusion with respect to the basepoint in $X(\mathbf{0})$.
We say that an $\Ical$-monoid $X$ is commutative if 
\begin{center}
\begin{tikzpicture}[auto]
\node (1) at (-2.5,1)  {$X(\Bm)\times X(\Bn)$};
\node (2) at (2.5,1)  {$X(\Bm\sqcup\Bn)$};
\node (3) at (-2.5,-1)  {$X(\Bn)\times X(\Bm)$};
\node (4) at (2.5,-1)  {$X(\Bn\sqcup\Bm)$};
\draw[->] (1) to node {$\Gm_{m,n}$} (2);
\draw[->](2) to node[right] {$\tau_{m,n}^{*}$} (4);
\draw[->] (1) to node {swap} (3);
\draw[->] (3) to node {$\Gm_{n,m}$} (4);
\end{tikzpicture}
\end{center}
commutes. We denote the homotopy colimit of $X$ over $\Ical$ by $X_{h\Ical}=\mathrm{hocolim}_{\Ical}(X)$. Then $X_{h\Ical}$ has the structure of a topological monoid given by the composition
\[X_{h\Ical}\times X_{h\Ical}=\mathrm{hocolim}_{\Ical\times\Ical}X(\Bm)\times X(\Bn)\xlongrightarrow{\Gm}\mathrm{hocolim}_{\Ical\times \Ical}X(\Bm\sqcup\Bn)\longrightarrow X_{h\Ical},\]
in which the last map is induced by the monoidal structure of $\Ical$. We say that $X$ is group-like if $\pi_{0}X_{h\Ical}$ is a group. In \cite[Section 5]{Sc2}, it was shown that if an $\Ical$-monoid X is commutative and group-like, then $X_{h\Ical}$ has the structure of an infinite loop space.
\vskip\baselineskip
We now construct a commutative $\mathcal{I}$-monoid to show that $\Aut{G}{\dst}$ admits an infinite loop space structure.
Define an $\Ical$-space $\mathfrak{G}^{G}_{D}$ as follows. 

For $\Bn\in\Zbb_{\geq 0}$, set $\Gcal^{G}_{D}(\Bn)=\Aut{G}{\dstn}$. Any morphism $f:\Bm\to\Bn$ in $\Ical$ can be decomposed as $f=\Gs\circ \iota$ where $\Gi:\Bm\to\Bn$ is the inclusion that identifies $\{1,\cdots,m\}$ with the first $m$ elements of $\Bn$ and $\Gs:\Bn\to\Bn$ is a permutation. Define $\hat{\Gs}:\dstn\to\dstn$ to be the permutation according to $\Gs$ and let $\Gcal^{G}_{D}(\Gs):\ggd(\Bn)\to\ggd(\Bn)$ be given by conjugation by $\hat{\Gs}$.
Let 
\[\ggd(\Gi):\ggd(\Bm)\to\ggd(\Bn),\quad \phi\mapsto\phi\otimes\id{\dst^{\otimes n-m}}.\] 
If $\Gs\circ \Gi=\Gs'\circ\Gi$, then we have $\ggd(\Gs)\circ\ggd(\Gi)=\ggd(\Gs')\circ\ggd(\Gi)$. Therefore $\ggd(f)=\ggd(\Gs)\circ\ggd(\Gi)$ is well-defined. Then $\ggd$ is a covaiant functor from $\Ical$ to $\Tcal op_{*}$. Let 
\[\Gm_{m,n}:\ggd(\Bm)\times \ggd(\Bn)\to\ggd(\Bm\sqcup\Bn),\quad\Gm_{m,n}(\phi,\psi)=\phi\otimes\psi.\]
This is associative and unital. Let $\tau_{m,n}$ be a symmetry of $\Ical$. Then we have $\ggd(\tau_{m,n})(\phi\otimes\psi)=\psi\otimes \phi$. Hence $\ggd$ is a commutative $\Ical$-monoid. We show that $\ggd$ has more structure.
\begin{definition}
Let $\Gcal rp_{*}$ be the category of topological groups (which we assume to be well-pointed by the identity element) and continuous homomorphisms. A functor $G:\Ical\to\Gcal rp_{*}$ is called an $\Ical$-group.

An $\Ical$-group $\Gcal$ is called an Eckmann-Hilton $\Ical$-group (or EH-$\Ical$-group for short) if it is an $\Ical$-monoid in $\Gcal rp_{*}$ with multiplication $\Gm_{m,n}$ such that the following diagram of natural transformations between functors $\Ical^{2}\to\Gcal rp_{*}$ commutes.
\begin{center}
\begin{tikzpicture}[auto]
\node (1) at (-4,1) {$\Gcal(\Bm)\times \Gcal(\Bm)\times \Gcal(\Bn)\times \Gcal(\Bn)$};
\node (2) at (4,1) {$\Gcal(\Bm\sqcup\Bn)\times \Gcal(\Bm\sqcup\Bn)$};
\node (3) at (-4,-1) {$\Gcal(\Bm)\times \Gcal(\Bn)$};
\node (4) at (4,-1) {$\Gcal(\Bm\sqcup \Bn)$}; 
\draw[->] (1) to node {$(\Gm_{m,n}\times \Gm_{m,n})\circ \tau$} (2);
\draw[->] (1) to node[left] {$\Gv_{m}\times \Gv_{n}$} (3);
\draw[->] (3) to node {$\Gm_{m,n}$} (4);
\draw[->] (2) to node {$\Gv_{m+n}$} (4); 
\end{tikzpicture}
\end{center}
(Here, $\Gv_{m}:\Gcal(\Bm)\times \Gcal(\Bm)\to \Gcal(\Bm)$ is the group multiplication and $\tau$ is the map that switches the two innermost factors.)

We say that sn EH-$\Ical$-group $\Gcal$ is stable if all morphisms in $\Ical$ except for the maps $\mathbf{0}\to\Bn$ are mapped to homotopy equivalences.

Let $\Gi_{m}:\mathbf{0}\to\Bm$ be the unique homomorphism. We say that an EH-$\Ical$-group $\Gcal$ has compatible inverses if there is a path from $(\Gi_{m}\sqcup\id{\Bm})_{*}(g)$ to $(\id{\Bm}\sqcup\Gi_{m})_{*}(g)$ in $G(\Bm\sqcup\Bm)$ for all $\Bm\in \mathrm{obj}(\Ical)$ and $g\in G(\Bm)$. 
\end{definition}
Let $G$ be a compact group and let $D$ be a unital separable \cstar-algebra. Let $\Gg:G\act D$ be a strongly self-absorbing action such that $D^{\Gg}$ is $K_{1}$-injective and $\dstf$ has the cancellation property. It follows from \autoref{lem:NDR-pair} and \autoref{prop:well-pointed} that $(\Aut{G}{\dst},\id{\dst})$ is well-pointed. Since $(\dst,\Gg\otimes\id{\dst})\cong (\dstn,(\Gg\otimes \id{\Kbb})^{\otimes n})$, it follows that $(\Aut{G}{\dstn},\id{\dstn})$ is also well-pointed. Thus $\ggd$ is an $\Ical$-group. This remains valid if $\mathbb{K}$ is replaced by $\mathbb{K}(H)$, and the following lemma follows similarly.
\begin{lem}[cf. {\cite[Lemma 4.2]{EP}}]\label{lem:ehigrp}
The $\Ical$-group $\ggd$ is a stable EH-$\Ical$-group with compatible inverses.
\end{lem}
\begin{proof}
First we show that $\ggd$ is an EH-$\Ical$-group. For $\phi_{1},\phi_{2} \in \ggd(\Bm)$ and $\psi_{1},\psi_{2}\in\ggd(\Bn)$, we have
\begin{align*}
\Gv_{m+n}\circ (\Gm_{m,n}\times\Gm_{m,n})\circ\tau(\phi_{1},\phi_{2},\psi_{1},\psi_{2})&=\Gv_{m+n}\circ (\Gm_{m,n}\times\Gm_{m,n})(\phi_{1},\psi_{1},\phi_{2},\psi_{2})\\
&=\Gv_{m+n}(\phi_{1}\otimes\psi_{1},\phi_{2}\otimes\psi_{2})\\
&=(\phi_{1}\otimes\psi_{1})\circ(\phi_{2}\otimes\psi_{2})\\
&=(\phi_{1}\circ\phi_{2})\otimes (\psi_{1}\otimes\psi_{2})\\
&=\Gm_{m,n}(\phi_{1}\circ\phi_{2},\psi_{1}\circ\psi_{2})\\
&=\Gm_{m,n}\circ(\Gv_{m}\times\Gv_{n})(\phi_{1},\phi_{2},\psi_{1},\psi_{2}).
\end{align*}
Thus $\ggd$ is an EH-$\Ical$-group.

Next we show stability. It suffices to show that 
\[\Aut{G}{\dst}\to\Aut{G}{(\dst)^{\otimes 2}},\quad \phi\mapsto\phi\otimes\id{\dst}\]
is a homotopy equivalence. Let $\Phi:\dst\to(\dst)^{\otimes 2}$ be an equivariant $*$-isomorphism and let
\[\Aut{G}{(\dst)^{\otimes 2}}\to\Aut{G}{(\dst)^{\otimes 2}},\quad \psi\mapsto \Phi^{-1}\circ \psi\circ\Phi.\]
We claim that this is a homotopy inverse. By \autoref{prop:ssa-htpy} and the proof of \autoref{thm:contractible}, there exists $h:[0,1]\to\mathrm{Hom}_{G}(\dst,(\dst)^{\otimes 2})$ such that $h_{0}=\Phi$, $h_{1}(x)=l(x)\coloneqq x\otimes (\onee)$ and $h_{t}$ is a $*$-isomorphism for all $0\leq t< 1$ and $x\in \dst$, where $e\in\Kbb$ is a rank $1$ projection. Define $H:\Aut{G}{\dst}\times [0,1]\to\Aut{G}{\dst}$ by
\[H(\phi,t)=\begin{cases}
h_{t}^{-1}\circ (\phi\otimes \id{\dst})\circ h_{t} & 0\leq t<1,\\
\phi & t=1.
\end{cases}\]
We verify the continuity at $t=1$ as follows.
Let $t_{n}$ be any sequence in $[0,1]$ converging to $1$ and $\phi_{i}$ be any net in $\Aut{G}{\dst}$ converging to $\phi$. For any $x\in \dst$, we have
\begin{align*}
\|h_{t_{n}}^{-1}&\circ(\phi_{i}\otimes\id{\dst})\circ h_{t_{n}}(x)-\phi(x)\|\\
&=\|((\phi_{i}\otimes\id{\dst})\circ h_{t_{n}})(x)-h_{t_{n}}\circ \phi(x)\|\\
&\leq \|(\phi_{i}\otimes\id{\dst})\circ h_{t_{n}}(x)-(\phi_{i}\otimes\id{\dst})\circ l(x)\|\\
&\qquad +\|(\phi_{i}\otimes \id{\dst})\circ l(x)-(\phi\otimes\id{\dst})\circ l(x)\|\\
&\qquad +\|(\phi\otimes\id{\dst})\circ l(x)-h_{t_{n}}\circ \phi(x)\|\\
&\leq \|h_{t_{n}}(x)-l(x)\|+\|\phi_{i}(x)-\phi(x)\|+\|(\phi_{i}\otimes \id{\dst})\circ l(x)-(\phi\otimes\id{\dst})\circ l(x)\|.
\end{align*}
This inequality shows the continuity at $t=1$.
Thus this is a homotopy between $\phi$ and $\Phi^{-1}\circ( \phi\otimes \id{\dst})\circ\Phi$ and hence 
\[\Aut{G}{\dst}\to\Aut{G}{\dst},\quad \phi\mapsto \Phi^{-1}\circ(\phi\otimes\id{\dst})\circ\Phi\]
is homotopic to the identity. By replacing $\Phi$ with $\Phi\otimes\id{\dst}$ and applying the same argument, we obtain that the map
$\displaystyle\Aut{G}{(\dst)^{\otimes 2}}\to\Aut{G}{(\dst)^{\otimes 2}}$ defined by
\begin{align*}
\psi\mapsto (\Phi\otimes\id{\dst})^{-1}\circ (\psi\otimes\id{\dst})\circ(\Phi\otimes \id{\dst})=(\Phi^{-1}\circ\psi\circ\Phi)\otimes\id{\dst}
\end{align*}
is homotopic to the identity. Thus the map
\[\Aut{G}{\dst}\to\Aut{G}{(\dst)^{\otimes 2}},\quad \phi\mapsto\phi\otimes\id{\dst}\]
is a homotopy equivalence and thus stability follows.

Finally we show $\ggd$ has compatible inverses. Let $e$ be a rank-one projection in $\Kbb$ and set $p_{k}=(\onee)^{\otimes k}\in (\dst)^{\otimes k}$. Note that 
\[\pi_{0}(\ggd(\Bm\sqcup\Bm))\cong \pi_{0}(\Aut{G}{\dstm})\rightarrowtail K_{0}^{G}(D^{\otimes m}).\]
(The first isomorphism is given by conjugation by an equivariant $*$-isomorphism $\psi:(\dst)^{\otimes 2m}\to\dstm$ with $\psi(p_{2m})=p_{m}$ and the sencond map is given by $[\phi]\mapsto [\phi(p_{m})]_{0}$. Injectivity of the second map follows by the same argument as in \autoref{lem:onee-inj}.) Let $\phi\in\ggd(\Bm)$ be any element. The K-theory classes corresponding to $(\Gi_{m}\sqcup \id{\Bm})_{*}(\phi)\text{ and }(\id{\Bm}\sqcup\Gi_{m})_{*}(\phi)$ are $[\psi(p_{m}\otimes\phi(p_{m}))]_{0}$ and $[\psi(\phi(p_{m})\otimes p_{m})]_{0}$ respectively. Since we have a homotopy between $\id{\dst}\otimes \phi$ and $\phi\otimes \id{\dst}$, they represent the same element in $K_{0}^{G}(D^{\otimes m})$. Thus $\ggd$ has compatible inverses.
\end{proof}

\begin{theorem}[cf. {\cite[Theorem 4.3]{EP}}]
Let $G$ be a compact group and let $D$ be a unital separable \cstar-algebra. Let $\Gg:G\act D$ be a strongly self-absorbing action such that $D^{\Gg}$ is $K_{1}$-injective and $\dstf$ has the cancellation property. Then $\Aut{G}{\dst}$ is an infinite loop space. Let us denote the associated cohomology theory by $E^{*}_{D,G}$. Then we have
\[E^{0}_{D,G}(X)=[X,\Aut{G}{\dst}]\text{ and }E^{1}_{D,G}=[X,B\Aut{G}{\dst}].\]
We equip isomorphism classes of $G$-equivariant locally trivial fiber bundles with fibers $(\dst,\Gg\otimes\id{\Kbb})$ over a finite CW-complex $X$ with trivial action, with a group multiplication given by the fiberwise tensor product. Then this group is isomorphic to $E^{1}_{D,G}(X)$.\\
The same assertion remains valid if $\mathbb{K}$ is replaced by $\mathbb{K}(H)$, provided that we assume the cancellation property for $(D\otimes\mathbb{K}(H))\rtimes_{\Gg\otimes\ad\rho\otimes\id{}}G$.
\end{theorem}
\begin{proof}
By \autoref{lem:ehigrp}, $\ggd$ is a stable EH-$\Ical$-group. Since $\ggd(\Bn)$ has the homotopy type of a CW-complex for each $\Bn$, it follows from \cite[Lemma 3.5]{DP2} that $\ggd(\mathbf{1})\to(\ggd)_{h\Ical}$ is a homotopy equivalence. Since $\pi_{0}(\ggd(\mathbf{1}))$ is a group, $\ggd$ is group-like. Hence it follows from \cite[Section 5]{Sc2} that $\Aut{G}{\dst}$ is an infinite loop space with respect to the operation induced by the tensor product. Thus we have a delloping $B_{\otimes}\Aut{G}{\dst}$ as a part of the $\GO$-spectrum associated to $\Aut{G}{\dst}$. It follows from \cite[Theorem 3.6]{DP2} that $B_{\otimes}(\Aut{G}{\dst})$ is homotopy equivalent to $B\Aut{G}{\dst}$ which is the classifying space of $\Aut{G}{\dst}$ as a topological group (group structure is given by the composition). This implies 
\[E^{0}_{D,G}(X)=[X,\Aut{G}{\dst}]\quad\text{and}\quad E^{1}_{D,G}(X)=[X,B\Aut{G}{\dst}].\]
\end{proof}

\begin{remark}
If $D$ is a unital Kirchberg algebra and $\Gg$ is an isometrically shift-absorbing and strongly self-absorbing action, then the coefficients are given by 
\[\check{E}^{k}_{D,G}\cong\begin{cases}
0 & k>0,\\
K_{0}^{G}(D)^{\times}& k=0,\\
K_{-k}^{G}(D)& k<0.
\end{cases}\]
\begin{proof}
The coefficients $\check{E}^{k}_{D,G}$ are given by $\pi_{-k}(\Aut{G}{\dst})$ and we have already computed these groups in \autoref{cor:htpy-auto}. 
\end{proof}

\end{remark}
\begin{remark}
If $D$ is a unital Kirchberg algebra and $\Gg$ is an isometrically shift-absorbing and strongly self-absorbing action, then we have $E_{D,G}^{0}(X)\cong K_{0}^{G}(D\otimes C(X))^{\times}$. In fact, it follows from \autoref{prop:invertibility} and \autoref{cor:stbl-htpy-set} that $E_{D,G}^{0}(X)$ is isomorphic to $KK^{G}(D,D\otimes C(X))^{\times}$. By \autoref{prop:multi-str}, we have a group isomorphims $KK^{G}(D,D\otimes C(X))^{\times}\cong K_{0}^{G}(D\otimes C(X))^{\times}$. If $\Gg:G\act D$ is an action as in \autoref{ex:model-action}, then we have $E_{D,G}^{0}(X)\cong (K_{0}^{G}(C(X))[S^{-1}])^{\times}$.
\end{remark}

Let $D$ is a unital Kirchberg algebra and let $\Gg$ is an isometrically shift-absorbing and strongly self-absorbing action. Let $X$ be a finite CW-complex.  The $E_{2}$-page of the Atiyah-Horzebruch spectral sequence for the generalized cohomology $E^{*}_{D,G}$ looks as follows.
\begin{center}
\begin{tikzpicture}
  \matrix (m) [matrix of math nodes, nodes in empty cells,nodes={minimum width=1cm, minimum height=5ex,outer sep=-5pt}, column sep=1ex,row sep=1ex]{
                &   0  &  1  &  2 &  3 & \\
          0  &  H^0(X,K_{0}^{G}(D)^{\times})  &  H^1(X,K_{0}^{G}(D)^{\times})  & H^2(X,K_{0}^{G}(D)^{\times}) & H^3(X,K_{0}^{G}(D)^{\times}) \\
         -1\quad     &  0  & 0 &  0  & 0   \\
         -2\quad     &  H^0(X,K_{0}^{G}(D))  & H^1(X,K_{0}^{G}(D)) &  H^2(X,K_{0}^{G}(D))  & H^3(X,K_{0}^{G}(D))\\
         -3\quad     &  0  & 0 &  0  & 0  \\
         -4\quad     &  H^0(X,K_{0}^{G}(D))  & H^1(X,K_{0}^{G}(D)) &  H^2(X,K_{0}^{G}(D))  & H^3(X,K_{0}^{G}(D)) \\
         };
  \draw[thick] (m-1-1.east) -- (m-6-1.east) ;
  \draw[thick] (m-1-1.south) -- (m-1-6.south) ;
\end{tikzpicture}
\end{center}
Since the images of the differentials in the Atiyah-Hirzebruch spectral sequence are torsion subgroups by \cite[Theorem 2.7]{Arl}, we obtain the following.
\begin{corollary}
Let $D$ is a unital Kirchberg algebra and let $\Gg$ is an isometrically shift-absorbing and strongly self-absorbing action. Let $X$ be a finite CW-complex such that $H^{*}(X,K_{0}^{G}(D))$ is torsion free. Then
\[E^{1}_{D,G}\cong H^{1}(X,K_{0}^{G}(D)^{\times})\oplus\bigoplus_{k=1}^{\infty}H^{2k+1}(X,K_{0}^{G}(D)).\]
\end{corollary}

\nocite{*}
\addcontentsline{toc}{section}{\refname}
\bibliographystyle{amsalpha} 
\bibliography{toukouyou}

\end{document}